%% file: main.tex
\documentclass[11pt]{article}

\usepackage[utf8]{inputenc}
\usepackage{amsmath, amsthm, amssymb, comment}
\usepackage{graphicx}
\usepackage{tikz}
\usetikzlibrary{arrows.meta,decorations.markings}
\usepackage{overpic}
\usepackage{enumitem}
\usepackage[colorlinks=true]{hyperref}
\usepackage{pinlabel}
\input{bicontact_tikz_figures.tex}

\usepackage[margin=3cm]{geometry}

\usepackage[nameinlink]{cleveref}
\usepackage{microtype}

\theoremstyle{plain}
\newtheorem{introtheorem}{Theorem}

\newtheorem{introcorollary}{Corollary}[introtheorem]

\newtheorem{theorem}{Theorem}[section]
\newtheorem{corollary}[theorem]{Corollary}
\newtheorem{proposition}[theorem]{Proposition}
\newtheorem{definition}[theorem]{Definition}
\newtheorem{lemma}[theorem]{Lemma}

\newtheorem{remark}[theorem]{Remark}

\newtheorem*{theorem*}{Theorem}
\newtheorem*{proposition*}{Proposition}
\newtheorem*{definition*}{Definition}
\newtheorem*{lemma*}{Lemma}
\newtheorem*{claim*}{Claim}
\newtheorem*{corollary*}{Corollary}
\newtheorem*{convention*}{Convention}

\theoremstyle{definition}
\newtheorem{convention}[theorem]{Convention}
\newtheorem{example}[theorem]{Example}
\newtheorem{question}[theorem]{Question}

\theoremstyle{remark}

\newtheorem*{rem*}{Remark}

\title{Bi-contact plugs and dynamical hyperbolicity}

\author{Surena Hozoori}
\newcommand{\Addresses}{{
  \bigskip
  \footnotesize

Surena Hozoori, \textsc{Department of Mathematics, Brandeis University.}\par\nopagebreak
  \textit{E-mail address}: \texttt{hozoori@brandeis.edu}
  
  }}

\date{}

\begin{document}

\maketitle


\begin{abstract}
We develop a theory of {\em bi-contact
plugs} on oriented $3$-manifolds to construct Anosov flows or, more generally, structurally stable nonsingular flows. For a hyperbolic
plug with filling Morse-Smale boundary laminations, we
prove that every strongly transverse gluing
map can be isotoped through strongly transverse maps to identify a
strongly adapted unoriented contact form. The resulting flow is
hyperbolic and is Anosov when the quotient is closed. This answers a question of B\'eguin-Bonatti-Yu \cite{bby}. 
  The result also extends to gluing nonsingular partially hyperbolic flows. Anosov completion by reflection embeds
each such hyperbolic plug in an Anosov flow on its oriented double.

Further applications include a classification of structurally
generic projectively Anosov flows without
saddle periodic orbits and Morse-Smale examples with prescribed saddle
count. Generalizing the construction of Bonatti-Bowden-Potrie \cite{bbp}, finite capping
embeds any bi-contact plug in a closed projectively Anosov flow with
the same topological entropy, preserving partial hyperbolicity for
attracting plugs. We also construct flows with controlled invariant
torus dynamics on torus bundles and on doubles of plugs with
coorientable boundary bi-foliations. Finally, every cooriented
partially hyperbolic bi-contact plug admits a Liouville structure on its thickening, with Liouville
trajectories projecting to positively reparametrized flow lines.
Repelling plugs yield four-dimensional Liouville domains, often with chaotic skeletons.
\end{abstract}

{
  \hypersetup{linkcolor=black}
  \setcounter{tocdepth}{2}
  \tableofcontents
}

\section{Introduction}
\label{s1}

Recent years have seen a surge in our understanding of the intimate interactions between {\em contact geometry} and {\em hyperbolic dynamics} in dimension three.
It was observed in the work of Mitsumatsu \cite{mit} and Eliashberg-Thurston~\cite{et} in the mid 1990s that bi-contact structures characterize {\em dominated splittings} for nonsingular flows in dimension three. This condition, also called {\em projective Anosovity} when realized on a closed 3-manifold, is the cornerstone for the stability features appearing in hyperbolic dynamics, as emphasized by Mañé \cite{mane} and Palis \cite{palis} (see \cite{sam}). By now, we know that supporting
contact and {\em Liouville} geometries in fact provide a characterization of {\em Anosov
flows}~\cite{hoz3,hoz7} in purely geometric terms and connect their dynamics with {\em Floer
theory}~\cite{clmm} and contact geometric surgery~\cite{salmoiraghi}. Hence, it is natural to study further aspects of these flows in this geometric setting.

On the other hand, a complementary approach constructs flows by gluing compact pieces with prescribed dynamics. Inspired by the classic work of Franks-Williams \cite{fw}, a celebrated and fruitful recent work of B\'eguin-Bonatti-Yu~\cite{bby} developed such theory for {\em hyperbolic plugs}, using their entrance and exit laminations to control the dynamics of the result after gluing. Their self-gluing theorem guarantees that after an appropriate isotopy of the gluing map, the produced flow is Anosov under three conditions: (1) the exit and entrance laminations are {\em filling}, (2) deformation of the flow near its maximal invariant set is allowed, and (3) the interior dynamics contains no repelling or attracting pieces. The first condition provides topological convenience by enforcing the homeomorphism type of the extensions of such laminations to foliation. That is while condition (2) and (3) are somewhat technical in order to produce {\em affine Markov partitions} helpful for gluing. They then naturally ask whether one can drop conditions (2) and (3)~\cite[Question~1.4]{bby}.

In this paper, we develop a contact geometric theory of plugs that improves the construction of \cite{bby} and has applications beyond Anosov flows. In particular,
we give an affirmative answer to the aforementioned question in the filling MS setting, by
constructing {\em strongly adapted contact forms} compatible with
gluing. The required extension and boundary deformation methods
also apply beyond Anosov dynamics, yielding projectively
Anosov and partially hyperbolic flows, as well as Liouville
domains. A related contact geometric construction of Pinsky-Salmoiraghi \cite{pinsky} in the setting of {\em quasi-transverse} boundary and controlled geometry has recently produced new examples of Anosov flows.

\begin{definition}\label{bcp}
A \emph{bi-contact plug} is a compact oriented $3$-manifold $M$
with transverse contact structures $(\xi_-,\xi_+)$,
negative and positive respectively, whose supported line field
$\xi_-\cap\xi_+$ is oriented and transverse to $\partial M$. A nowhere-vanishing
vector field $X$ positively spanning this line field is a \emph{supported vector field}. See Section~\ref{s2} for related definitions and discussions.
\end{definition}

Contact forms are understood in the unoriented sense of
Section~\ref{s2.1}, with local representatives differing by locally
constant signs. The Liouville constructions use cooriented contact
structures and ordinary global forms.

The boundary carries the {\em characteristic bi-foliation}
$S_\pm=\xi_\pm\cap T\partial M$. On the maximal invariant set
$\Lambda=\operatorname{Inv}(M,X)$, the normal flow has a dominated
splitting. We call the plug \emph{partially hyperbolic} when its
dominating normal line is uniformly expanding. A positive contact
form $\alpha_+$ is \emph{strongly adapted} to $X$ if
$$
\alpha_+(X)=0,
\qquad
\mathcal L_X\alpha_+\text{ is a negative contact form}.
$$
This condition, introduced in \cite{hoz7}, is the contact geometric equivalent of saddle hyperbolicity on compact invariant sets
and provides our geometric criterion for hyperbolicity after gluing.

\paragraph{Contact geometric self-gluing.}
For a hyperbolic plug, the maximal invariant set has one stable
and one unstable direction. Its boundary laminations are
$$
\mathcal L^s=W^s(\Lambda)\cap\partial_{\mathrm{in}}M,
\qquad
\mathcal L^u=W^u(\Lambda)\cap\partial_{\mathrm{out}}M.
$$
We use the filling Morse-Smale (MS) condition recalled in
Section~\ref{s2.3}. Gluing maps are orientation-reversing diffeomorphisms, with
respect to the induced boundary orientations, from unions of
entrance components to unions of exit components. Such a map $\psi$ is \emph{strongly transverse}
when $\psi_*\mathcal L^s$ and $\mathcal L^u$ are transverse and
their complementary regions are quadrilaterals with alternating
sides. This allows simultaneous extensions of the two laminations into a bi-foliation on the boundary.

\begin{introtheorem}[Contact geometric self-gluing]
\label{intro:MS-filling-self-gluing}
\label{intro:filling-gluing}
Let $(M,X)$ be a smooth hyperbolic plug
with filling MS boundary laminations, and let $\psi_0$
be a strongly transverse gluing map. There are a smooth strongly
adapted contact form $\alpha_+$ for $X$ and a smooth isotopy
$(\psi_r)_{r\in[0,1]}$ through strongly transverse maps such that
$\psi_1$, extended to boundary neighborhoods using $X$, preserves
$\alpha_+$.

The quotient $M_{\psi_1}=M/(p\sim\psi_1(p))$ with the vector field
induced by $X$ is a hyperbolic plug. If it is closed,
its flow is Anosov.
\end{introtheorem}

The vector field $X$ is unchanged. We allow repelling or attracting pieces in the interior dynamics, and $M$ may be disconnected. For complete gluing, this gives the conclusion requested
in~\cite[Question~1.4]{bby}.
The form $\alpha_+$ descends to a
strongly adapted contact form on the quotient. The relative version
preserves a prescribed positive characteristic foliation $S_+$, i.e. if $\psi_0$ identifies an {\em admissible} bi-foliation
$(S_-,S_+)$, every $\psi_r$ can be chosen to identify $S_+$, with
$\ker(\alpha_+|_{T\partial M})=S_+$. An analogous gluing result holds for structurally generic
partially hyperbolic plugs under the corresponding filling
MS boundary hypotheses (see Corollary~\ref{cor:partially-hyperbolic-self-gluing}).

\begin{introcorollary}[Anosov completion by reflection]
\label{intro:reflection}
Every smooth hyperbolic plug $(M,X)$ with
filling MS boundary laminations embeds in an Anosov
flow on its oriented double $D(M)=M\cup_{\partial M}(-M)$.
\end{introcorollary}

Apply self-gluing to $(M,X)\sqcup(-M,-X)$, after displacing the
corresponding boundary laminations into strongly transverse
position. The gluing maps are isotopic to the canonical
identifications, so the quotient is $D(M)$. This gives an
embedding with the original vector field on a specified closed
manifold, strengthening the embedding conclusion
of~\cite[Theorem~1.12(1)]{bby}.

\paragraph{Extension and boundary preparation.}
Section~\ref{s3} introduces {\em admissibility} for characteristic
bi-foliations and {\em strong admissibility} for boundary $X$-jets, as the necessary conditions on the boundary of a hyperbolic plug for extension to supporting (strongly adapted) contact geometry. We then prove that the space of such extensions is nonempty and contractible (see Theorems~\ref{hypb} and \ref{hyps}, with Theorems~\ref{prob} and \ref{pros} as precursors on {\em product} plugs). Away from the entrance and exit laminations, the geometry of admissibility can be encoded in a $C^0$-open projective cyclic-order
condition on characteristic bi-foliations. The condition for strong admissibility can be formulated as a $C^1$-open condition on the {\em $X$-jet} of the characteristic foliation, required for solving a curvature extension problem on each flow line. The construction is relative and
parametric.

In Section~\ref{s4}, we use Giroux’s {\em convex surface theory} \cite{gir,massot,hondanotes} to prepare characteristic bi-foliations for smooth gluing through coordinated divided deformations. The supported vector fields remain
transverse to parallel boundary surfaces, so the resulting collar
changes can be absorbed into the gluing map. Further boundary manipulation and smoothing theorems are also proved (see Theorem~\ref{thm:stable-gluing} for the culmination). For hyperbolic self-gluing, we need the more technical preparation of strongly admissible boundary conditions.
We show that regular invariant completions and uniform holonomy distortion
on complementary strips then permit stretching until the boundary
forms and their derivatives satisfy the requirements.
The same tools give a contact geometric proof of
acyclic hyperbolic gluing and a relative self-gluing result for
partially hyperbolic plugs. See Section~\ref{s5}.
\vskip0.3cm

Our theory is developed beyond its applications to Anosov flows and our techniques are useful in constructing larger classes of structurally stable flows in dimension three, namely projectively Anosov or nonsingular partially hyperbolic flows. These flows are known to be $C^1$-generically structurally stable, thanks to a structure theorem of Arroyo-Rodriguez-Hertz \cite{arh}. Hence, we use the term {\em structurally generic} to refer to the corresponding generic subclass. These flows can have isolated sink or source periodic orbits, and removing neighborhoods of such orbits leaves a hyperbolic plug in case the remainder has a saddle hyperbolic periodic orbit (and a product plug otherwise).

\paragraph{Projectively Anosov flows with finite recurrent dynamics.}
We first classify all structurally generic projectively Anosov flows with no saddle hyperbolic periodic orbits.

\begin{introtheorem}[Classification without saddle periodic orbits]
\label{intro:two-orbit}
A connected closed oriented $3$-manifold admits a structurally
generic projectively Anosov flow without
saddle periodic orbits if and only if it is diffeomorphic to
$L(n,1)$ or $L(4n,2n+1)$ for some $n\in\mathbb Z$.
After a suitable orientation-preserving identification, every such
flow is homotopic through projectively Anosov flows to a standard
model $X_n$ or $X_n^{\mathrm{tw}}$, according as it is transversely
orientable or not.
\end{introtheorem}

The models $X_n$ include examples on $L(1,1)=\mathbb{S}^3$ and $L(0,1)=\mathbb S^2\times\mathbb S^1$. On each $L(n,1)$, the construction also gives examples with a prescribed number $k\in\mathbb{N}$ of saddle hyperbolic periodic orbits. All these flows are {\em Morse-Smale} with zero topological entropy.

\begin{introtheorem}[Prescribed saddle count]
\label{intro:saddle-count}
For every $k\geq0$, some closed oriented $3$-manifold admits a
transversely oriented projectively Anosov Morse-Smale flow with
exactly $k$ saddle, one source and $k+1$ sink periodic orbits.
These flows have zero topological entropy.
\end{introtheorem}

These examples raise further classification questions for Morse-Smale projectively Anosov flows.

\begin{question}
Which manifolds admit a Morse-Smale projectively Anosov flow? Is there a bi-contact homotopy class admitting infinitely many distinct examples of such flows? How is this condition related to the contact topological features of the underlying contact structures?
\end{question}

\paragraph{Finite capping.}
The Morse-Smale zero-entropy plugs used above cap off an arbitrary bi-contact plug by adding only finitely many periodic orbits.
This generalizes the capping constructions of~\cite{bbp}. We use
suspended-gradient and twisted-saddle pieces to prepare boundary
bi-foliations for solid-torus caps, without coorientability or winding
restrictions. See
Section~\ref{subsec:finite-PA-capping}.

\begin{introtheorem}[Finite projectively Anosov capping]
\label{intro:finite-capping}
The flow supported on any bi-contact plug embeds in a projectively
Anosov flow on a closed manifold, with the same topological entropy
as the flow on the original maximal invariant set. The added nonwandering dynamics consist
of finitely many hyperbolic periodic orbits. If the original bi-contact
plug is an attracting partially hyperbolic plug, the result is a
nonsingular partially hyperbolic flow on a closed manifold.
\end{introtheorem}

\paragraph{Bi-contact doubling and invariant tori.}
Normal models $Y+Kz\partial_z$, with $|K|$ sufficiently large,
realize prescribed nonsingular vector fields $Y$ on normally
repelling or attracting invariant tori. They give a doubling
construction for bi-contact plugs with coorientable boundary
bi-foliations. Unlike Anosov completion by
reflection, this construction uses $X$ on both copies.

\begin{introtheorem}[Bi-contact doubling]
\label{intro:doubling}
Let $(M,\xi_-,\xi_+,X)$ be a compact bi-contact plug with coorientable
boundary characteristic bi-foliations. Its oriented
double carries a projectively Anosov flow extending $X$ on two
disjoint copies of $M$. Between the copies, each boundary component
gives an invariant torus with Morse-Smale dynamics, normally
repelling at an entrance and attracting at an exit. If the original bi-contact plug is an attracting  partially hyperbolic plug, the result is a nonsingular partially hyperbolic flow on a closed manifold.
\end{introtheorem}

We conjecture a negative answer to the following question about the above construction.

\begin{question}
Can the (partially hyperbolic) projectively Anosov flows of Theorem~\ref{intro:doubling} be homotopic through (partially hyperbolic) projectively Anosov flows to an Anosov flow?
\end{question}

Our toric constructions can also be used to generalize the regular
$\mathbb T^2\times I$-models of Noda~\cite{noda2000,noda}
by allowing prescribed nonsingular dynamics on successive
invariant tori, subject to explicit transversality conditions.
Although arbitrary dynamics on a
single invariant torus can already be realized by
\cite[Proposition~3.3]{noda2000}, our construction prescribes
these dynamics simultaneously and controls the gluing and
the intervening orbits, without requiring smooth global
invariant foliations.


\begin{introtheorem}[Projectively Anosov flows on torus bundles]
\label{intro:toric-stacks}
Every oriented torus bundle over $\mathbb S^1$ admits a cooriented
bi-contact structure supporting a projectively Anosov flow with
any prescribed even number $2m\geq2$ of distinguished invariant
tori, alternating between normal repulsion and attraction.
Every other orbit is wandering and approaches a repelling torus
backward in time and an attracting torus forward in time.
\end{introtheorem}

\paragraph{Liouville geometry of partially hyperbolic plugs.}
Our final construction extends the relation between hyperbolic
dynamics and Liouville geometry~\cite{mit,hoz3,hoz6} to partially
hyperbolic plugs, based on their underlying contact geometry.

\begin{introtheorem}[Liouville structures from partially hyperbolic plugs]
\label{intro:liouville}
Let $(M,\xi_-,\xi_+,X)$ be a compact cooriented partially hyperbolic
bi-contact plug. There are smooth defining contact forms
$\alpha_\pm$ for $\xi_\pm$ such that
$$
\lambda=(1-s)\alpha_-+(1+s)\alpha_+
$$
is a nowhere-vanishing Liouville form on $[-1,1]_s\times M$, whose
Liouville vector field has the form
$Z=A\partial_s+BX$ with $B>0$.
If $M$ is repelling, rounding the corners gives a Liouville domain whose skeleton projects homeomorphically onto
$\operatorname{Inv}(M,X)$.
\end{introtheorem}

Both horizontal faces are convex. Vertical faces are convex over
exits and concave over entrances. Thus general plugs give exact
symplectic cobordisms with corners. The condition $B>0$ makes
projected Liouville trajectories positively reparametrized
trajectories of $X$.

The construction uses an expanding defining form
$\mathcal L_X\alpha=r\alpha$, $r>0$, followed by smoothing that
preserves the contact planes. For a strongly adapted contact
form $\alpha_+$ on a hyperbolic plug, one can instead set
$\alpha_-=\mathcal L_X\alpha_+$ and use directly
$$
\lambda_C=C(1-s)\alpha_-+(1+s)\alpha_+
$$
for sufficiently large $C>0$.

In most examples, the Liouville dynamics of the above examples is chaotic, in contrast to the gradient-like dynamics of {\em Weinstein} Liouville domains. These examples provide candidates for Liouville domains with at most half-dimensional topological type that are not Liouville homotopic to Weinstein domains. The existence of such domains is a well-known open problem in symplectic topology.

\begin{question}
Are all Liouville domains constructed in Theorem~\ref{intro:liouville} homotopic, through Liouville domains, to some Weinstein Liouville domain?
\end{question}

\paragraph{Organization of the paper.}
Section~\ref{s2} recalls the dynamical and contact geometric
preliminaries, including ingredients from the theory of hyperbolic plugs \cite{bby} and the structure theorem for dominated splittings \cite{arh}. Section~\ref{s3} focuses on the construction of various hyperbolic and non-hyperbolic bi-contact plugs and proves useful extension results. Section~\ref{s4} develops our
boundary deformation, smoothing and gluing methods.
Section~\ref{s5} proves contact geometric self-gluing and Anosov
completion by reflection. Section~\ref{s6} gives applications beyond Anosov flows, including the classification,
finite capping, toric and bi-contact doubling constructions discussed in this introduction.
Section~\ref{sec:liouville-domains-repelling-plugs} constructs Liouville structures from partially hyperbolic plugs.

\paragraph{Acknowledgement.}
We are grateful to Federico Salmoiraghi for many fruitful discussions on the topic. This work has also benefitted from insightful feedback and conversations with Masayuki Asaoka, Rafael Potrie and Julian Chaidez.

\paragraph{AI statement.} AI was used solely for proofreading and improving the prose in this manuscript, with the exception of Lemma~\ref{lem:regularity-invariant-completions}, where the sketch of the proof was conducted by the author and some computational details were provided by LLM. The author takes responsibility for all the material in this paper.

%

\section{Invariant sets in bi-contact geometry and hyperbolic dynamics}
\label{s2}

This section recalls the contact geometric and dynamical notions used
throughout the paper.  

\begin{convention}
Unless stated otherwise, all manifolds and vector fields are smooth ($C^\infty$), and $3$-manifolds are oriented and compact. We impose no coorientability assumption on contact structures or orientability assumption on invariant line bundles.
\end{convention}

\subsection{From bi-contact structures to dominated splittings}\label{s2.1}

Let $M$ be an oriented $3$-manifold. A plane field $\xi$ is a
\emph{contact structure} if locally $\xi=\ker\alpha$ for a $1$-form satisfying
$
\alpha\wedge d\alpha\neq0.
$
It is positive or negative according to whether $\alpha\wedge d\alpha$
agrees or disagrees with the orientation of $M$. This sign is unchanged
when $\alpha$ is multiplied by any nowhere-zero function. One can consult \cite{gei} for the basic theory of contact structures in dimension three.

An \emph{unoriented contact form} is given by local contact forms
$\alpha_i$ with $\alpha_j=\varepsilon_{ij}\alpha_i$ on overlaps of the domains, where
$\varepsilon_{ij}\in\{\pm1\}$ is locally constant. Equivalently, it is a contact structure $\xi$ together with a norm on $TM/\xi$, given by $|\alpha_i|$.
Except in Section~\ref{sec:liouville-domains-repelling-plugs}, we henceforth call these simply \emph{contact forms}, unless stated otherwise, and interpret formulas in local representatives.

\begin{definition}
A \emph{bi-contact structure} on $M$ is a transverse pair
$
(\xi_-,\xi_+)
$
consisting of a negative and a positive contact structure.  A
nowhere-vanishing vector field $X$ is \emph{supported} by this structure if
$
\langle X\rangle=\xi_-\cap\xi_+.
$
\end{definition}
\begin{convention}
Given a supported vector field $X$, choose local representatives
$\alpha_\pm$ so that $\alpha_-\wedge\alpha_+$ and $\iota_X\Omega$ define the same orientation on $TM/\langle X\rangle$, where $\Omega$ is an orientation form of $M$.
Their transition signs then coincide. Sums and differences of these forms are taken in such compatible local representatives.
\end{convention}

Below are some key definitions in our theory.

\begin{definition}
Let $\Sigma\subset M$ be a smoothly embedded surface and let $\xi$ be a
contact structure on $M$.  The \emph{characteristic foliation} of $\xi$ on
$\Sigma$ is the possibly singular foliation
$
S(\xi,\Sigma):=\ker\bigl(\alpha|_{T\Sigma}\bigr)
$,
where $\alpha$ is any defining contact form for $\xi$.  If
$(\xi_-,\xi_+)$ supports $X$ and $X\pitchfork\Sigma$, then the two
characteristic foliations
$
S_\pm(\Sigma):=S(\xi_\pm,\Sigma)
$
are nonsingular and transverse.  Their ordered pair
$$
\mathcal S(\Sigma):=(S_-(\Sigma),S_+(\Sigma))
$$
is called the \emph{characteristic bi-foliation} induced on $\Sigma$.
More generally, a \emph{bi-foliation} on a surface is an ordered pair of
transverse nonsingular one-dimensional foliations.
\end{definition}

\begin{definition}
A \emph{bi-contact plug} is a compact oriented $3$-manifold $M$ with
boundary, equipped with a bi-contact structure and an orientation of its supported line field $\xi_- \cap \xi_+$, which is
transverse to $\partial M$. After choosing a positively oriented supported vector field $X$, we
write
$$
\partial M=\partial_{\mathrm{in}}M\sqcup
\partial_{\mathrm{out}}M,
$$
where $X$ points inward along $\partial_{\mathrm{in}}M$ and outward along
$\partial_{\mathrm{out}}M$.  Its maximal invariant set is
$$
\operatorname{Inv}(M,X)
:=
\bigcap_{t\in\mathbb R}\varphi_X^t(M),
$$
where the intersection is considered using the partially defined flow on
$M$. See Figure~\ref{fig:bicontact-plug}.
\end{definition}

\begin{figure}[htbp]
  \centering
  \includegraphics[width=.5\textwidth]{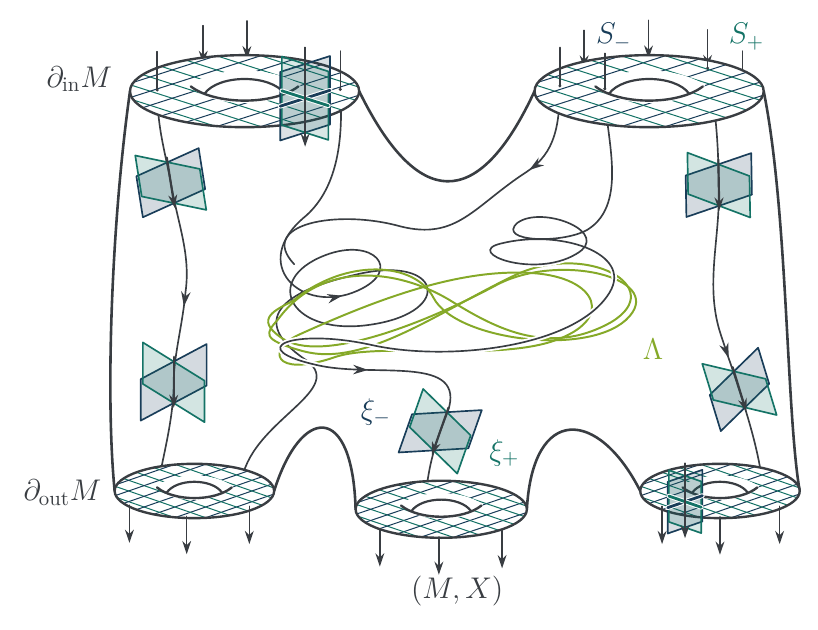}
  \caption{A schematic bi-contact plug, with characteristic
  bi-foliation $\mathcal S =(S_-,S_+)$ on its boundary and maximal invariant set $\Lambda=\operatorname{Inv}(M,X)$. Compare with Figure 4 in \cite{bby}.}
  \label{fig:bicontact-plug}
\end{figure}

To discuss the interplay with hyperbolic dynamics, we first recall some definitions. One should consult \cite{hflows} for a thorough reference on the topic. Let $X$ be nowhere vanishing and let $\varphi_X^t$ denote its flow.  The
normal bundle of the flow is denoted by
$
\mathcal N_X:=TM/\langle X\rangle
$.
The derivative $D\varphi_X^t$ induces the \emph{linear Poincar\'e flow}
$$
P_X^t\colon\mathcal N_X(x)\longrightarrow
\mathcal N_X(\varphi_X^t(x)).
$$

\begin{definition}\label{defdom}
Let $\Lambda$ be a compact invariant set disjoint from the zero set of $X$. We call $\Lambda$ \emph{projectively hyperbolic} or \emph{dominated} if the linear Poincar\'e flow has a \emph{dominated splitting} over $\Lambda$, i.e.
there is a continuous $P_X^t$-invariant splitting into line bundles
$$
\mathcal N_X|_\Lambda=E\oplus F
$$
and constants $C>0$ and $\lambda>0$ such that, for every $x\in\Lambda$ and
$t\geq0$,
$$
\left\|P_X^{t}|_{F(x)}\right\| / \left\|P_X^t|_{E(x)}\right\|
\geq Ce^{\lambda t}.
$$
That means that vectors in $F$ expand relative to vectors in $E$ at a uniform
exponential rate and we write $E\prec F$. When $\Lambda=M$ is a closed manifold, $X$ is called \emph{projectively Anosov}.
\end{definition}

The following contact geometric characterization of domination was proved for closed manifolds by Mitsumatsu \cite{mit} and later by Eliashberg-Thurston \cite{et}. By openness of the contact condition, the same cone-field argument gives the stated local form.

\begin{theorem}[Mitsumatsu 1995; local version for invariant sets]
\label{dom}
Let $X$ be a nowhere-vanishing vector field on a $3$-manifold and
$\Lambda$ a compact invariant set.  The following are equivalent:

\begin{enumerate}
  \item $\Lambda$ is projectively hyperbolic;

  \item there is a neighborhood $U$ of $\Lambda$ carrying a bi-contact
  structure $(\xi_-,\xi_+)$ supporting $X|_U$.
\end{enumerate}

In particular, the maximal invariant set of every bi-contact plug admits a
dominated splitting.  If $M$ is closed, a vector field is projectively
Anosov if and only if it is supported by a bi-contact structure.
\end{theorem}

The bi-contact condition has a useful dynamical interpretation in this theorem.
Along the flow, the two contact planes rotate strictly in opposite senses.
Their images in the projectivized normal bundle bound invariant cone fields,
and compactness turns this strict separation into exponential domination.
No transverse orientability is needed: since $\mathcal N_X$ is oriented, adapted local dual forms for the invariant line bundles can be chosen with common sign transitions; their sums and differences define unoriented contact forms. The converse uses only projective cone fields.
There are two important strengthenings of this condition: \emph{partial hyperbolicity} and \emph{saddle hyperbolicity}.

\begin{definition}
For a nowhere-vanishing $X$, a compact invariant set $\Lambda$ is \emph{partially hyperbolic} if the linear Poincar\'e flow has a \emph{dominated splitting} over $\Lambda$, i.e.
there is a continuous $P_X^t$-invariant splitting into line bundles
$$
\mathcal N_X|_\Lambda=E\oplus E^u
$$
with $E\prec E^u$, as well as $Id \prec E^u$, i.e. $$\left\|P_X^{t}|_{E^u(x)}\right\|\geq Ce^{\lambda t},$$
for every $x\in\Lambda$ and $t\geq0$, with $C>0$ and $\lambda>0$.

Moreover, a compact nonsingular invariant set $\Lambda$ is \emph{hyperbolic}
if there is a continuous $P_X^t$-invariant splitting
$$
\mathcal N_X|_\Lambda=E^s\oplus E^u
$$
and constants $C\geq1$, $\lambda>0$ such that
$$
\|P_X^t|_{E^s(x)}\|\leq Ce^{-\lambda t},
\qquad
\|P_X^{-t}|_{E^u(x)}\|\leq Ce^{-\lambda t}
$$
for every $x\in\Lambda$ and $t\geq0$, whenever the corresponding
bundle is nonzero. Either bundle may have rank zero.
If both are line bundles, $\Lambda$ is \emph{saddle hyperbolic} (note that saddle hyperbolic invariant sets are dominated).
 By Doering's lemma \cite{doering}, saddle hyperbolicity is equivalent to the
existence of a continuous invariant splitting
$$
T_\Lambda M=E^s\oplus\langle X\rangle\oplus E^u
$$
into line bundles and constants $C\geq1$ and $\lambda>0$ such that
$$
\left\|D\varphi_X^t|_{E^s(x)}\right\|\leq Ce^{-\lambda t},
\qquad
\left\|D\varphi_X^{t}|_{E^u(x)}\right\|\geq C^{-1}e^{\lambda t}
$$
for all $x\in\Lambda$ and $t\geq0$.  The vector field $X$ is \emph{Anosov} or \emph{partially hyperbolic}
if $M$ itself is a saddle hyperbolic set or partially hyperbolic, respectively. 

Finally, $X$ is called \emph{uniformly hyperbolic}, if its nonwandering set is hyperbolic, and \emph{Morse-Smale} if its nonwandering set
consists of finitely many hyperbolic periodic orbits and their
stable and unstable manifolds intersect transversely.
\end{definition}
\begin{remark}
We note that in the hyperbolic dynamics literature, the more common convention is to define partial hyperbolicity based on the existence of strong stable bundle $E^s$ in a dominated splitting (for instance, see \cite{mpp}), which is the opposite of our convention. However, our definition is more natural from the viewpoint of contact and symplectic geometry \cite{hoz6}.
\end{remark}
Note that by Theorem~\ref{dom}, the maximal invariant set of a flow supported by a bi-contact plug is always projectively hyperbolic. Hence, the following terminology is natural.

\begin{definition}
We call a bi-contact plug {\em hyperbolic} or {\em partially hyperbolic}, if the maximal invariant set of its supported flow is saddle hyperbolic or partially hyperbolic, respectively.
\end{definition}

The following lemma is useful for characterizing when partial hyperbolicity is established.

\begin{lemma}[Partial hyperbolicity from recurrent dynamics]
\label{lem:partially-hyperbolic-plug}
Let $K$ be a compact nonsingular invariant set with dominated normal
splitting $E\oplus F$, where $E\prec F$. If its nonwandering set
is contained in finitely many compact invariant subsets on each of
which $F$ is uniformly expanding, then $K$ is partially hyperbolic.

In particular, a $C^1$ structurally stable bi-contact plug is
partially hyperbolic if and only if the nonwandering set of its
maximal invariant set is a disjoint union of finitely many saddle
hyperbolic basic sets and finitely many repelling periodic orbits.
\end{lemma}

\begin{proof}
Fix a smooth metric and let
$$
r_F(x)=\left.\frac{d}{dt}\right|_{t=0}
\log\|P_X^t|_{F_x}\|.
$$
This function is continuous. Uniform expansion on the finitely
many invariant subsets gives $\kappa>0$ such that
$\int r_F\,d\mu\geq\kappa$ for every invariant probability
measure on $K$. Each ergodic measure is supported on one of these
subsets, and the general case follows by ergodic decomposition.

Consequently, for all sufficiently large $t$, uniformly in $x\in K$,
$$
\log\|P_X^t|_{F_x}\|
=\int_0^t r_F(\varphi^s x)\,ds\geq\frac{\kappa t}{2}.
$$
Otherwise, a weak limit of the corresponding normalized orbit
measures would be invariant with integral at most $\kappa/2$.
Thus $F$ is uniformly expanding on $K$.

For a structurally stable plug, the hyperbolic spectral decomposition
consists of saddle basic sets and attracting or repelling periodic
orbits. Partial hyperbolicity excludes attracting periodic orbits.
Conversely, $F$ expands on every saddle basic set and repelling
periodic orbit, so the first assertion applies.
\end{proof}


\subsection{The structure theorem for dominated splittings}\label{s2.2}

The structure theorem of Arroyo-Rodríguez Hertz \cite{arh} describes the obstructions to saddle hyperbolicity for a projectively hyperbolic invariant set.

\begin{theorem}[Arroyo-Rodr\'iguez Hertz 2003]
\label{arh}
Let $X$ be a $C^2$ vector field on a closed $3$-manifold and let $\Lambda$
be a compact invariant set such that
$
\operatorname{Sing}(X)\cap\Lambda=\varnothing
$.
Suppose that $\Lambda$ is projectively hyperbolic and every periodic orbit contained in $\Lambda$ is a
hyperbolic saddle.  Then
$$
\Lambda=\widetilde\Lambda\cup\mathcal T,
$$
where $\widetilde\Lambda$ is a saddle hyperbolic set and $\mathcal T$ is a
finite union of normally hyperbolic invariant tori on which the flow is
topologically equivalent to an irrational linear flow.
\end{theorem}

We record important consequences of the above theorem.

\begin{corollary}[Hyperbolicity criterion]
\label{cor:arh-hyperbolicity-criterion}
Let $\Lambda$ be a compact nonsingular invariant set of a $C^2$ flow on a
$3$-manifold.  If

\begin{enumerate}
  \item $\Lambda$ is projectively hyperbolic;
  \item every periodic orbit in $\Lambda$ is a hyperbolic saddle; and
  \item $\Lambda$ contains no invariant irrational torus,
\end{enumerate}

then $\Lambda$ is saddle hyperbolic.
\end{corollary}

In particular, a projectively Anosov vector field on a closed $3$-manifold
is Anosov as soon as all of its periodic orbits are hyperbolic saddles and it
has no invariant irrational torus.

The structure theorem also has the following standard generic consequence \cite{arh,bbp}.
Here, genericity is taken inside the $C^1$-open set of projectively Anosov
vector fields. 

\begin{corollary}[Generic projectively Anosov flows]
\label{corgeneric}
A $C^1$-generic projectively Anosov vector field on a closed oriented
$3$-manifold is uniformly hyperbolic and satisfies the strong transversality condition.
In particular, it is structurally stable and its nonwandering set is the
disjoint union of a saddle hyperbolic set and finitely many attracting or
repelling periodic orbits.
\end{corollary}

We call a projectively Anosov flow satisfying the conditions of Corollary~\ref{corgeneric} {\em structurally generic}.

\subsection{Hyperbolic plugs and their boundary laminations}\label{s2.3}

\emph{Hyperbolic plugs} as treated by Béguin-Bonatti-Yu \cite{bby} provide the chief motivation behind our study of bi-contact plugs. We refer to their paper for more thorough discussions and detailed definitions.

\begin{definition}\label{hplugdef}
A \emph{hyperbolic plug} is a pair $(M,X)$ consisting of a compact
$3$-manifold with boundary and a vector field transverse to $\partial M$
whose nonempty maximal invariant set
$
\Lambda=\operatorname{Inv}(M,X)
$
is saddle hyperbolic.  The plug is \emph{attracting} if
$\partial_{\mathrm{out}}M=\varnothing$ and \emph{repelling} if
$\partial_{\mathrm{in}}M=\varnothing$.
\end{definition}

The stable and unstable manifolds of $\Lambda$ meet the boundary in the
one-dimensional laminations
$$
\mathcal L^s
:=W^s(\Lambda)\cap\partial_{\mathrm{in}}M,
\qquad
\mathcal L^u
:=W^u(\Lambda)\cap\partial_{\mathrm{out}}M,
$$
called the \emph{entrance} and \emph{exit laminations}.  The crossing map is
the diffeomorphism
$$
\Gamma\colon
\partial_{\mathrm{in}}M\setminus\mathcal L^s
\longrightarrow
\partial_{\mathrm{out}}M\setminus\mathcal L^u
$$
obtained by following positive flow lines.

\begin{remark}
It is noteworthy that while our build-up to Definition~\ref{hplugdef} would make it more natural to call such plugs {\em saddle} hyperbolic, in \cite{bby} the term {\em saddle hyperbolic plug} is reserved for hyperbolic plugs whose spectral decomposition has no attractor or repeller components. In this paper, we stick with the definition of hyperbolic plugs as in \cite{bby} and do not use the terminology {\em saddle hyperbolic plug} to avoid confusion.
\end{remark}

We next recall the boundary terminology of B\'eguin-Bonatti-Yu.  An
\emph{MS (Morse-Smale) lamination} on a closed surface is a one-dimensional lamination
such that

\begin{enumerate}
  \item it has finitely many compact leaves;
  \item every half noncompact leaf is asymptotic to a compact leaf; and
  \item every compact leaf admits an orientation for which its holonomy is a
  contraction.
\end{enumerate}

The entrance and exit laminations of every hyperbolic plug are MS
laminations.  An MS lamination is a \emph{prefoliation} if it is contained in
a nonsingular foliation of the ambient surface.  For a hyperbolic plug,
$\mathcal L^s$ is a prefoliation if and only if $\mathcal L^u$ is a
prefoliation.  Note that this is a necessary condition for the embedding of a hyperbolic plug in an Anosov flow. An MS lamination is \emph{filling} if every component of its
complement is a strip whose width \emph{tends to 0 at both ends}.  For entrance and exit laminations of an orientable
hyperbolic plug which are prefoliations, this is equivalent to the absence of \emph{exceptional annular
components}.  Filling MS laminations are prefoliations and the filling condition is equivalent to having a unique extension, up to homeomorphism, to a foliation.

Here are our favorite sources of hyperbolic plugs.

\begin{example}[Reduction of a nonsingular uniformly hyperbolic flow]
Let $X$ generate a nonsingular uniformly hyperbolic flow satisfying the strong transversality condition.  All attracting and repelling
periodic orbits in its maximal invariant set form a finite collection of
hyperbolic orbits.  Excise pairwise disjoint filtrating solid torus
neighborhoods of these orbits.  If the maximal invariant set of the remaining
plug is nonempty, it is saddle hyperbolic and the resulting compact manifold equipped with $X$ is a hyperbolic plug. In the case when $X$ generates a structurally generic projectively Anosov flow (see Corollary~\ref{corgeneric}), the induced entrance and exit laminations are prefoliations.
\end{example}

\begin{example}[Derived-from-Anosov plugs]
Start with an Anosov flow and perform attracting or repelling
derived-from-Anosov modifications along finitely many periodic orbits.  After
removing suitable tubular neighborhoods of the new attracting or repelling
orbits, the complement is a hyperbolic plug.  B\'eguin-Bonatti-Yu show that
these plugs have filling MS entrance and exit laminations.  The
removed neighborhoods provide the elementary attracting and repelling
pieces used in their gluing constructions. See \cite{fw} or Section~8.1.1 in \cite{bby}.
\end{example}

\begin{example}[Suspended gradient plugs]
Let $\Sigma$ be a compact surface with boundary and let $f\colon\Sigma\to
\mathbb R$ be a Morse function whose critical points all have index one.
Choose a metric for which $\nabla f$ is transverse to $\partial\Sigma$ and
Morse-Smale.  Then
$
X=\partial_\theta+\nabla f
$
defines a hyperbolic plug on $\mathbb S^1_\theta\times\Sigma$.  Indeed, the
absence of saddle connections implies
$$
\operatorname{Inv}(\mathbb S_\theta^1\times\Sigma,X)
=
\coprod_{p\in\operatorname{Crit}(f)}
\mathbb S_\theta^1\times\{p\}.
$$
Each component is a saddle periodic orbit, since the Hessian of $f$ at an
index-one critical point has one positive and one negative eigenvalue. See Figure~\ref{fig:suspended-gradient-plug}.
\end{example}

\begin{figure}[htbp]
  \centering
  \includegraphics[width=.7\textwidth]{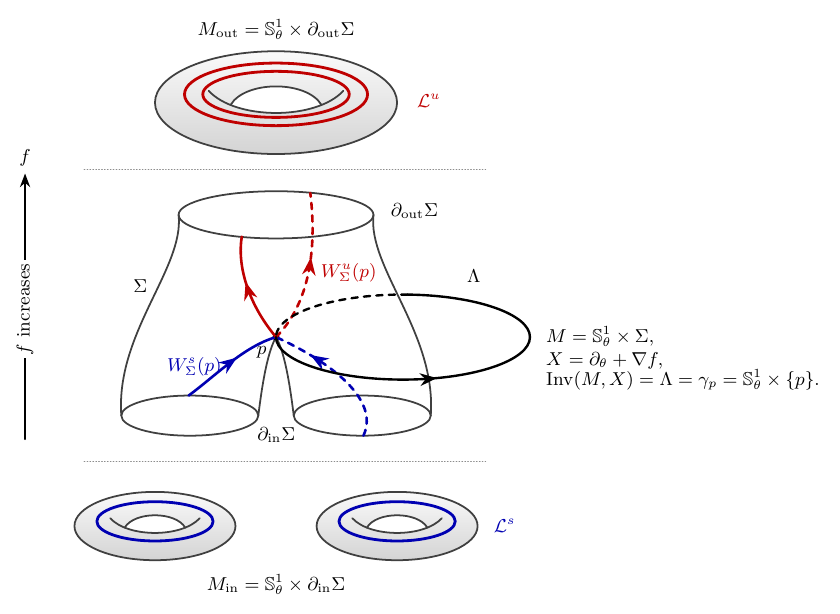}
  \caption{The suspended-gradient plug associated with a pair of pants
  $\Sigma$ and a Morse function having one index-one critical point $p$.
  The stable separatrices (blue) meet the two entrance components, while the
  unstable separatrices (red) meet the exit component.  After taking the
  product with $\mathbb S^1_\theta$, the point $p$ becomes the saddle
  periodic orbit $\gamma_p$, and the endpoints of the separatrices become
  the compact leaves of the entrance and exit laminations.}
  \label{fig:suspended-gradient-plug}
\end{figure}

We shall also use the following extension result of
B\'eguin-Bonatti-Yu.  Their \emph{$\lambda$-lemma} argument gives an invariant
foliation of class $C^{0,1}$, meaning that its leaves are $C^1$ and tangent
to a continuous plane field.  In
Lemma~\ref{lem:regularity-invariant-completions}, we establish a regularity refinement of this lemma, showing that for a smooth
three-dimensional hyperbolic plug these invariant completions are in fact
of class $C^{1+\vartheta}$ for some $\vartheta>0$.

\begin{lemma}[B\'eguin-Bonatti-Yu 2017; Invariant extension]
\label{lem:bby-invariant-extension}
Assume that the entrance and exit laminations of $(M,X)$ are prefoliations.
Every foliation $F_{\mathrm{out}}$ on $\partial_{\mathrm{out}}M$ transverse
to $\mathcal L^u$ extends to an $X$-invariant $C^{0,1}$ foliation of $M$
containing $W^s(\Lambda)$ as a sublamination.  Its restriction to
$\partial_{\mathrm{in}}M$ is a foliation containing $\mathcal L^s$.
The time-reversed statement holds with $s$ and $u$ interchanged.
\end{lemma}

Two filling MS laminations $\mathcal L_1$ and $\mathcal L_2$ on the same
surface $S$ are \emph{strongly transverse} if they are transverse and every
component of
$
S\setminus(\mathcal L_1\cup\mathcal L_2)
$
is a topological disk whose boundary consists of four alternating segments,
two contained in leaves of each lamination.



Gluing two distinct hyperbolic plugs is comparatively straightforward. B\'eguin-Bonatti-Yu showed that any gluing diffeomorphism which maps the exit lamination transversely to the entrance lamination produces a hyperbolic plug, without requiring strong transversality. Self-gluing is considerably more subtle because it creates return orbits. We give an independent bi-contact treatment of both situations in Section~\ref{s5}. The main self-gluing theorem of B\'eguin-Bonatti-Yu may be stated as follows.


\begin{theorem}[B\'eguin-Bonatti-Yu 2017; Self-gluing]
\label{thm:bby-main-gluing}
Let $(M,X)$ be a hyperbolic plug with filling MS laminations, and suppose that
its maximal invariant set has no basic piece which is an attractor or a
repeller.  Let
$$
\psi\colon\partial_{\mathrm{out}}M
\longrightarrow\partial_{\mathrm{in}}M
$$
be strongly transverse, meaning that
$\psi_*(\mathcal L^u)$ and $\mathcal L^s$ are strongly transverse.  Then one
may deform $X$ through topologically equivalent hyperbolic plugs and deform
$\psi$ through strongly transverse gluing maps so that the vector field on
the resulting closed quotient is Anosov.
\end{theorem}

It is noteworthy that B\'eguin-Bonatti-Yu also proved that every hyperbolic plug with filling
MS laminations embeds, up to topological equivalence, into an Anosov flow.



These results form the dynamical starting point for our constructions.  Our
goal is to replace part of the dynamical adjustment in
Theorem~\ref{thm:bby-main-gluing} by deformation and gluing of the supporting
contact geometry.

\subsection{Strongly adapted contact geometry}

There is a stronger contact geometric criterion which detects actual
hyperbolicity rather than domination alone. The idea was developed by the author in \cite{hoz7} as a purely contact geometric characterization of Anosov flows in dimension three. The same criterion applies to an arbitrary compact invariant set of a flow on a $3$-manifold.

\begin{definition}
Let $\Lambda$ be a compact invariant set and let $\alpha_+$ be a positive
contact form defined on a neighborhood of $\Lambda$.  We say that
$\alpha_+$ is \emph{strongly adapted} to $X$ near $\Lambda$ if
$
\alpha_+(X)=0
$
and
$
\mathcal L_X\alpha_+
$
is a negative contact form.  Its first $X$-jet is denoted by
$$
j_X^1\alpha_+
:=
[(\alpha_+,\mathcal L_X\alpha_+)]_{\pm},
$$
where the sign acts simultaneously on both entries.
\end{definition}

If $\alpha_+$ is strongly adapted and
$\alpha_-:=\mathcal L_X\alpha_+$, then Cartan's formula gives
$$
\alpha_-\wedge\alpha_+
=
\iota_X(\alpha_+\wedge d\alpha_+)>0.
$$
Thus
$$
(\ker\alpha_-,\ker\alpha_+)
$$
is a supporting bi-contact structure.  Strong adaptation is precisely the
additional differential condition which upgrades the resulting domination
to saddle hyperbolicity.

The following is the compact invariant-set version of the author's strongly
adapted contact geometric characterization of Anosov flows \cite{hoz7}.  Although the
original statement concerns closed manifolds, the very same proof applies locally along the invariant set.  Indeed, one chooses adapted norms on the
stable and unstable line bundles and applies asymptotic synchronization so
that the associated infinitesimal rates satisfy the strict contact
inequality. Compactness makes all the
estimates uniform. The openness of the contact condition then allows one to extend and smoothen the
resulting form on a neighborhood of $\Lambda$.
The norms and rates are independent of orientations, and the local forms have common sign transitions. Extension and smoothing may be performed equivariantly on the corresponding double cover, so the proof applies to unoriented contact forms.


\begin{theorem}[H. 2025; local version for invariant sets]
\label{adapted}
Let $X$ be a nowhere-vanishing smooth vector field on an oriented
$3$-manifold and let $\Lambda$ be a compact invariant set.  Then $\Lambda$
is a saddle hyperbolic set if and only if there is a
neighborhood of $\Lambda$ carrying a positive contact form $\alpha_+$ such
that
$
\alpha_+(X)=0
$
and
$
\mathcal L_X\alpha_+
$
is a negative contact form.
\end{theorem}


When $\Lambda=M$ is closed, Theorem~\ref{adapted} gives the purely
contact geometric characterization of Anosov flows in \cite{hoz7} without requiring transverse orientability.


Theorems~\ref{arh} and~\ref{adapted} give two criteria for hyperbolicity. The first uses a dominated splitting together with the absence of dynamical obstructions. The second uses a strongly adapted contact form.
\section{Construction of bi-contact plugs}\label{s3}

This section develops the three local models used throughout the paper.  We
first treat product plugs and their strongly adapted refinement.  We then
extend the product construction across the stable and unstable sets of a
hyperbolic plug.  Finally, we construct non-hyperbolic plugs around normally
attracting or repelling invariant surfaces. We use the unoriented forms of
Section~\ref{s2.1}, also for defining forms of characteristic foliations.
Forms used together have common sign transitions. Local coframes and
representatives are transported consistently along the flow. The scalar
coefficients and inequalities below are unchanged by their simultaneous
sign reversal.

\subsection{Product bi-contact plugs}\label{s3.1}




We begin with bi-contact plugs whose maximal invariant set $\operatorname{Inv}(M,X)$ is empty. The supported flow carries each entrance component to an exit component in finite time, giving $M$ a product structure. After reparametrizing the flow, we have
$$
M=[0,1]_t\times\Sigma,
\qquad
X=\partial_t,
$$
where $\Sigma$ is a union of tori and the product flow gives an identification of $\partial_{in}M$ and $\partial_{out}M$. We call these \emph{product bi-contact plugs}. By boundary transversality, the same model describes a neighborhood of each boundary component of any bi-contact plug.

Let
$\alpha_-,\alpha_+$ be supporting bi-contact forms and put
$$
S_\pm^{\mathrm{in}}
=\ker(\alpha_\pm|_{\{0\}\times\Sigma}),
\qquad
S_\pm^{\mathrm{out}}
=\ker(\alpha_\pm|_{\{1\}\times\Sigma}).
$$
Choose a $t$-invariant positive coframe $(\beta_1,\beta_2)$ on the slices.
A local representative of a form $\alpha$ annihilating $\partial_t$ has a unique expression
$$
\alpha=\phi_1\beta_1+\phi_2\beta_2,
\qquad
\gamma=(\phi_1,\phi_2).
$$
Since tangential three-forms vanish,
$$
\iota_{\partial_t}(\alpha\wedge d\alpha)
=(\mathcal L_{\partial_t}\alpha)\wedge\alpha
=\det(\partial_t\gamma,\gamma)\,\beta_1\wedge\beta_2.
$$
Writing $\gamma$ in local polar coordinates, with $\theta$ defined modulo $\pi$,
$$
\gamma=r(\cos\theta,\sin\theta),
\qquad r>0,
$$
we obtain
$
\det(\partial_t\gamma,\gamma)=-r^2\partial_t\theta.
$
Thus a positive contact form rotates clockwise and a negative contact form
rotates counterclockwise in the oriented projective line.  Since the two
contact planes remain transverse, their four boundary directions must occur
in the order
\begin{equation}
\label{adnec}
S_-^{\mathrm{in}},
\quad
S_-^{\mathrm{out}},
\quad
S_+^{\mathrm{out}},
\quad
S_+^{\mathrm{in}}.
\end{equation}

\begin{figure}[htbp]
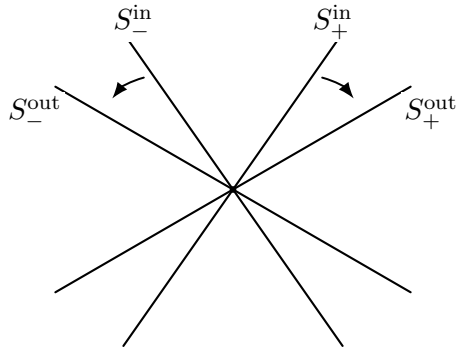
\label{fig2}
    \centering
    \ProductAdmissibilityDiagram
    \caption{The cyclic order of an admissible boundary bi-foliation.}
    \label{fig:admissible-product}
\end{figure}

\begin{definition}
\label{admissible-bifoliation}
A bi-foliation on $\partial([0,1]\times\mathbb T^2)$ is
\emph{admissible} if its four projective directions satisfy the cyclic
order in Equation~\eqref{adnec}. Choose local defining covectors in that
order in a common open half-plane. This fixes their relative signs,
and their transition signs coincide. Boundary defining forms are called
admissible when their kernels are admissible, and are compared using these
representatives. See Figure~\ref{fig:admissible-product}.
\end{definition}

The cyclic-order condition is stronger than the four adjacent wedge-sign
conditions considered separately. It fixes compatible lifts of all four
directions to one interval in the universal cover of the projective line and
therefore excludes an additional turn between the endpoints.

\begin{theorem}[Product extension]
\label{prob}
Every admissible boundary bi-foliation on $[0,1]\times\mathbb T^2$
is induced by a bi-contact structure supporting $\partial_t$.
If such a structure is prescribed near the boundary, then, after
restricting it to a sufficiently small boundary neighborhood, its
space of supporting extensions is nonempty and contractible.

These extension results are relative and parametric for compact
parameter spaces. In particular, every homotopy of admissible
boundary bi-foliations lifts to a homotopy of supporting bi-contact
structures starting from any prescribed initial extension.

The same statements hold for $[0,1]\times\Sigma$ when $\Sigma$ has
boundary, provided a supporting bi-contact structure compatible
with the prescribed boundary bi-foliation is fixed near the
vertical boundary $[0,1]\times\partial\Sigma$.
\end{theorem}

\begin{proof}
Choose compatible local representatives $\lambda_+^{\mathrm{in/out}}$ for the
prescribed foliations $S_+^{\mathrm{in/out}}$.
For the positive form, use the local $t$-invariant basis
$$
\beta_1=\pi^*\lambda_+^{\mathrm{out}},
\qquad
\beta_2=\pi^*\lambda_+^{\mathrm{in}},
$$
oriented so that $\beta_1\wedge\beta_2>0$, and write
$$
\alpha_+
=r_+(\cos\theta_+\,\beta_1+\sin\theta_+\,\beta_2).
$$
The endpoint values are $\theta_+(0)=\pi/2$ and $\theta_+(1)=0$, and
the contact condition is exactly $\partial_t\theta_+<0$.  Positive radial functions and
strictly decreasing angular interpolations with prescribed behavior near the endpoints form
convex spaces.

The negative form is constructed in the same way using the oppositely
oriented basis determined by its outgoing and incoming boundary forms.  The
two angular interpolations take place in the two disjoint sectors specified by
Equation~\eqref{adnec}. Hence, their kernels remain transverse. The radial
and angular functions agree on overlaps, so the local forms glue up to sign.
All choices are convex and depend continuously on parameters.  The same construction on
the complementary subcylinder proves the relative statement when the bi-contact structure is fixed near the boundary.
\end{proof}

\subsubsection{Strongly adapted product plugs}

Given a product plug $([0,1]_t\times \Sigma,\partial_t)$, define the space of strongly adapted contact forms
$$
\mathcal{SA}^+(\partial_t)
=
\left\{
\alpha_+\ \middle|\
\alpha_+(\partial_t)=0,\quad
\alpha_+\wedge d\alpha_+>0,\quad
(\mathcal L_{\partial_t}\alpha_+)
\wedge d(\mathcal L_{\partial_t}\alpha_+)<0
\right\}.
$$
For $\alpha_+\in\mathcal{SA}^+(\partial_t)$, the forms $\mathcal L_{\partial_t}\alpha_+$ and $\alpha_+$ are supporting bi-contact forms.

Write $\alpha_+=\gamma_1\beta_1+\gamma_2\beta_2$ in any positive
$t$-invariant coframe and put $\gamma=(\gamma_1,\gamma_2)$.  Then,
$$
\alpha_+\text{ is positive contact}
\quad\Longleftrightarrow\quad
\det(\partial_t\gamma,\gamma)>0,
$$
whereas
$$
\mathcal L_{\partial_t}\alpha_+\text{ is negative contact}
\quad\Longleftrightarrow\quad
\det(\partial_t\gamma,\partial_t^2\gamma)>0.
$$
Since the first condition gives $\partial_t\theta<0$, we may use $\theta$ as
a reversed parameter and write $\gamma=r(\theta)(\cos\theta,\sin\theta)$.
A direct calculation gives
\begin{equation}
\label{streq}
\det(\partial_t\gamma,\partial_t^2\gamma)>0
\quad\Longleftrightarrow\quad
r^2+2r_\theta^2-r r_{\theta\theta}<0
\quad\Longleftrightarrow\quad
u_{\theta\theta}+u<0,
\end{equation}
where $u=1/r$.  In other words, strong adaptation is the strict signed
curvature condition for the coefficient curve.

In the following, a \emph{boundary first $\partial_t$-jet}, written
$(\lambda,\mu)$, consists of local boundary one-forms modulo simultaneous
sign change, prescribing
$$
\alpha_+|_{\partial_{\mathrm{in/out}}M}
=\lambda^{\mathrm{in/out}},
\qquad
(\mathcal L_{\partial_t}\alpha_+)
|_{\partial_{\mathrm{in/out}}M}
=\mu^{\mathrm{in/out}}
$$
for a form $\alpha_+$ satisfying $\alpha_+(\partial_t)=0$.
When these boundary forms come from supporting bi-contact forms
$(\alpha_-,\alpha_+)$, we take $\lambda$ and $\mu$ to be their
positive and negative boundary restrictions, respectively.
Thus the negative boundary form is interpreted as the prescribed
$\partial_t$-derivative of the positive form to be extended.

\begin{definition}
\label{strongly-admissible-jet}
The boundary $\partial_t$-jet $(\lambda,\mu)$ is \emph{strongly admissible}
if compatible local representatives satisfy
$$
\lambda^{\mathrm{out}}\wedge\lambda^{\mathrm{in}}>0
$$
and
\begin{equation}
\label{sadin}
-\lambda^{\mathrm{in}}\wedge\mu^{\mathrm{in}}>0,
\qquad
(\lambda^{\mathrm{out}}-\lambda^{\mathrm{in}})
\wedge\mu^{\mathrm{in}}<0,
\end{equation}
\begin{equation}
\label{sadout}
\lambda^{\mathrm{out}}\wedge\mu^{\mathrm{out}}<0,
\qquad
(\lambda^{\mathrm{out}}-\lambda^{\mathrm{in}})
\wedge\mu^{\mathrm{out}}>0.
\end{equation}
These inequalities say that the two endpoint tangent vectors lie on the
correct sides of the chord joining the endpoint values. Simultaneous sign
change preserves every inequality. See Figure~\ref{fig:strongly-admissible}.
\end{definition}

\begin{figure}[htbp]
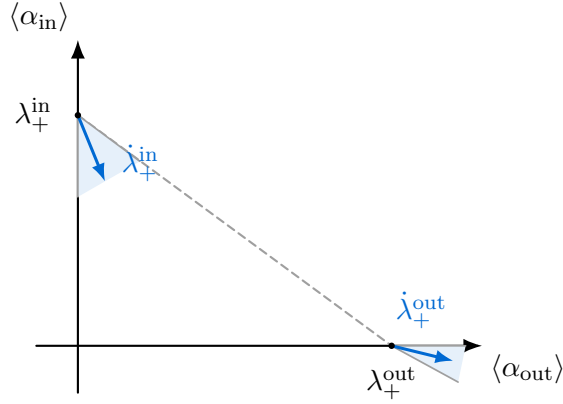

    \centering
    \StrongAdmissibilityDiagram
    \caption{The endpoint tangent cones defining strong admissibility.}
    \label{fig:strongly-admissible}
\end{figure}

To express these conditions radially, choose a positive
$t$-invariant coframe in which
$\lambda^{\mathrm{out}}=r_{\mathrm{out}}\beta_1$ and
$\lambda^{\mathrm{in}}=r_{\mathrm{in}}\beta_2$.
Thus the endpoint angles are $0$ and $\pi/2$, respectively.
Under the positive contact condition $\partial_t\theta<0$,
Equations~\eqref{sadin} and~\eqref{sadout} are equivalent to
\begin{equation}
\label{radad}
r_\theta(0)<-\frac{r_{\mathrm{out}}^2}{r_{\mathrm{in}}},
\qquad
r_\theta(\pi/2)>\frac{r_{\mathrm{in}}^2}{r_{\mathrm{out}}}.
\end{equation}

The following elementary observation will be used only for preparing
forms near the boundary before applying the extension theorem.  It is not, by itself,
a compactly supported modification of a form already defined in the
interior.

\begin{lemma}[Scaling near the boundary]
\label{scaling-strong-admissibility}
Suppose strongly adapted forms are prescribed near the two ends and their
radial derivatives satisfy
$$
r_\theta(0)<0<r_\theta(\pi/2).
$$
Then their boundary first jets can be made strongly admissible either by sufficiently large
angular scaling at both ends, or by constant radial scaling at one end
followed there by angular scaling.
\end{lemma}

\begin{proof}
Angular scaling by $K\geq1$ replaces $u(\theta)$ near an endpoint by
$u(\theta_0+K(\theta-\theta_0))$, where $\theta_0$ is the angle at the endpoint.  Hence
$$
\widetilde u_{\theta\theta}+\widetilde u
=K^2u_{\theta\theta}+u<0,
$$
so strong adaptation is preserved, while the endpoint radial derivative is
multiplied by $K$.  Constant radial scaling also preserves both contact
conditions.  The conclusion now follows from Equation~\eqref{radad}.
\end{proof}

\begin{theorem}[Strong product extension]
\label{pros}
Every strongly admissible boundary first $\partial_t$-jet extends
to a strongly adapted positive contact form on
$[0,1]\times\mathbb T^2$.
If a strongly adapted form realizing this jet is prescribed near
the boundary, it can be preserved on a sufficiently small boundary
neighborhood, and the space of extensions agreeing there is
contractible.

The construction is relative and parametric for compact parameter
spaces. It also holds for $[0,1]\times\Sigma$ when $\Sigma$ has
boundary, provided a compatible strongly adapted form is prescribed
near $[0,1]\times\partial\Sigma$.
\end{theorem}

\begin{proof}
Using the product identification, take the positive local $t$-invariant
coframe
$$
\beta_1=\pi^*\lambda^{\mathrm{out}},
\qquad
\beta_2=\pi^*\lambda^{\mathrm{in}}.
$$
We seek the extension in the form
$$
\alpha_+
=\frac{\cos\theta\,\beta_1+\sin\theta\,\beta_2}{u(\theta,x)},
\qquad
u>0,
$$
where $\theta(0,x)=\pi/2$, $\theta(1,x)=0$, and
$\partial_t\theta<0$.
The prescribed first jet determines the endpoint values of
$\partial_t\theta$. Interpolating the positive speed
$-\partial_t\theta$, with integral $\pi/2$, gives such angular
functions. Their space is convex. If a form is prescribed near
the boundary, the interpolation can agree with its angular
function on smaller boundary neighborhoods.

Suppressing $x$ from the notation, the radial boundary conditions
are $u(0)=u(\pi/2)=1$, and strong admissibility gives
$$
a:=u_\theta(0)-1>0,
\qquad
b:=-1-u_\theta(\pi/2)>0.
$$
Set $p=-(u_{\theta\theta}+u)$.
By Equation~\eqref{streq}, strong adaptation requires $p>0$.
The solution with the prescribed endpoint values is
\begin{align*}
u(\theta)={}&\cos\theta+\sin\theta\\
&+\sin\theta\int_\theta^{\pi/2}p(s)\cos s\,ds
+\cos\theta\int_0^\theta p(s)\sin s\,ds.
\end{align*}
In particular, $u>0$. Differentiating at the endpoints shows that
the prescribed radial derivatives are equivalent to
\begin{equation}
\label{strong-product-moments}
\int_0^{\pi/2}p(s)\cos s\,ds=a,
\qquad
\int_0^{\pi/2}p(s)\sin s\,ds=b.
\end{equation}

Choose a smooth positive function $p_0$ whose two moments are
smaller than $a$ and $b$. If a strongly adapted form is prescribed
near the endpoints, require $p_0$ to agree there with its function
$-(u_{\theta\theta}+u)$. This is possible after shrinking those
neighborhoods, since their moment contributions tend to zero.
Elsewhere, choose $p_0$ sufficiently small.

The residual moments are positive. Choose two nonnegative
interior bump functions, one near each endpoint, whose ratios
of sine to cosine moments lie on opposite sides of the
corresponding residual ratio. Adding uniquely determined positive
multiples of these functions to $p_0$ gives
\eqref{strong-product-moments}. For compact parameter families,
the endpoint neighborhoods and bump functions can be chosen
uniformly, and the coefficients depend continuously on parameters,
smoothly for smooth families. When a form is prescribed near the
boundary, uniqueness for $u_{\theta\theta}+u=-p$, with the
prescribed endpoint value and derivative, ensures agreement there.

For fixed boundary first jets, the positive functions $p$
satisfying \eqref{strong-product-moments} form a convex space.
The same holds when their values near the endpoints are fixed.
Together with the convex angular choices, this proves
contractibility and gives the relative and parametric
constructions. In particular, when a compatible form is fixed
near the vertical boundary, interpolate the angular functions
and the functions $p$ separately using a cutoff on $\Sigma$.
All the required inequalities and moment equations are preserved.

Finally, these coordinates describe every strongly adapted
extension. An additional angular turn would give an interval
of length $\pi$ on which $u>0$ and $u_{\theta\theta}+u<0$,
contrary to
$$
0>\int_0^\pi (u_{\theta\theta}+u)\sin\theta\,d\theta
=u(0)+u(\pi)>0.
$$
\end{proof}

\paragraph{Winding conventions.}
In a specified coordinate framing of a torus $T$, a nonsingular line
field $S$ defines a map $T\to\mathbb{RP}^1$ and its \emph{projective
winding class}
$$
d(S)\in H^1(T;\mathbb Z),\qquad d(S)\bmod2=w_1(S).
$$
If $S$ is orientable, its \emph{oriented winding class} is
$w(S)=d(S)/2$. Reversing its orientation does not change $w$.
The two members of a transverse bi-foliation have the same class.
We call a coorientable bi-foliation's oriented winding \emph{even}
if $w\equiv0\pmod2$, and \emph{odd} otherwise. These conditions are
independent of the torus coordinates. Thus even projective winding
means coorientability, while odd oriented winding means
$d\not\equiv0\pmod4$ with $d\equiv0\pmod2$.

\begin{remark}[Homotopy type of product extensions]
\label{product-homotopy-type}
Let $\mathcal F_{\mathrm{ad}}^b(\partial M)$ denote the space of
admissible boundary bi-foliations for
$M=[0,1]\times\mathbb T^2$, with the $C^\infty$ topology. Then
$$
\mathcal F_{\mathrm{ad}}^b(\partial M)
\simeq C^\infty(\mathbb T^2,\mathbb{RP}^1)
\simeq \coprod_{\mathbf k\in\mathbb Z^2}\mathbb S^1.
$$
In a coordinate framing of $T\mathbb T^2$, the two integers record
the projective winding of any one of the four boundary foliations along a
homology basis. Coorientable bi-foliations correspond to $\mathbf k\in2\mathbb Z^2$.
Each component has the homotopy type of a circle.

Indeed, after fixing $S_-^{\mathrm{in}}$, the remaining directions
are described by smooth angular differences satisfying
$$
0<\delta_1(x)<\delta_2(x)<\delta_3(x)<\pi,
\qquad x\in\mathbb T^2,
$$
with relative signs determined by admissibility.
This is a convex condition, so the space of choices is
contractible. The displayed homotopy equivalence therefore follows
from the homotopy classification of maps to $\mathbb S^1$.

For a fixed admissible boundary bi-foliation, the choices of
unoriented defining forms are contractible, as are those satisfying
the strong-admissibility inequalities. The relative parametric
extension theorems consequently give the same homotopy type for
the spaces of supporting bi-contact structures and strongly
adapted positive contact forms, with $X=\partial_t$ fixed.
\end{remark}

\subsection{Hyperbolic bi-contact plugs}
Our goal in this subsection is to equip hyperbolic plugs of \cite{bby} with (strongly) adapted contact geometry.
Let $(M,X)$ be a hyperbolic plug with maximal
invariant set $\Lambda$.  Assume that its entrance and exit laminations
$$
\mathcal L^s
=W^s(\Lambda)\cap\partial_{\mathrm{in}}M,
\qquad
\mathcal L^u
=W^u(\Lambda)\cap\partial_{\mathrm{out}}M
$$
are prefoliations.  By Lemma~3.17 in \cite{bby}, it is enough to make
this assumption for either one of them.  We denote the crossing map by
$$
\Gamma\colon
\partial_{\mathrm{in}}M\setminus\mathcal L^s
\longrightarrow
\partial_{\mathrm{out}}M\setminus\mathcal L^u.
$$

\subsubsection{Admissible boundary bi-foliations}

We describe the necessary conditions on the boundary characteristic bi-foliation and prove that they are sufficient for extension to a supporting bi-contact structure.

Let $\lambda_-,\lambda_+$ be boundary defining forms with
$\lambda_-\wedge\lambda_+>0$ in compatible local representatives, and write
$S_\pm=\ker\lambda_\pm$.

\begin{definition}
\label{admissible-hyperbolic-boundary}
These forms are \emph{admissible} if the following conditions hold.

\begin{enumerate}
   \item after transport by the crossing map $\Gamma$, the entrance
and exit bi-foliations satisfy the product admissibility
condition of Definition~\ref{admissible-bifoliation}. That is, for every
    $p\in\partial_{\mathrm{in}}M\setminus\mathcal L^s$, with
    $q=\Gamma(p)$, the directions occur in the cyclic order
    $$
    D\Gamma_p(S_-^{\mathrm{in}}(p)),
    \quad
    S_-^{\mathrm{out}}(q),
    \quad
    S_+^{\mathrm{out}}(q),
    \quad
    D\Gamma_p(S_+^{\mathrm{in}}(p));
    $$

    \item for every nonzero $v_s\in T\mathcal L^s$,
    $$
    \lambda_-^{\mathrm{in}}(v_s)
    \lambda_+^{\mathrm{in}}(v_s)<0;
    $$

    \item for every nonzero $v_u\in T\mathcal L^u$,
    $$
    \lambda_-^{\mathrm{out}}(v_u)
    \lambda_+^{\mathrm{out}}(v_u)>0.
    $$
\end{enumerate}
\end{definition}

The last two conditions say that the contact directions straddle the stable
and unstable laminations in the limiting orders forced by the flow.  Their
signs are independent both of the orientation chosen on a lamination leaf
and of simultaneous sign change of the defining forms. Thus no orientation
of $T\mathcal L^s$ or $T\mathcal L^u$ is required, even near the laminations.
See Figure~\ref{fig:admissible-hyperbolic}.

\begin{figure}[htbp]
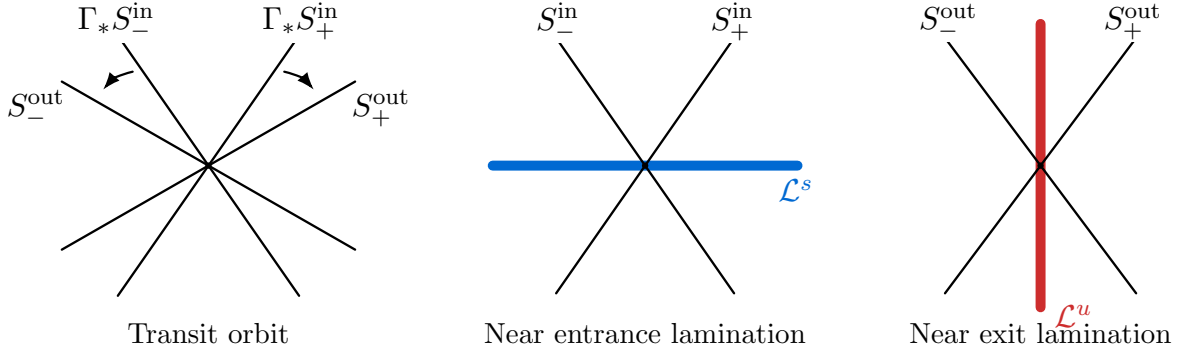

    \centering
    \resizebox{\linewidth}{!}{\HyperbolicAdmissibilityDiagram}
    \caption{Admissible boundary bi-foliations for a hyperbolic bi-contact plug. On
    transit orbits one has the product cyclic order. At $\mathcal L^s$ and
    $\mathcal L^u$ the contact directions straddle the corresponding
    lamination in opposite orders.}
    \label{fig:admissible-hyperbolic}
\end{figure}

\begin{theorem}[Hyperbolic boundary extension]
\label{hypb}
Let $(M,X)$ be as above.  Every admissible choice of boundary defining forms extends
to supporting bi-contact forms on $M$.  More precisely, one may fix a
supporting bi-contact structure on an isolating neighborhood of $\Lambda$ and a prescribed bi-contact structure near $\partial M$ with admissible characteristic bi-foliation, and find an extension agreeing with both
on smaller neighborhoods.  The construction is relative and parametric for
compact parameter spaces.
\end{theorem}

\begin{proof}
Write $\varphi^t$ for the flow of $X$. Fix the prescribed bi-contact structure near $\partial M$
and a supporting bi-contact structure in the neighborhood $N(\Lambda)$ of $\Lambda$ given by Theorem~\ref{dom}. We only require agreement with this bi-contact structure on a smaller
neighborhood of $\Lambda$.

First, we extend near the stable set. By admissibility, the incoming
bi-foliation is transverse to the stable direction along $\mathcal{L}^s$, so the forward images of its two foliations
approach the unstable direction by the inclination lemma (see \cite{hflows}). 
By compactness, we may choose a single time $T_s$ for which $\varphi^{t}(\mathcal L^s)$ lies in $N(\Lambda)$ for $t\geq T_s$ and the cyclic order in Equation~\eqref{adnec} holds on the product $[0,T_s]\times \mathcal L^s$ (here $S_\pm^{in}$ is the push-forward of the prescribed $S_\pm$ on $\mathcal{L}^s$ and $S_\pm^{out}$ is given by the bi-contact structure we constructed on $N(\Lambda)$). Openness then gives a compact neighborhood
$N_s$ of $\mathcal L^s$ on which the flow tube
$$
\{\varphi^t(x):x\in N_s,\ 0\leq t\leq T_s\}
$$
has admissible characteristic bi-foliations induced by the prescribed structure at its entrance and the structure near $\Lambda$ at its terminal slice. Theorem~\ref{prob} extends the bi-contact structure across this tube, preserving the prescribed structures near both ends. Applying the same argument backward from $\mathcal L^u$
fills the unstable tubes. Shrink the neighborhood of $\Lambda$ on which the original structure is preserved and the
tube neighborhoods so that their overlaps lie in the terminal collars,
where the constructions agree. We have therefore defined the bi-contact structure on a
neighborhood $V$ of the stable and unstable sets in $M$, including
$\Lambda$, and of $\partial M$.

It remains to fill a compact product region. Every point of $M\setminus V$
lies on a finite transit orbit. Projecting these points backward to
$\partial_{\mathrm{in}}M$ gives a compact set
$A\subset\partial_{\mathrm{in}}M\setminus\mathcal L^s$.
Choose a compact subsurface $\Sigma_0$ of this complement with
$A\subset\operatorname{int}(\Sigma_0)$. Its full transit saturation is
$$
R=\{\varphi^t(x):x\in\Sigma_0,\ 0\leq t\leq\tau(x)\}
\simeq[0,1]\times\Sigma_0,
$$
where $\tau$ is the smooth transit-time function. The vertical boundary
of $R$ has a collar contained in $V$, since its orbits avoid $M\setminus V$.
Moreover, $R$ is compact and disjoint from the stable and unstable sets.

Preserve the previous construction outside $R$ and near its vertical
boundary. The characteristic bi-foliations on the horizontal faces of $R$ satisfy the product cyclic order by
admissibility along the crossing map $\Gamma$. The relative form of
Theorem~\ref{prob} fills $R$ and agrees with the preserved portion. This gives
the required bi-contact structure on all of $M$, agreeing with the prescribed bi-contact structure near
$\Lambda$.

For compact parameter families, the inclination estimates and strict
inequalities allow uniform choices of times and neighborhoods. Applying
the relative parametric version of Theorem~\ref{prob} at each step gives
the corresponding relative parametric extension.
\end{proof}

\subsubsection{Strongly admissible boundary jets}

We now give necessary and sufficient conditions for a boundary $X$-jet to extend to a strongly adapted contact form.

For a positive contact form annihilating $X$, write
$$
j_X^1\alpha_+|_{\partial M}
=[(\lambda,\mu)]_{\pm}
:=[(\alpha_+,\mathcal L_X\alpha_+)|_{\partial M}]_{\pm}.
$$
We call this jet \emph{strongly admissible} if $(\mu,\lambda)$ is admissible
in the sense of Definition~\ref{admissible-hyperbolic-boundary} and, on every
transit orbit, the pulled-back incoming and outgoing jets satisfy
Equations~\eqref{sadin} and~\eqref{sadout}.  Explicitly, after pulling the
outgoing forms back by $\Gamma$ and choosing the common signs fixed by
product admissibility, the additional inequalities are
$$
(\Gamma^*\lambda^{\mathrm{out}}-\lambda^{\mathrm{in}})
\wedge\mu^{\mathrm{in}}<0,
$$
and
$$
(\Gamma^*\lambda^{\mathrm{out}}-\lambda^{\mathrm{in}})
\wedge\Gamma^*\mu^{\mathrm{out}}>0.
$$
The two remaining endpoint inequalities are already contained in
$\mu\wedge\lambda>0$.

The following estimate supplies the radial bounds needed for the strongly adapted construction.

\begin{lemma}[Long-tube jet matching]
\label{lem:long-tube-jet-matching}
Let a strongly adapted contact form be fixed near $\Lambda$, and let a positive
boundary jet $(\lambda_0,\mu_0)$, with
$\mu_0\wedge\lambda_0>0$, be given near $\mathcal L^s$ such that
$(\mu_0,\lambda_0)$ straddles $T\mathcal L^s$ in the admissible order.  For a
sufficiently long entrance tube, the jet induced by this fixed form at
the inner end and $(\lambda_0,\mu_0)$ at the outer end are strongly
admissible for the product plug. The same holds on an exit tube after
reversing the flow.  The required length can be chosen uniformly in compact
families.
\end{lemma}

\begin{proof}
We prove the entrance statement.  Represent this prescribed form on a neighborhood $N(\Lambda)$ given by Theorem~\ref{adapted} and pull it back to the entrance
cross-section by the time-$T$ flow (for sufficiently large $T$). Along $\mathcal L^s$, choose a local adapted
coframe $(e^u,e^s)$ for the dual unstable and stable directions, oriented so
that the positive contact condition is
$\det(\lambda_0,\mu_0)>0$.  Write
$$
\lambda_0=p e^u+q e^s,
\qquad
\mu_0=r e^u+s e^s.
$$
The straddling condition and compactness allow compatible local signs to be
chosen so that
$$
q\leq-\delta<0<\delta\leq s,
$$
while $p$ and $r$ remain uniformly bounded.

Write the pulled-back jet of the fixed form as
$$
\lambda_T=U_Ta\,e^u+S_Tb\,e^s,
\qquad
\mu_T=U_Tc\,e^u+S_Td\,e^s.
$$
The invariant lines are the forward and backward limits of the two contact
cones.  Hence the defining forms near $\Lambda$ straddle both invariant directions and, after
choosing compatible local signs, we have
$$
a,c>0,
\qquad
b<0<d,
\qquad
ad-bc>0,
$$
uniformly.  Hyperbolicity gives constants $C,\chi>0$ such that
$$
U_T\geq C^{-1}e^{\chi T},
\qquad
0<S_T\leq Ce^{-\chi T}.
$$

Allow temporarily a constant rescaling $h>0$ of the inner jet.  The two
chord inequalities in Definition~\ref{strongly-admissible-jet} are
equivalent, with the above orientation, to
\begin{equation}
\label{eq:long-tube-radial-interval}
L_T<h<R_T,
\end{equation}
where
$$
L_T=
\frac{\det(\lambda_0,\mu_0)}
     {\det(\lambda_T,\mu_0)},
\qquad
R_T=
\frac{\det(\lambda_0,\mu_T)}
     {\det(\lambda_T,\mu_T)}.
$$
For large $T$, all denominators and numerators displayed here have the
required positive sign.  More precisely,
$$
\det(\lambda_T,\mu_0)
=U_Tas-S_Tbr,
$$
$$
\det(\lambda_0,\mu_T)
=pS_Td-qU_Tc,
$$
and
$$
\det(\lambda_T,\mu_T)
=U_TS_T(ad-bc).
$$
Uniform boundedness of all coefficients therefore gives
$$
0<L_T\leq C_1e^{-\chi T},
\qquad
R_T\geq C_1^{-1}e^{\chi T}
$$
for a constant $C_1>0$.  For sufficiently large $T$, the interval in \eqref{eq:long-tube-radial-interval} contains $h=1$, so the fixed form near $\Lambda$ requires no modification.

These estimates initially hold on $\mathcal L^s$. By openness and uniformity, they also hold on a sufficiently small compact neighborhood of $\mathcal L^s$. The resulting
finite flow tube, together with the prescribed neighborhood
of $\Lambda$, covers $W^s(\Lambda)\cap M$.
Compactness also makes the estimates uniform in auxiliary
parameters. Simultaneous reversal of the local coframe and jet
representatives leaves the coefficients and estimates unchanged, so they
hold without orientability assumptions. The exit statement is the time-reversed argument.
\end{proof}

\begin{theorem}[Strong hyperbolic boundary extension]
\label{hyps}
Let $(M,X)$ be a hyperbolic plug whose entrance and
exit laminations are prefoliations.  Every strongly admissible boundary
$X$-jet extends to a strongly adapted positive contact form on $M$.
Moreover, the extension may agree on smaller neighborhoods with strongly adapted forms prescribed near $\Lambda$ and near $\partial M$.  The
construction is parametric for compact parameter spaces.
\end{theorem}

\begin{proof}
Fix a strongly adapted form near $\Lambda$ given by Theorem~\ref{adapted}.  Apply
Lemma~\ref{lem:long-tube-jet-matching} near $\mathcal L^s$ and
$\mathcal L^u$.  The two endpoint jets on each resulting flow tube are
strongly admissible, so Theorem~\ref{pros} extends the form across the tube,
relative to both ends.

The complement is the same finite union of compact product pieces used in
the proof of Theorem~\ref{hypb}.  On each such piece, strong admissibility is
exactly the collection of endpoint inequalities imposed in the definition.
The relative form of Theorem~\ref{pros} therefore fills the complement.
Strictness of the determinant estimates and of all endpoint cones gives the
relative parametric statement.
\end{proof}

\begin{remark}
The inclination lemma controls the limiting projective directions and is enough
for Theorem~\ref{hypb}. Strong adaptation also requires control of the radial values and first derivatives. These bounds follow from the exponentially
wide interval in Equation~\eqref{eq:long-tube-radial-interval}.
\end{remark}

\begin{remark}[Homotopy classification of admissible boundary bi-foliations]
\label{homotopy-hyperbolic-bicontact}
Let $(M,X)$ be a hyperbolic plug as above, with no product components.
Fix the transverse sectors along the boundary laminations determined
by the hyperbolic core.
Let $I$ index the exceptional annular components of
$\partial_{\mathrm{in}}M\setminus\mathcal L^s$. The space
$\mathcal F_{\mathrm{ad}}^b(\partial M)$ of admissible boundary
bi-foliations respecting these choices satisfies
$$
\mathcal F_{\mathrm{ad}}^b(\partial M)
\simeq \coprod_{\mathbf n\in\mathbb Z^I}\{*\}.
$$
The integers are relative projective twisting numbers across the exceptional
annuli. A reference admissible bi-foliation fixes their origins.
Prescribing a sign system restricts the allowed components.
Every component is contractible. In particular, the space is
contractible in the filling MS case.

The key point is that the incoming positive foliation determines
the remaining boundary directions up to contractible choices.
Its directions near $\mathcal L^s$ vary in prescribed transverse
intervals and may therefore be fixed. After compact truncation,
the complementary strips are disks and contribute no relative
homotopy invariant, whereas each exceptional annulus $A$ contributes
$$
H^1(A,\partial A;\mathbb Z)\simeq\mathbb Z.
$$
Once these integers are fixed, angular lifts identify the remaining
choices with convex spaces of functions. The relative parametric
extension results give the same classification for supporting
bi-contact structures and strongly adapted positive contact forms,
with $X$ and the transverse sectors fixed. Product components
contribute the factors described in Remark~\ref{product-homotopy-type}.
We leave the details to the reader.
\end{remark}

Finally, the following extension lemma will be useful in our gluing constructions.

\begin{lemma}[One-sided strong admissibility]
\label{lem:one-sided-strong-admissibility}
Let $(M,X)$ be as in Theorem~\ref{hyps}. Suppose an ordered
entrance bi-foliation admits an admissible completion to $\partial M$.
Every defining coframe $(\mu^{\mathrm{in}},\lambda^{\mathrm{in}})$, understood modulo simultaneous sign,
extends to a strongly admissible boundary $X$-jet $(\lambda,\mu)$,
preserving any chosen admissible exit bi-foliation. The same holds
with entrance and exit interchanged.
\end{lemma}

\begin{proof}
Choose defining forms $(\mu^{\mathrm{out}},\lambda^{\mathrm{out}})$
for the exit completion. By Theorem~\ref{hypb}, the crossing map has
positive transition matrices in compatible local representatives of these coframes,
$$
\begin{pmatrix}\Gamma^*\mu^{\mathrm{out}}\\
\Gamma^*\lambda^{\mathrm{out}}\end{pmatrix}
=\begin{pmatrix}a&b\\c&d\end{pmatrix}
\begin{pmatrix}\mu^{\mathrm{in}}\\\lambda^{\mathrm{in}}\end{pmatrix},
\qquad a,b,c,d>0,\quad \Delta=ad-bc>0.
$$
Replacing $\lambda^{\mathrm{out}}$ by $h\lambda^{\mathrm{out}}$
makes the jet strongly admissible precisely when
$$
\frac1d<h\circ\Gamma<\frac a\Delta.
$$
This interval is nonempty because $bc>0$. Its endpoints are unchanged by
simultaneous sign reversal, so the choice of $h$ is global.

For sufficiently long transits, Lemma~\ref{lem:long-tube-jet-matching}
at both ends, with a strongly adapted form near $\Lambda$ between them, gives
$d>1$ and $a/\Delta>1$. These inequalities are preserved under
multiplication of positive transition matrices. Thus $h$ can equal
$1$ near $\mathcal L^u$ and be chosen smoothly in the displayed
interval on the remaining compact transit region. Positive rescaling
preserves the exit bi-foliation. For a prescribed exit coframe, reverse the flow direction and do the same.
\end{proof}


\subsubsection{Admissibility and suspended-gradient plugs}
\label{subsubsec:gradient-topological-admissibility}

For suspended-gradient plugs, admissibility imposes explicit relations
on boundary winding. We also describe two-saddle attachments that
simplify the Reeb annuli of a characteristic bi-foliation.

Let $\Sigma$ be a connected compact oriented surface with boundary and
$Y$ a gradient-like vector field transverse to the boundary, with only
index-one critical points and no saddle connections. Choose $Y$ linear
near its critical points. Set
$$
M=\mathbb S^1_\theta\times\Sigma,
\qquad X=\partial_\theta+Y.
$$
The maximal invariant set consists of the saddle periodic orbits over the
critical points. We distinguish oriented winding from projective winding
according to the convention above.

\begin{lemma}[Boundary winding balance]
\label{lem:gradient-winding-balance}
For a coorientable admissible boundary bi-foliation and each boundary
circle $C$, choose a tangent vector $\tau_C$ such that
$(Y,\tau_C)$ is positive. On $T_C=\mathbb S^1\times C$, use the frame
$(-\partial_\theta,\tau_C)$ and write $(q_C,\rho_C)$ for the oriented winding
of either member of an admissible boundary bi-foliation, along the
$\theta$-circle and the $\tau_C$-oriented circle, respectively. The two foliations $S_-$ and $S_+$ have the same winding class, all $q_C$ agree, and
\begin{equation}\label{eq:gradient-winding-balance}
\sum_{C\subset\partial_{\mathrm{out}}\Sigma}\rho_C
-\sum_{C\subset\partial_{\mathrm{in}}\Sigma}\rho_C
=-\chi(\Sigma).
\end{equation}
In particular, for genus $g$ and one entrance and one exit circle,
$$
\rho_{\mathrm{out}}=\rho_{\mathrm{in}}+2g.
$$
Transversality to a longitudinal lamination leaf also forces $q_C=0$.
Without coorientability, the same balance holds for the projective winding
classes, with right-hand side $-2\chi(\Sigma)$.
\end{lemma}
Note that in the above, equality of the longitudinal windings and
Equation~\eqref{eq:gradient-winding-balance} characterize extension as
ordered transverse line fields in the normal bundle
$N_X:=TM/\langle X\rangle$. They do not by themselves characterize
admissibility.

\begin{proof}
The normal bundle has the canonical identification (here $\pi_\Sigma:M\rightarrow \Sigma$ is the projection)
$$
TM/\langle X\rangle\longrightarrow\pi_\Sigma^*T\Sigma,
\qquad [a\partial_\theta+v]\longmapsto v-aY,
$$
which sends the boundary frame to $(Y,\tau_C)$. Transverse projective
directions have angular difference in an interval, so their winding
classes agree. A global framing of $T\Sigma$ identifies a normal line
field with a map $M\to\mathbb{RP}^1$. Its longitudinal degree is constant,
and its total signed boundary degree on each $\Sigma$-slice is zero.
The boundary frame contributes $2\chi(\Sigma)$ to projective winding.
Dividing by two for coorientable boundary bi-foliations gives
\eqref{eq:gradient-winding-balance}. Conversely, standard obstruction theory implies these are exactly the
conditions for the boundary class to extend in $H^1(M;\mathbb Z)$.
The interval-valued angular difference then extends the second direction.
Extend the admissible boundary bi-foliation by
Theorem~\ref{hypb} and apply the preceding calculation to the
resulting transverse line fields in $N_X$.
\end{proof}

\paragraph{The two-saddle transfer.}
Take $\Sigma$ with $g=1$ and one entrance and one exit boundary component. The entrance lamination consists of four longitudinal circles,
labelled $a,b,c,d$ in cyclic order, with stable branches paired as
$(a,c)$ and $(b,d)$. Consider the annuli
$$
I_1=(a,b),\quad I_2=(b,c),\quad I_3=(c,d),\quad I_4=(d,a).
$$
The crossing map sends the four annuli $\mathbb S^1\times I_j$ to the
exit in the following cyclic order (see Figure~\ref{fig:tw-saddle})
\begin{equation}\label{eq:gradient-four-annulus-transfer}
\ I_4,\ I_3,\ I_2,\ I_1.
\end{equation}
\begin{figure}[htbp]
    \centering
    \includegraphics[page=1,width=0.25\linewidth]
        {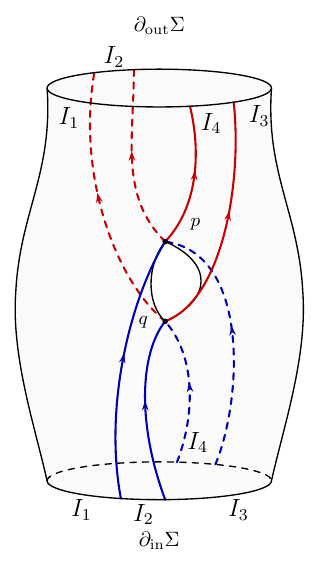}
    \caption{Two-saddle transfer}
    \label{fig:tw-saddle}
\end{figure}
Indeed, crossing a saddle joins the end of an interval to the beginning
at its paired endpoint. On each annulus the map has the form
$$
\Gamma(\theta,u)=(\theta+T(u),\gamma(u)),\qquad
D\Gamma(A\partial_\theta+B\partial_u)
=(A+BT')\partial_\theta+B\gamma'\partial_r.
$$
Here $T(u)$ is the transit time of the surface gradient flow $Y$,
and $\gamma(u)$ is the exit coordinate of the gradient trajectory
starting at entrance coordinate $u$.

By Lemma~\ref{lem:bby-invariant-extension} applied to $-X$, 
an entrance foliation transverse to the four circles can be pushed forward to the exit complement and extended across the four leaves of $\mathcal L^u$.
The required local information follows from the linear saddle model
$$
X=\partial_\theta-\alpha x\partial_x+\beta y\partial_y,
\qquad \alpha,\beta>0.
$$
Take $x=\pm1$ as incoming sections and $y=\pm1$ as outgoing sections.
On each outgoing section, use coordinates $(\theta,r)$ with $r=x$.
Thus $r$ is the signed coordinate transverse to the exit lamination,
which is locally the circle $\{r=0\}$.

For an incoming foliation transverse to the stable separatrices,
a direct computation of the crossing map gives a $C^1$ extension
of the transported foliation across $r=0$, with transverse linearization
\begin{equation}\label{eq:gradient-saddle-transfer}
\frac{dr}{d\theta}=-\alpha r.
\end{equation}
Each saddle therefore contributes two hyperbolic compact leaves of
negative sign, one on each outgoing section. When the four entrance
circles avoid the original compact leaves, those leaves persist under the crossing map and exactly four
compact leaves are added. 

To record the resulting Reeb arrangement, use coordinates $(\theta,u)$
with $(Y,\partial_u)$ positive. A compact leaf has holonomy sign
$\delta:=\operatorname{sgn}\lambda$ when its local equation is
$du/d\theta=\lambda u+O(u^2)$. Write $R_+$ and $R_-$ for Reeb annuli with net turning
$+\pi$ and $-\pi$, respectively, measured in the frame
$(\partial_\theta,\partial_u)$ as $u$ increases. Consecutive leaves with the same sign
bound a turning annulus $R_\delta$ and opposite signs give a parallel
annulus $P$. In this frame the transverse projective winding is
$$
d=\#R_+-\#R_-.
$$
For coorientable bi-foliations, the oriented winding is $w=d/2=-\rho$
(see Lemma~\ref{lem:gradient-winding-balance}).
For the upcoming boundary simplification in
Lemma~\ref{lem:gradient-boundary-simplification}, choose $a,b$
in one gap between the original compact leaves and $c,d$ in
another. Thus the annuli
$\mathbb S^1\times I_1$ and $\mathbb S^1\times I_3$ contain no old
compact leaves. In the exit ordering
\eqref{eq:gradient-four-annulus-transfer}, each transported annulus
is bounded by two new compact leaves. These form consecutive pairs
in the compact-leaf ordering, so each of the two selected gaps
receives two new leaves.

All four new leaves have $\lambda<0$ by
\eqref{eq:gradient-saddle-transfer}. Inserting two such leaves between
consecutive old compact leaves therefore changes the annular
decomposition as follows (see Figure~\ref{fig:reeb-positive-substitution}).
\begin{equation}\label{eq:gradient-reeb-substitution}
\begin{aligned}
R_+&\longmapsto P\,R_-\,P,
&
R_-&\longmapsto R_-\,R_-\,R_-,
\\
P&\longmapsto P\,R_-\,R_-
\quad\text{or}\quad R_-\,R_-\,P.
\end{aligned}
\end{equation}

\begin{figure}[htbp]
    \centering
    \includegraphics[page=1,width=0.5\linewidth]
        {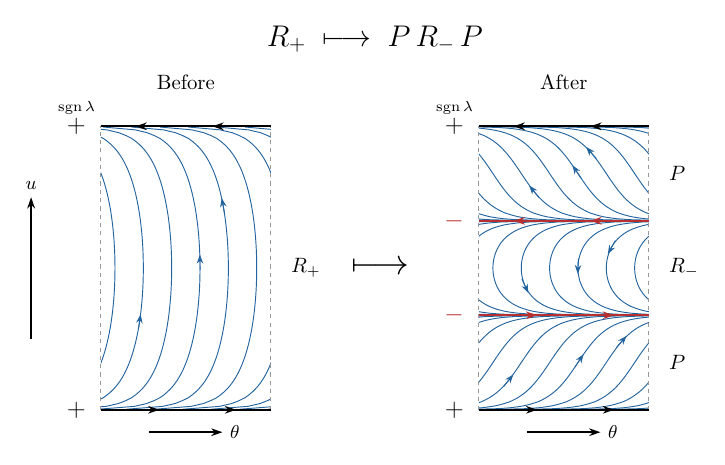}
    \caption{The substitution $R_+\mapsto P\,R_-\,P$.
    Read the annuli from bottom to top. Red compact leaves are new,
    and the signs indicate holonomy in the positive $\theta$-direction.}
    \label{fig:reeb-positive-substitution}
\end{figure}

The two cases for $P$ correspond to the boundary sign orders
$(+,-)$ and $(-,+)$, respectively.

We will use these plugs to modify boundary characteristic bi-foliations. The following lemma describes the required modification.

\begin{lemma}[Simultaneous boundary simplification]
\label{lem:gradient-boundary-simplification}
Let $(S_-,S_+)$ be a smooth transverse Morse-Smale
bi-foliation on a torus, and suppose $S_-$ has a positive (counterclockwise) turning Reeb annulus.
One can attach a two-saddle suspended-gradient bi-contact plug with this
exact entrance bi-foliation and an admissible exit bi-foliation for which
the number of positive turning annuli of $S_-$ decreases. Neither
foliation acquires a new positive turning annulus, and
$$
d_{\mathrm{out}}=d_{\mathrm{in}}-4,
\qquad
w_{\mathrm{out}}=w_{\mathrm{in}}-2\quad\text{in the coorientable case}.
$$
\end{lemma}

\begin{proof}
Write $S_\pm=\ker\eta_\pm$, with $\eta_-\wedge\eta_+>0$ in the
flow orientation, using local representatives. Straighten their compact leaves in their common
longitudinal slope. At the two boundary leaves of a positive turning
annulus of $S_-$, the signs of $\eta_+(\partial_\theta)$ are opposite.
Hyperbolicity supplies nearby transverse bands on which
$\eta_-(\partial_\theta)\eta_+(\partial_\theta)<0$, with opposite
ordered sign patterns. One band lies inside the chosen turning annulus.
Choose two nearby circles $a,b$ in the first band and $c,d$ in the second,
avoiding all compact leaves, and pair them as above. This gives the
entrance straddling condition and compatible local completions.

Push both foliations forward and complete them. Their defining directions
have angular separation in $(0,\pi)$ off $\mathcal L^u$, tending to
$\pi$ on $\mathcal L^u$. Rotate the negative direction slightly forward
and the positive direction slightly backward inside this ordered angular
interval. Compactness gives a uniform choice. A sufficiently small
$C^1$ smoothing preserves the Morse-Smale types and yields
$$
D\Gamma(S_-),\quad S_-^{\mathrm{out}},\quad
S_+^{\mathrm{out}},\quad D\Gamma(S_+)
$$
in the required order, together with exit straddling. Theorem~\ref{hypb}
therefore extends these boundary bi-foliations, preserving the
prescribed entrance bi-foliation.

The annuli $I_1$ and $I_3$ contain no old compact leaves. Hence
\eqref{eq:gradient-four-annulus-transfer} preserves their old cyclic
order and inserts two negative compact leaves in each selected gap.
The substitutions \eqref{eq:gradient-reeb-substitution} remove the chosen
$R_+$ and create no $R_+$. Each gap lowers the projective winding by two.
\end{proof}

Iterating Lemma~\ref{lem:gradient-boundary-simplification} keeps the
original entrance bi-foliation fixed and makes the chosen exit foliation
have only negative turning annuli, after finitely many two-saddle blocks.
For coorientable bi-foliations, odd oriented winding remains odd, so the final reduced Reeb
arrangement is nonempty. This is the boundary simplification used in the
cap construction below. It concerns one of the characteristic foliations.
Matching an
auxiliary annular double and its smooth holonomy is a separate step.
For iteration, the new bands may be chosen near compact characteristic
leaves and disjoint from the preceding exit lamination, avoiding saddle
connections between successive blocks. Each block adds two saddle periodic
orbits. Since the interfaces are crossed only in one direction, no periodic
orbit crosses an interface.

\subsubsection{Twisted saddle blocks and parity change}
\label{subsubsec:twisted-saddle-blocks}

The following twisted version of the pair-of-pants plug changes
coorientability at the boundary. Put
$$
F(x,y)=(x^2-1)^2+y^2,\qquad
P=\{\tfrac14\leq F\leq2\},\qquad V=\nabla F,
$$
and let $r(x,y)=(-x,-y)$. The involution exchanges the two entrance
circles of $P$, preserves its exit circle, and commutes with $V$.
Thus
$$
Q=(P\times\mathbb R)/((p,t)\sim(r(p),t+1)),
\qquad X=\partial_t+V
$$
is an oriented plug with one entrance torus and one exit torus.
Since $XF=|\nabla F|^2$, its maximal invariant set is the saddle
orbit over $(0,0)$, with multipliers $-e^{-4}$ and $-e^2$.
This is a model of the twisted saddle plug in \cite[Lemma~9.6]{bby}.

For an explicit supporting structure, set
$$
\begin{gathered}
u=dx-4x(x^2-1)\,dt,\qquad v=dy-2y\,dt,\qquad a=e^{-CF},\\
\alpha_-=v+au,\qquad\alpha_+=v-au.
\end{gathered}
$$
For large $C>0$, the rates of $au$ and $v$ are respectively
$12x^2-4-C|\nabla F|^2<0$ and $2$, giving the indicated contact
signs for $dt\wedge dx\wedge dy$. Both forms change sign under the
identification, so they descend as unoriented contact forms.

The suspension circles make $Q$ Seifert fibered over an annulus
with one exceptional fiber of order two. Let $h$ denote a regular
fiber, represented by a time-two suspension circle. On the entrance,
let $s$ be either inner circle, oriented counterclockwise. On the
exit, let $m$ be the counterclockwise outer circle and $\ell$ the
time-one longitude with a positive half-turn around that circle.
Use the positive section-fiber bases
$$
(s,h)_{\mathrm{in}},\qquad
(s,h)_{\mathrm{out}}=(\ell,2\ell-m).
$$
The entrance and exit laminations each consist of one regular fiber.

\begin{proposition}[Twisted-saddle boundary preparation]
\label{prop:twisted-saddle-preparation}
Every smooth ordered transverse bi-foliation on an oriented torus
can be prescribed exactly on the entrance of $Q$, after choosing
the entrance identification. The resulting supporting bi-contact
structure satisfies
\begin{equation}\label{eq:twisted-saddle-projective-transfer}
d_{\mathrm{in}}=(k,0),\qquad d_{\mathrm{out}}=(k+1,0)
\end{equation}
in the section-fiber bases, for some $k\in\mathbb Z$.
In particular, a noncoorientable entrance gives a coorientable exit.
Two successive blocks take coorientable even oriented winding to
coorientable odd oriented winding. Time reversal gives the
corresponding attachments at an entrance boundary.
\end{proposition}

\begin{proof}
Choose a smooth auxiliary foliation in the entrance straddling
sector $\lambda_-(v)\lambda_+(v)<0$ and a compact leaf $c$ after
a small perturbation within that sector. Such a leaf can always
be obtained on a torus. An orientable foliation admits a
Morse-Smale perturbation, and a nonorientable nonsingular torus
foliation already has a compact leaf. The curve $c$ is essential
and transverse to both prescribed foliations. Identify the stable
entrance curve of $Q$ with $c$.

The long-tube construction in the proof of Theorem~\ref{hypb}
realizes the prescribed bi-foliation near $c$, keeping the structure
near the saddle fixed. To realize the bi-foliation on the rest of the entrance,
choose an orientation-preserving projective automorphism $B$ of
$\mathcal N_X$, equal to the identity near $c$, whose pullback carries
the current bi-foliation to the prescribed one. The ordering and
transversality give such an automorphism. The choices of positive
rescalings form a contractible space. Extend $B$ by linear Poincar\'e transport along transit
orbits and by the identity elsewhere. Its saturated support is
compact and misses the invariant set. Determinant-one local lifts
satisfy $\mathcal L_XB=0$, so
$$
\mathcal L_X(\alpha_\pm\circ B)
=(\mathcal L_X\alpha_\pm)\circ B.
$$
Consequently transversality and both contact signs are preserved.
Lifts of $B$ may differ by a sign, which is allowed for unoriented
forms. This realizes the entire prescribed entrance bi-foliation.

Transversality to the two lamination curves gives $d(h)=0$ on both
boundaries. On a pair-of-pants slice, the projective winding balance
of Lemma~\ref{lem:gradient-winding-balance}, with counterclockwise
boundary circles, gives
$$
d_{\mathrm{out}}(m)=2d_{\mathrm{in}}(s)+2=2k+2.
$$
Since $h=2\ell-m$, this yields $d_{\mathrm{out}}(\ell)=k+1$.
Thus odd $k$ gives a coorientable exit.

For two blocks, prescribe the first exit bi-foliation exactly on the second
entrance and apply Theorem~\ref{thm:smooth-matched-gluing}.
When the first entrance is coorientable, write $k=2n$. The successive
projective degrees, in compatible section-fiber bases, are
$$
(2n,0)\longmapsto(2n+1,0)\longmapsto(2n+2,0),
$$
so the exterior oriented winding changes from $(n,0)$ to $(n+1,0)$.
The intermediate projective class is nonzero, so the auxiliary compact
leaf for the second entrance has fiber slope and misses the first exit
lamination, which is transverse to that auxiliary foliation. Thus the
maximal invariant set consists of exactly the two saddle orbits.
More generally, opposite
straddling sectors ensure transverse gluing, and no periodic orbit
crosses an interface.
\end{proof}

For the finite capping construction in Section~\ref{subsec:finite-PA-capping},
first use one twisted block if the boundary bi-foliation is
noncoorientable. If the resulting oriented winding is even, use
two more blocks. The terminal bi-foliation is then coorientable
with odd oriented winding, so the boundary simplification and
annular-double preparation apply. Thus at most three twisted
saddles remove the orientability and winding restrictions, while
preserving the initial boundary bi-foliation. The subsequent
preparation supplies the geometric conditions for capping.

\subsection{Non-hyperbolic bi-contact plugs}
In this subsection, we construct bi-contact plugs by thickening nonsingular surface
flows with strong normal expansion or contraction, then specialize
to Morse-Smale solid-torus models.

\begin{theorem}[Normal thickening]
\label{surface-extension-plug}
Let $\Sigma$ be a compact oriented surface and let $Y$ be a smooth
nowhere-vanishing vector field on $\Sigma$.  If $\partial\Sigma\neq
\varnothing$, assume $Y\pitchfork\partial\Sigma$.  On
$$
M_0=[-1,1]_\theta\times\Sigma
$$
set
$$
\overline Y=Y+K\theta\partial_\theta.
$$
For all sufficiently large $|K|$, the vector field $\overline Y$ is supported
by a bi-contact structure.  If $K>0$ and $Y$ points outward along
$\partial\Sigma$, corner rounding gives a repelling plug.  If $K<0$ and $Y$
points inward, it gives an attracting plug.  In either case,
$$
\operatorname{Inv}(M_0,\overline Y)
=\{0\}\times\operatorname{Inv}(\Sigma,Y).
$$
\end{theorem}

\begin{proof}
Choose $Z$ so that $(Y,Z)$ is a positive frame of $T\Sigma$, and let
$(\beta_Y,\beta_Z)$ be its dual coframe.  Write
$$
d\beta_Y=f_Y\beta_Y\wedge\beta_Z,
\qquad
d\beta_Z=f_Z\beta_Y\wedge\beta_Z.
$$
Define
$$
\alpha_+
=\beta_Z+\theta\beta_Y-\frac1Kd\theta,
\qquad
\widetilde\alpha_-
=\beta_Z-\theta\beta_Y+\frac1Kd\theta.
$$
Both annihilate $\overline Y$, their kernels are transverse, and, for
$$
\Omega=\beta_Y\wedge\beta_Z\wedge d\theta,
$$
a direct calculation gives
$$
\alpha_+\wedge d\alpha_+
=\left(1-\frac{f_Z+\theta f_Y}{K}\right)\Omega,
$$
$$
\widetilde\alpha_-\wedge d\widetilde\alpha_-
=\left(-1+\frac{f_Z-\theta f_Y}{K}\right)\Omega.
$$
Thus the first form is positive contact and the second is negative contact
for large $|K|$.  Setting
$$
\alpha_-=-\operatorname{sgn}(K)\widetilde\alpha_-
$$
gives
$$
\alpha_-\wedge\alpha_+
=\frac{2}{|K|}\iota_{\overline Y}\Omega.
$$

Finally, $\dot\theta=K\theta$.  Hence every point off the central surface
reaches a horizontal face in the expanding time direction, while the
central surface is invariant.  The asserted boundary behavior follows from
the signs of $K$ and of the normal component of $Y$ along
$\partial\Sigma$.
\end{proof}

In particular, every nonsingular flow on $\mathbb T^2$ occurs on the central
invariant torus of both a normally repelling and a normally attracting
bi-contact plug.  For large $|K|$, the normal rate dominates the derivative
of the flow tangent to that torus.

\begin{remark}
The same construction applies to the interval bundle of the orientation
line of a nonorientable surface $\Sigma$. Reversing the local orientation
changes $(Z,\theta)$ to $(-Z,-\theta)$. Thus $\theta\partial_\theta$
and $\overline Y$ are well-defined, both contact forms change sign,
and $\Omega$ is unchanged. The total space is oriented, and the
contact estimates descend.
\end{remark}

For the product model, the two horizontal characteristic bi-foliations are
coupled.  This is a feature of the formula, not an intrinsic restriction.
Let $\ell_Y=\langle Y\rangle$.  A bi-foliation $(S_-^+,S_+^+)$ on
$\{1\}\times\Sigma$ is called \emph{core-compatible} if the projective
directions occur in the order
$$
\ell_Y,\quad S_-^+,\quad S_+^+,
$$
and a bi-foliation $(S_-^-,S_+^-)$ on $\{-1\}\times\Sigma$ is core-compatible if
they occur in the order
$$
\ell_Y,\quad S_+^-,\quad S_-^-.
$$
The statement concerns the kernel foliations.  If coorientations are also
prescribed, they must agree with the coorientations induced by these lifted
orders.

\begin{proposition}[Independent prescription of horizontal bi-foliations]
\label{two-ended-boundary-realization}
Any two core-compatible smooth bi-foliations
$$
(S_-^+,S_+^+),
\qquad
(S_-^-,S_+^-)
$$
are realized independently as the characteristic bi-foliations on the two
horizontal faces of a bi-contact structure supporting
$\overline Y=Y+K\theta\partial_\theta$, after increasing $|K|$ if necessary.
The structure may be kept equal to the affine model near
$\{0\}\times\Sigma$.  The construction is relative and parametric while the
cyclic orders remain strict.
\end{proposition}

\begin{proof}
Fix a positive frame $(Y,Z)$ and let $\omega=\beta_Y\wedge\beta_Z$.  The two
cyclic-order conditions give smooth angle functions
$\varphi_\pm(\theta,x)$ such that
$$
\partial_\theta\varphi_->0,
\qquad
\partial_\theta\varphi_+<0,
$$
with
$$
\varphi_-(\theta,x)=\arctan\theta,
\qquad
\varphi_+(\theta,x)=-\arctan\theta
$$
near $\theta=0$, and with the prescribed endpoint kernels.  They may be
chosen so that $\varphi_-$ and $\varphi_+$ remain distinct modulo $\pi$ for $\theta\neq 0$.
Put
$$
\eta_\pm
=-\sin(\varphi_\pm)\beta_Y
+\cos(\varphi_\pm)\beta_Z,
\qquad
a_\pm
=\frac{\sin(\varphi_\pm)}{K\theta},
$$
where the second expression extends smoothly across $\theta=0$, and define
$$
\widetilde\alpha_-=\eta_-+a_-d\theta,
\qquad
\alpha_+=\eta_++a_+d\theta.
$$
Both forms annihilate $\overline Y$.  Moreover,
$$
\alpha_\pm\wedge d\alpha_\pm
=d\theta\wedge
\left(
a_\pm d_\Sigma\eta_\pm
-\eta_\pm\wedge\partial_\theta\eta_\pm
+\eta_\pm\wedge d_\Sigma a_\pm
\right),
$$
and
$$
\eta_\pm\wedge\partial_\theta\eta_\pm
=(\partial_\theta\varphi_\pm)\omega.
$$
The other two terms in $\alpha_\pm\wedge d\alpha_\pm$ are uniformly $O(|K|^{-1})\omega$.  Hence the two forms
have the required opposite contact signs for large $|K|$.  Their angular
separation makes their kernels transverse.  Multiplying
$\widetilde\alpha_-$ by $-\operatorname{sgn}(K)$ gives the prescribed
relative sign. All angle choices are made in convex intervals,
which proves the relative and parametric assertion.
\end{proof}

\paragraph{Side boundary and corner rounding.}
Let $B=[-1,1]\times\partial\Sigma$, let $T$ be tangent to
$\partial\Sigma$, and write a supporting form as
$$
\alpha=a\,d\theta+\eta_\theta.
$$
Then
\begin{equation}
\label{nh-side-characteristic}
\ker(\alpha|_{TB})
=\operatorname{span}
\{aT-\eta_\theta(T)\partial_\theta\}.
\end{equation}
Thus the side foliation is determined explicitly by the same angle functions $\varphi_\pm$.
At an edge $E=H\cap B$, the prescribed characteristic lines
$L_H\subset TH|_E$ and $L_B\subset TB|_E$ must satisfy the
pointwise compatibility condition
$$
\operatorname{span}\{\overline Y|_E,L_H\}
=
\operatorname{span}\{\overline Y|_E,L_B\}.
$$
Thus both facewise prescriptions determine the same plane containing
$\overline Y$. Where $L_H\ne L_B$, this is equivalent to
$\overline Y|_E\in\operatorname{span}\{L_H,L_B\}$.
The equality of planes also allows the valid case $L_H=L_B=TE$,
which the latter condition would incorrectly exclude.

\begin{remark}[Corner rounding]
\label{side-boundary-rounding}
Corner rounding of characteristic foliations is well understood in contact geometry \cite{hondanotes}. Near an edge, let $T$ be tangent to the edge and let $R$ be tangent to a
monotone rounding curve in the normal plane.  The characteristic line on the
rounded face is generated by
$$
\alpha(R)T-\alpha(T)R.
$$
It converges to the two facewise characteristic directions at the ends of
the rounding curve.  If the supported vector field points strictly inward,
or strictly outward, along both incident faces, the same holds after a
sufficiently small monotone rounding.  The two rounded characteristic
foliations remain transverse, since a common characteristic vector would lie
in the supported line field and would force it to be tangent to the rounded
boundary.
\end{remark}

\subsubsection{A dominated source or sink orbit}

Consider
$$
M_0=\mathbb S^1_\vartheta\times[-1,1]_x\times[-1,1]_y,
\qquad
Z=\partial_\vartheta+\nu x\partial_x+\mu y\partial_y,
$$
where either $0<\nu<\mu$ or $\nu<\mu<0$.  Put
$$
\alpha_u=dy-\mu y\,d\vartheta,
\qquad
\alpha_s=dx-\nu x\,d\vartheta,
$$
and
$$
\alpha_+=\alpha_u-\alpha_s,
\qquad
\alpha_-=\alpha_u+\alpha_s.
$$
Then
$$
\mathcal L_Z\alpha_u=\mu\alpha_u,
\qquad
\mathcal L_Z\alpha_s=\nu\alpha_s,
$$
and, for $\Omega=d\vartheta\wedge dx\wedge dy$,
$$
\alpha_-\wedge\alpha_+=2\iota_Z\Omega,
$$
$$
\alpha_+\wedge d\alpha_+=(\mu-\nu)\Omega,
\qquad
\alpha_-\wedge d\alpha_-=-(\mu-\nu)\Omega.
$$
After rounding, $M_0$ is a solid torus whose maximal invariant set is the
periodic orbit
$$
\gamma=\mathbb S^1\times\{(0,0)\}.
$$
Its normal multipliers are $e^{\nu t}$ and $e^{\mu t}$.  Thus $\gamma$ is a
dominated source in the first case and a dominated sink in the second.  It is
hyperbolic as a periodic orbit, but it is not saddle hyperbolic.

\begin{proposition}
\label{reeb-boundary-source-sink}
Each characteristic foliation induced by $\ker\alpha_-$ and
$\ker\alpha_+$ on the rounded boundary torus is the union of two Reeb
annuli.
\end{proposition}

\begin{proof}
Before rounding, on each face $x=\pm1$ the characteristic foliation of $\ker\alpha_+$ is generated by
$$
\partial_\vartheta+(\mu y\mp\nu)\partial_y,
$$
and the corresponding foliation for $\alpha_-$ is obtained by reversing the
sign of $\nu$.  In the repelling case, each has one hyperbolic compact leaf
on each of the two $x$-faces, while the $y$-faces are crossed monotonically.
The complement of the two compact leaves therefore consists of two Reeb annuli.
For the attracting case the same calculation is made on the $y$-faces.
The compact leaves are disjoint from the corners, and
Remark~\ref{side-boundary-rounding} shows that a small rounding preserves
this decomposition.
\end{proof}

\begin{remark}
The orientation-line thickening of a M\"obius band gives the twisted
source and sink models
$$
(\vartheta,x,y)\sim(\vartheta+1,-x,-y)
$$
on $\mathbb R_\vartheta\times[-1,1]^2$. The vector field descends,
and both defining forms change sign. After rounding, the plug is an
oriented solid torus whose periodic orbit has multipliers
$-e^\nu$ and $-e^\mu$.
On the coorientation cover, the deck involution exchanges the two
compact boundary leaves and their complementary Reeb annuli.
Thus each quotient characteristic foliation is noncoorientable
and has one compact leaf. Cutting along that leaf gives a Reeb annulus.
\end{remark}




\subsubsection{Boundary winding and solid-torus caps}
\label{subsec:solid-torus-boundary-parity}
\label{subsec:symmetric-cap-realization}

Fix a meridian-longitude parametrization of an oriented torus $T$.
We use the projective winding $d(S_-,S_+)$ and, for coorientable
bi-foliations, the oriented winding $w(S_-,S_+)=d(S_-,S_+)/2$.
Reversing local representatives and inserting parallel strips preserve
these classes.

\begin{proposition}[Meridional parity]
\label{prop:solid-torus-meridional-parity}
If $(S_-,S_+)$ is the characteristic bi-foliation of an oriented
bi-contact solid-torus plug with meridian $m$, then
$$
d(S_-,S_+)(m)\equiv2\pmod4.
$$
For coorientable boundary bi-foliations this reads $w(S_-,S_+)(m)\equiv1\pmod2$,
which is necessary and sufficient for extension of the frame
formed by an outward normal and the two oriented characteristic directions.
It alone does not establish a contact realization.
\end{proposition}

\begin{proof}
Over a meridional disk the contact directions admit orientations, and
together with the supported vector field give a frame. Along its boundary, the normal and coordinate
tangent frame contributes the nonzero class in
$\pi_1(SO(3))=\mathbb Z/2$, while the characteristic winding contributes
$d(m)/2\bmod2$. Hence $1+d(m)/2=0\bmod2$. For coorientable boundary
bi-foliations, the converse equality permits
extension across the disk, and $\pi_2(SO(3))=0$ permits extension across
the remaining ball.
\end{proof}

\begin{proposition}[Explicit caps with odd oriented winding]
\label{prop:odd-winding-cap-models}
For every odd integer $k$, there are coorientable repelling and attracting
bi-contact solid-torus plugs with oriented winding
$$
w(m)=k,\qquad w(\ell)=0.
$$
Each boundary characteristic foliation has $2|k|$ Reeb annuli. The
repelling model has $(|k|+1)/2$ sources and $(|k|-1)/2$ saddle periodic
orbits. The attracting model has the corresponding sinks and saddles.
\end{proposition}

\begin{proof}
Put $n=|k|$ and, on $A=\mathbb S^1_\vartheta\times[-1,1]_x$, take
$$
Y=\epsilon\cos\left(\frac{n\pi x}{2}\right)\partial_\vartheta
 +(-1)^{(n-1)/2}\sin\left(\frac{n\pi x}{2}\right)\partial_x,
\qquad \epsilon\in\{-1,1\}.
$$
This vector field is nonsingular and points outward along both boundary circles.
Its $n$ periodic orbits alternate between annular repellers and
attractors, beginning and ending with repellers. Apply
Theorem~\ref{surface-extension-plug} to
$\overline Y=Y+Kz\partial_z$ for large $K>0$, and round the corners.
The affine supporting forms give $n$ hyperbolic compact leaves on each
horizontal face and none on the side faces. Their longitudinal
orientations alternate, giving $2n$ Reeb annuli and meridional winding
of absolute value $n$. Choose $\epsilon$ to obtain the sign of $k$.
The coefficients are independent of $\vartheta$, so $w(\ell)=0$.
Normal expansion turns annular repellers into sources and annular
attractors into saddles. Time reversal gives the attracting model.
\end{proof}

For these annular models, a prescribed filling meridian $m$ must
admit a complementary longitude $\ell$ with $w(\ell)=0$, and
$w(m)$ must be odd. If the meridional slope is not prescribed,
every winding class $w\not\equiv0\pmod2$ admits a meridian-longitude
basis with these properties. Thus, with a suitable choice of
filling map, the explicit models realize a representative of every
oriented winding class that is nonzero modulo two.
To prescribe the boundary bi-foliation itself, we also specify an auxiliary annular foliation and its cyclic order with both characteristic directions.

\begin{definition}[Compatible annular double]
\label{def:cap-annular-double}
Let $Y$ be a nonsingular Morse-Smale vector field on an oriented annulus $A$, outward
along both boundary circles. On the symmetrically rounded solid torus
$V=[-1,1]_z\times A$, the product foliation
$$
\mathcal H_Y=\operatorname{span}\{Y,\partial_z\}
$$
induces an \emph{annular double} $G_Y$ on $T_Y=\partial V$, with its
boundary orientation. Its arrangement of Reeb components has reflection symmetry.
An ordered cooriented bi-foliation $(S_-,S_+)$ on $T$ admits a
\emph{compatible annular double} if an orientation-preserving
diffeomorphism $h:T_Y\to T$ carries $G_Y$ to a foliation $G$ satisfying
\begin{equation}\label{eq:cap-double-order}
\eta_-(g)\eta_+(g)>0,
\qquad \eta_-\wedge\eta_+>0,
\end{equation}
where $\eta_\pm$ define $S_\pm$ and $g$ generates $G$. This specifies
transversality and cyclic order. For a prescribed filling meridian,
require $h$ to preserve that class.
\end{definition}

\begin{proposition}[Realization of prescribed boundary bi-foliations]
\label{stable-nonhyperbolic-cap-realization}
Every smooth bi-foliation admitting a compatible annular double is
realized exactly, under $h$, on the boundary of a repelling bi-contact
solid-torus plug supporting $\overline Y=Y+Kz\partial_z$, for large $K$.
The boundary adjustment keeps $\overline Y$ fixed and is supported away
from its maximal invariant set. Its periodic orbits are precisely those
of $Y$, becoming sources or saddles as above. Time reversal gives the
corresponding attracting statement with the reversed boundary
flow-orientation convention.
\end{proposition}

\begin{proof}
Pull back by $h$. In a positive frame $(Y,Z)$ with dual coframe
$(\beta_Y,\beta_Z)$, the initial supporting forms are
$$
\alpha_\pm^0=\pm\beta_Z+c,\qquad c=z\beta_Y-K^{-1}dz.
$$
Put $\gamma=\beta_Z|_{TT_Y}$ and $\delta=c|_{TT_Y}$.
The order condition, after a common sign choice matching the
coorientations, normalizes the target forms to
$$
\eta_\pm=b_\pm\gamma+\delta,\qquad b_-<b_+.
$$
For the negative-time flow $\varphi_{-L}$ of $Y$, write
$$
\varphi_{-L}^*\beta_Z=A_L\beta_Z,\qquad
\varphi_{-L}^*\beta_Y=\beta_Y+B_L\beta_Z.
$$
Uniform derivative bounds give $A_L\geq C^{-1}e^{-bL}>0$ and
$|B_L|\leq Ce^{bL}$. Choose $K>2b$. If $\Phi_t$ is the flow of
$\overline Y$, the inner characteristic forms transported from
$\Phi_{-L}(T_Y)$ to $T_Y$ are, up to positive scaling,
$$
e^{KL}(\Phi_{-L}^*\alpha_\pm^0)|_{TT_Y}
=a_\pm^L\gamma+\delta,
\qquad a_\pm^L=zB_L\pm e^{KL}A_L.
$$
Thus $a_-^L\to-\infty$ and $a_+^L\to+\infty$ uniformly. For large
$L$, both the original coefficients $(-1,1)$ and the target
$(b_-,b_+)$ lie in $\{(a,b):a_-^L<a<b<a_+^L\}$. Their linear
interpolation is therefore admissible. Apply
Lemma~\ref{lem:nonconvex-bicontact-realization} with identity transport
on this flow collar, preserving the original bi-contact structure near the inner boundary of the collar. The vector field stays fixed,
and the collar misses the maximal invariant set. Finally,
$\dot z=Kz$ places every periodic orbit in the central annulus.
\end{proof}

\begin{remark}
The annular-double hypothesis concerns the whole bi-foliation, which
need not have reflection symmetry. For realization up to
homeomorphism, one may apply the proposition to a smooth representative
of its simultaneous homeomorphism class. Sufficiently small elementary
stabilizations are allowed by
Corollary~\ref{cor:small-nonconvex-realization}. They add parallel strips
without changing winding or the supported dynamics. The realization
uses a sufficiently long flow collar, whose width is not prescribed.
Odd oriented winding alone does not supply a compatible annular double, and
realization of every prescribed bi-foliation with odd oriented winding remains outside this
statement.
\end{remark}


\section{Deformation and gluing}
\label{s4}

In this section, we establish relative realization in a flow collar, coordinated divided deformation,
relative smoothing, and gluing of diffeomorphic characteristic bi-foliations.
We use the unoriented forms of Section~\ref{s2.1}, also for characteristic
foliations. Forms used together have common transition signs, fixed along
parameter intervals. The characteristic bi-foliations remain ordered.
Boundary preparations
are supported in collars disjoint from the maximal invariant set, where the product-plug results apply.
All parameter intervals are compact. Statements relative to a subset of the plug preserve the specified structures on a neighborhood of that subset.

\subsection{Flow collars and relative realization}

Let $U=\Sigma\times[0,1]_t$ be an exit collar, with inner
boundary $t=0$, actual boundary $t=1$, and original supported
vector field $X$ a positive multiple of $\partial_t$.  Orient $\Sigma$ so that $dt\wedge\omega>0$.
We have analyzed the bi-contact geometry of product plugs in Section~\ref{s3.1}. Recall that in a positive local coframe $(e_1,e_2)$ on $\Sigma$, a supporting contact form can be written as
\begin{equation}\label{eq:s4-contact-angle}
\alpha=R(\cos\theta\,e_1+\sin\theta\,e_2)
\quad\Longrightarrow\quad
\alpha\wedge d\alpha
=-R^2\partial_t\theta\,dt\wedge e_1\wedge e_2,
\end{equation}
where $\theta$ is defined modulo $\pi$ and the sign of $-\partial_t\theta$ is the contact sign of $\alpha$.
The following lemma preserves the orbit structure during boundary deformations.

\begin{lemma}[Straightening a boundary deformation]
\label{lem:collar-straightening}
Suppose a smooth family of bi-contact structures on $U$ agrees with the
original bi-contact structure near $t=0$, and its supported oriented line field is positively
transverse to every $\Sigma\times\{t\}$.  Write its normalized vector field as
$Z_r=\partial_t+V_r$.  There is a slice-preserving isotopy $\Phi_r$,
equal to the identity near the inner boundary, such that
$(\Phi_r)_*\partial_t=Z_r$.  Its restriction $h_r$ to $t=1$ is the
collar transit map.  If $Z_0=\partial_t$, then $\Phi_0=\operatorname{id}$.

Extending $\Phi_r$ by the identity gives an oriented smooth orbit
equivalence of the original and deformed plug flows.  Choosing the
deformed vector field to be $(\Phi_r)_*X$ gives an actual conjugacy.
The maximal invariant set is unchanged and is fixed pointwise.
\end{lemma}

\begin{proof}
Solve the nonautonomous equation
$$
\partial_t\theta_{r,t}(x)
=V_r(\theta_{r,t}(x),t),\qquad \theta_{r,0}(x)=x,
$$
and set $\Phi_r(x,t)=(\theta_{r,t}(x),t)$.  Compactness gives the
solution on the whole interval, and uniqueness makes every
$\theta_{r,t}$ a diffeomorphism.  The stated push-forward identity
follows by differentiation.  Since $V_r=0$ near the inner boundary,
$\Phi_r$ extends by the identity.  No full orbit in the plug can meet
an exit collar whose vector field points strictly towards its outer boundary.
The collar therefore misses both maximal invariant sets, and the
extended conjugacy is the identity on them.
\end{proof}

For an entrance collar, apply the construction in the backward flow
direction.

Now, identify the collar slices with $\Sigma$ using the flow of $X$.
Choose compatible local defining forms $e_-,e_+$ for the original characteristic
bi-foliation on $\Sigma\times\{0\}$, with $e_-\wedge e_+>0$, and set
$$
\mathcal C_x=\{a e_-(x)+b e_+(x):a,b>0\}.
$$
Let $(S_-^r,S_+^r)$, $r\in[0,1]$, be a smooth path of ordered
outer boundary bi-foliations, starting at the original
bi-foliation, with compatible local defining forms $\eta_\pm^r$ satisfying
$\eta_-^r\wedge\eta_+^r>0$. The following relative realization
statement follows from Theorem~\ref{prob}. The entrance version
follows by reversing the collar direction and the corresponding
flow-orientation conventions. The local cone condition below is unchanged
when $e_-,e_+,\eta_-^r,\eta_+^r$ change sign together.

\begin{lemma}[Relative realization with fixed supported vector field]
\label{lem:nonconvex-bicontact-realization}
The family $(S_-^r,S_+^r)$ on $\Sigma\times\{1\}$ is induced by
a smooth family of bi-contact structures on
$U=\Sigma\times[0,1]$ supporting the fixed vector field $X$, starting at
the original structure and agreeing with it near
$\Sigma\times\{0\}$, if and only if
$$
\eta_-^r(x),\ \eta_+^r(x)\in\mathcal C_x
\qquad\text{for every }x\in\Sigma,\ r\in[0,1].
$$
The preserved neighborhood of the inner boundary may be chosen
sufficiently small. Where the prescribed bi-foliations are independent
of $r$, the realizing structures may also be chosen independent of $r$.
If one characteristic foliation is unchanged throughout the family,
its original contact structure may be preserved on all of $U$.
The construction is parametric for compact parameter spaces and
requires no convexity or Morse-Smale assumption.
\end{lemma}

\begin{proof}
The cone condition, together with $\eta_-^r\wedge\eta_+^r>0$, is
precisely the product admissibility condition of
Definition~\ref{admissible-bifoliation}, relative to the inner
characteristic bi-foliation. Necessity therefore follows from
the product-plug analysis in Section~\ref{s3}.

Conversely, strict admissibility and compactness allow a sufficiently
small inner collar to remain fixed while the boundary bi-foliation on the
complementary collar remains admissible for every $r$. Apply the
relative-parametric form of Theorem~\ref{prob} to that complementary
collar, relative to $r=0$ and to regions where the prescribed
bi-foliations are independent of $r$. Its angular interpolation also allows an unchanged contact
structure to be preserved throughout the collar.
\end{proof}

The cone condition is $C^0$-open, so
Lemma~\ref{lem:nonconvex-bicontact-realization} gives the following corollary.

\begin{corollary}
\label{cor:small-nonconvex-realization}
There is a $C^0$ neighborhood $\mathcal N_U$ of the original outer
characteristic bi-foliation such that every smooth path in $\mathcal N_U$
starting at that bi-foliation is realizable as in
Lemma~\ref{lem:nonconvex-bicontact-realization}.
\end{corollary}

\subsection{Divided foliations and coordinated deformation}

We use Giroux realization to deform bi-foliations while preserving appropriate dividing curves. See \cite{hondanotes,massot} for convex surface theory.

We use the following form of the dividedness condition.  On an oriented
surface $(\Sigma,\omega)$, use local defining forms $\beta_i$ and
local functions $f_i$ with the same transition signs,
$\beta_j=\varepsilon_{ij}\beta_i$ and $f_j=\varepsilon_{ij}f_i$.
All dividing functions below are understood in this sense.
Writing $\beta,f$ locally, the nonsingular foliation $\ker\beta$
is divided with contact sign $\epsilon\in\{-1,1\}$ if
\begin{equation}\label{eq:Qr-sign}
\epsilon Q>0,\qquad Q=f\,d_\Sigma\beta+\beta\wedge d_\Sigma f.
\end{equation}
The two-form $Q$ is global, and $f\,ds+\beta$ defines an invariant
unoriented contact form of sign $\epsilon$ on the product oriented by
$ds\wedge\omega$. The dividing set $\Gamma=f^{-1}(0)$ is global.
Its signed regions are understood on the coorientation double cover,
where the deck involution exchanges their signs.
If $\beta=\iota_Y\omega$ locally, this expression becomes
\begin{equation}\label{eq:dividing-divergence}
Q=(f\operatorname{div}_\omega Y-Y(f))\omega.
\end{equation}

The cutoffs and convex choices in the usual Giroux flexibility proof
can be made equivariantly on this double cover. The Gray vector field is
invariant because its defining equation is unchanged when the contact
form changes sign. Hence realization descends.

\begin{lemma}[Parametric divided forms]
\label{lem:parametric-divided-forms}
A smooth family $S_q$ of foliations divided by the same
dividing set with its local signs admits smooth defining forms $\beta_q$ satisfying
$$
\epsilon(f\,d_\Sigma\beta_q+\beta_q\wedge d_\Sigma f)>0
$$
for a fixed defining function $f$. The initial form may be prescribed
if it defines $S_0$ with the compatible local signs and satisfies this inequality.

Alternatively, for prescribed smooth defining forms $\eta_q$, if each
admits a dividing function $f$ satisfying
$$
f\,d_\Sigma\eta_q+\eta_q\wedge d_\Sigma f<0,
$$
these functions can be chosen smoothly as $f_q$, with a uniform
strict inequality.
\end{lemma}

\begin{proof}
For the first assertion, choose a preliminary family of defining forms with fixed transition signs. Locally
in the parameter, a positive multiplier making the inequality strict
continues to do so.  A partition of unity in the parameter combines
these choices because their kernels and local signs coincide and the
expression in~\eqref{eq:Qr-sign} is linear in $\beta$.  The initial
choice can be fixed in this construction.  A compatible invariant model
using another defining function with the same zero set and local signs can
first be rescaled to use $f$.
For the second assertion, the expression is linear in $f$.
The same partition of unity argument applies.  Compactness supplies the uniform
strict margin.
\end{proof}

For a bi-contact structure, control of the two contact signs must be
accompanied by transversality of the supported line field to the collar
slices. The next construction supplies both.

\begin{theorem}[Coordinated divided deformation]
\label{prop:coordinated-divided-deformation}
Let $\eta_r,\zeta_r$ be smooth families of local $1$-forms with common transition signs on a closed
oriented surface $(\Sigma,\omega)$, with $\eta_r\wedge\zeta_r=\omega$,
and suppose smooth local functions $f_r$ with the same transition signs satisfy
$$
Q_r:=f_r\,d_\Sigma\eta_r+\eta_r\wedge d_\Sigma f_r
\leq-c\omega,\qquad c>0.
$$
For $b(s)>0$ sufficiently small with $b'(s)>0$, the bi-contact forms
$$
\alpha_-^0=\eta_0+b(s)\zeta_0,\qquad
\alpha_+^0=-\eta_0+b(s)\zeta_0
$$
on $U=\Sigma\times[0,1]_s$ admit a deformation fixed near $s=0$
and inducing the boundary forms
$$
\eta_r+b(1)\zeta_r,\qquad -\eta_r+b(1)\zeta_r
$$
at $s=1$. Throughout, the supported line field remains positively
transverse to the collar slices. If every $\ker\eta_r$ contains
a fixed lamination $L$, the transit maps preserve $L$.

The same construction holds for $C^1$ forms and functions depending
smoothly on $r$ in the $C^1$ topology.
\end{theorem}

\begin{proof}
On $\Sigma\times[0,1]_s$, choose $\chi=0$ near $0$, $\chi=1$ near $1$,
and $a\geq0$ zero near both ends and equal to one on
$\operatorname{supp}\chi'$.  Put $q=r\chi(s)$, let $C$ be a constant and set
\begin{equation}\label{eq:s4-coordinated-pair}
\begin{split}
\alpha_-^r&=Cr a(s)f_q\,ds+\eta_q+b(s)\zeta_q,\\
\alpha_+^r&=Cr a(s)f_q\,ds-\eta_q+b(s)\zeta_q.
\end{split}
\end{equation}
These forms agree with $\alpha_\pm^0$ near $s=0$ and induce the prescribed boundary forms at $s=1$. Put
$R_q=f_qd_\Sigma\zeta_q+\zeta_q\wedge d_\Sigma f_q$ and choose $b$
small enough that $|bR_q|\leq c\omega/2$.
Writing $\beta_-=\eta_q+b\zeta_q$, direct calculation gives
\begin{equation}\label{eq:s4-coordinated-contact}
\begin{split}
\alpha_-^r\wedge d\alpha_-^r
=ds\wedge\bigl[&Cr a(Q_q+bR_q)-b'\omega\\
&-r\chi'\beta_-\wedge(\dot\eta_q+b\dot\zeta_q)\bigr],
\end{split}
\end{equation}
where the dots denote differentiation in $q$.
Replacing $\eta_q$ by $-\eta_q$ gives the formula for $\alpha_+^r$.
The terms $Q_q$ and $-b'\omega$ change sign.
The last wedge product in each formula has absolute value at most $K\omega$
uniformly for some $K>0$. Taking
$C>2K\|\chi'\|_\infty/c$ gives the negative contact sign for
$\alpha_-^r$ and the positive sign for $\alpha_+^r$ on the whole collar.
In addition,
$$
(\alpha_-^r|_{T\Sigma_s})\wedge(\alpha_+^r|_{T\Sigma_s})
=2b(s)\omega>0.
$$
If $\eta_q(Y_q)=0$ and $\zeta_q(Y_q)=1$, the normalized supported vector field is
\begin{equation}\label{eq:s4-coordinated-field}
Z_r=\partial_s-\frac{Cr a(s)f_q}{b(s)}Y_q.
\end{equation}
The local vector field $Y_q$ changes sign with $f_q$, so $Z_r$ is
global. The contact estimates involve global two-forms and are unchanged
by the local sign choices.
Along $L$, the tangential component of $Z_r$ is tangent to its leaves.
Uniqueness for the time-dependent $C^1$ differential equation implies that the transit maps preserve these leaves and hence $L$. Only first spatial derivatives
enter the calculation, so the stated $C^1$-regularity version follows
as well.
\end{proof}

This theorem starts from the displayed collar model, which is chosen as
part of the initial supporting structure in extension arguments. If the
supported vector field $X$ and the original bi-contact structure near the inner
boundary are to be preserved, the prescribed
outer foliations must satisfy the cone condition of
Lemma~\ref{lem:nonconvex-bicontact-realization}.
Rescaling $(f_r,\eta_r)$ by a positive function and $\zeta_r$ by its
reciprocal preserves the dividedness inequality and allows the sizes of
the opposite endpoint tilts to be adjusted.

\subsection{Relative smoothing and gluing}

There are two smoothing problems. A family of $C^1$ supporting bi-contact forms can be smoothed because the contact and transversality conditions are $C^1$-open.
At an identified boundary, the contact plane fields are initially only
piecewise smooth, and smoothing must also preserve the monotonicity
in~\eqref{eq:s4-contact-angle}.

\begin{lemma}[Relative smoothing of a supporting family]
\label{lem:relative-pair-smoothing}
Let $(\alpha_-^r,\alpha_+^r)$, $r\in[0,1]$, be a family of $C^1$
bi-contact forms on a compact manifold $M$, continuous in the $C^1$
topology. Assume that on each chosen deformation collar
$U=\Sigma\times[0,1]_s$, the supported line field is positively
transverse to every surface $\Sigma\times\{s\}$ throughout the family.
 The family admits smooth
approximation, uniformly in $C^1(M)$, through bi-contact forms with these properties.

The approximation may agree with the original family on prescribed neighborhoods where it is jointly smooth in the spatial variables and the parameter. For a fixed smooth nonvanishing vector field $X$, it may also preserve $\alpha_\pm^r(X)=0$ on prescribed regions or at a prescribed parameter. At a prescribed
parameter it may also preserve the identification of the ordered
boundary characteristic bi-foliations by a fixed smooth diffeomorphism,
including any prescribed compatible local comparison signs.
These prescriptions are assumed compatible.

On each collar, the supported vector fields normalized by $ds(Z_r)=1$
and their transit maps converge in $C^1$, uniformly in $r$.
\end{lemma}

\begin{proof}
Apply relative parametric $C^1$ approximation on the common coorientation
double cover and take the odd part of the approximating forms.
Work in $\operatorname{Ann}(X)$ wherever $X$ must remain supported
and preserve the prescribed smooth portions of the family.

To preserve a boundary identification, write the boundary restrictions,
for each of the two contact forms, as $\beta_{\mathrm{out}}=\varepsilon u\phi^*\beta_{\mathrm{in}}$,
in compatible local representatives, with $u>0$ and prescribed locally
constant comparison signs $\varepsilon\in\{-1,1\}$. The function $u$ is global.
Approximate $\beta_{\mathrm{in}}$ and $u$ smoothly, keep $\varepsilon$ fixed, and set
$$
\widetilde\beta_{\mathrm{out}}
=\varepsilon\widetilde u\,\phi^*\widetilde\beta_{\mathrm{in}}.
$$
Correct the preliminary approximation to have these boundary
restrictions using collar and parameter cutoffs, within
$\operatorname{Ann}(X)$ where required.

The contact signs and transversality conditions are $C^1$-open,
so sufficiently close approximation preserves them uniformly.
The normalized supported vector fields depend continuously in $C^1$ on
the defining forms. Continuous dependence of solutions of the flow and variational
equations then gives $C^1$ convergence of the transit maps.
\end{proof}

Next, diffeomorphic characteristic bi-foliations allow the bi-contact structures to be glued after a $C^0$-small deformation near the identified boundary, preserving the supported vector field.

\begin{theorem}[Smooth gluing of bi-contact structures]
\label{thm:smooth-matched-gluing}
Let $\phi:\Sigma_{\mathrm{out}}\to\Sigma_{\mathrm{in}}$ be an
orientation-compatible diffeomorphism between unions of exit and
entrance boundary components, possibly of the same bi-contact plug.
Suppose $\phi$ identifies their ordered characteristic bi-foliations.

After arbitrarily $C^0$-small deformations supported in arbitrarily
thin boundary collars, the bi-contact structures glue smoothly,
preserving the supported vector fields and the characteristic
bi-foliations on the identified boundary. The quotient is given
the smooth structure determined by the flow collars.

Any chosen coorientations extend when they are consistent
under $\phi$. The construction is relative away from the identified
boundary and parametric for compact smooth families.
\end{theorem}

\begin{proof}
Use flow-time coordinates $\Sigma\times(-\delta,\delta)_t$ across the
identified boundary, with the prescribed vector fields equal to
$\partial_t$. Orientation compatibility makes the two local comparison
signs equal. Choose compatible local representatives agreeing at $t=0$
after positive rescaling. Their transition signs then agree across the boundary.
The contact plane fields are continuous across $t=0$ and smooth on each closed half-collar.
Choose a positive local $t$-invariant coframe using defining forms for the
identified characteristic bi-foliation. For each contact plane field,
write the angle increment from the identified boundary as
$u(x,t)=\theta(x,t)-\theta(x,0)$. These increments have real lifts
that agree on overlaps under simultaneous sign reversal.

For the decreasing angle, $g=-\partial_tu>0$ on both sides.
Smooth $g$ by positive averaging in a very small neighborhood of $0$.
Correct the two integral errors by bumps supported in the interiors
of the two half-collars, so that
$$
\int_{-\delta}^0\widetilde g\,dt=\int_{-\delta}^0g\,dt,
\qquad
\int_0^\delta\widetilde g\,dt=\int_0^\delta g\,dt.
$$
The errors tend uniformly to zero with the averaging width.  Since
$g$ has a positive lower bound, these corrections preserve positivity.
Set $\widetilde u(x,t)=-\int_0^t\widetilde g(x,s)\,ds$.
It is smooth, has the original value at $0$, and agrees with $u$ near
both collar ends.  Treat the increasing angle with the opposite sign.

The two changes in angle are arbitrarily $C^0$-small, so the contact
planes remain transverse.  Interpolating the angular speeds gives the
deformation through contact structures on each half-collar.  All forms
annihilate $\partial_t$, and the bi-foliations on the gluing locus stay fixed.  The
choices and estimates are uniform in compact auxiliary parameters.
\end{proof}

\subsection{Stabilization and gluing after boundary preparation}
The $C^0$-openness of admissibility allows two parallel
strips to be inserted near a compact characteristic leaf.
We call this an \emph{elementary stabilization}.

\begin{definition}[Elementary stabilization]
\label{def:elementary-stabilization}
In an annular neighborhood of a compact leaf of one member of a
Morse-Smale characteristic bi-foliation, consider the model
$$
S_{-\sigma}=\operatorname{span}\{\partial_y\},\qquad
S_\sigma=\operatorname{span}
\{\partial_\vartheta+f(y)\partial_y\},
$$
where $f$ has one simple zero.  An elementary stabilization replaces
$f$, relative to the annulus boundary, by a $C^0$-close function with
two additional simple zeros, one attracting and one repelling.
It adds two parallel strips and leaves the reduced Reeb arrangement
and the homotopy class of the projective directions unchanged.
\end{definition}

\begin{lemma}[Realization of stabilization]
\label{lem:stabilization-realization}
Any finite pattern of elementary stabilizations can be carried out by an arbitrarily $C^0$-small path of characteristic bi-foliations and
realized in a fixed flow collar with the supported vector field unchanged.
Each insertion passes through one saddle-node parameter. The endpoint
foliations are Morse-Smale.
\end{lemma}

\begin{proof}
Both foliations are coorientable on an annular neighborhood of the compact leaf.
The annular model may first be arranged by a small modification.
Straighten the other foliation and put the compact leaf at $y=0$.
The first direction is $\partial_\vartheta+F(\vartheta,y)\partial_y$.
A periodic rescaling of $y$ makes
$F(\vartheta,y)=\lambda y+O(y^2)$, with $\lambda\ne0$ by hyperbolic
holonomy.  On a sufficiently small annulus, interpolate to
$F=\lambda y$ near its center.  This creates no additional compact
leaf and preserves transversality to $\partial_y$.

Now modify $f$ by a bump on one side of its zero, creating two simple
zeros through one nondegenerate double zero.  The coefficient of
$\partial_\vartheta$ remains one.  Shrinking the annular supports
makes the whole path arbitrarily $C^0$-small.  For finitely many
insertions, keep the combined path in $\mathcal N_U$ and apply
Corollary~\ref{cor:small-nonconvex-realization}, relative to the
complement of those supports.
\end{proof}

Note that these homotopies are covered by relative flow-collar realization, while the deformation is not through divided characteristic foliations.

We now combine the boundary preparation and gluing results.

\begin{theorem}[Gluing after boundary preparation]
\label{thm:stable-gluing}
Let $\phi:\Sigma_{\mathrm{out}}\to\Sigma_{\mathrm{in}}$ be an
orientation-compatible smooth diffeomorphism between unions of exit
and entrance boundary components of finitely many bi-contact plugs.
Suppose finite compositions of the boundary deformations above,
satisfying their respective hypotheses, produce smooth characteristic
bi-foliations that $\phi$ identifies with their ordering.

Then the prepared bi-contact structures glue to a smooth bi-contact
structure on the quotient, after smoothing in arbitrarily thin
boundary collars. Before gluing, each preparation preserves the
original plug flow up to smooth oriented orbit equivalence and fixes
its maximal invariant set pointwise. Preparations using only
realizations with a fixed vector field, including small stabilizations, keep the
supported vector field unchanged.

The boundary components identified by $\phi$ may belong to the same plug.
\end{theorem}

\begin{proof}
Perform the preparations in disjoint boundary collars. For finite-regularity
models, apply Lemma~\ref{lem:relative-pair-smoothing}, preserving the original
smooth bi-contact structure outside the deformation collars, near their
inner boundaries, and at the initial parameter. Where the prescribed forms
are only $C^1$, preserve the condition of supporting $X$ instead, and make
the approximation parameter-independent near the inner boundaries.
Preserve the final identification of the smooth characteristic
bi-foliations under $\phi$.
Lemma~\ref{lem:collar-straightening} gives the dynamical assertion for
each preparation, and Theorem~\ref{thm:smooth-matched-gluing} glues the
prepared bi-contact structures.
\end{proof}

A homeomorphism of the initial characteristic bi-foliations does not itself supply
the relative deformation or the final gluing diffeomorphism.  In the
filling MS construction of Section~\ref{s5}, the common completion and
coordinated divided deformation supply the required deformation and identification.  For the prescribed
caps of Section~\ref{s3}, they are supplied by the compatible annular
double and its flow-collar realization.

\begin{remark}[Absorbing the preparation into the gluing map]
\label{rem:realization-input-self-gluing}
Let $h_{\mathrm{in},r}$ and $h_{\mathrm{out},r}$ be the boundary
restrictions of the collar straightening maps, from the original plug
to the prepared plug. Using the entrance-to-exit convention of
Section~\ref{s5}, returning to the original plug vector field replaces a map
$\psi:\Sigma_{\mathrm{in}}\to\Sigma_{\mathrm{out}}$ by
$$
\widehat\psi_r
=h_{\mathrm{out},r}^{-1}\circ\psi\circ h_{\mathrm{in},r}.
$$
If the two boundary isotopies preserve $\mathcal L^s$ and
$\mathcal L^u$ setwise, respectively, then $h_{\mathrm{out},r}$
carries the two laminations for the modified gluing map to those for the original map.
Thus strong transversality persists.  After smooth approximation the
same conclusion follows from ordinary transversality along the path
and its isotopy stability discussed in Section~\ref{s5}.
The assertion that no invariant set is created concerns the boundary
preparation of each plug.  Self-gluing can create return orbits and
requires the separate dynamical estimates in Section~\ref{s5}.
\end{remark}
\section{Building Anosov flows}
\label{s5}

We combine the boundary preparation of Section~\ref{s4} with the
strong extension theorem of Section~\ref{s3} to establish hyperbolicity
after gluing hyperbolic plugs. We use the unoriented forms of
Section~\ref{s2}, with boundary coframes and jets taken modulo
simultaneous sign as in Section~\ref{s3}. Hyperbolic plugs have one stable and
one unstable normal direction. Attracting and repelling basic pieces
are allowed. The self-gluing theorem concerns filling MS laminations
and keeps the plug vector field fixed.

\subsection{Acyclic gluing}

We begin with a contact geometric version of
\cite[Proposition~1.1]{bby}, concerning the gluing of two distinct
hyperbolic plugs.

\begin{theorem}[Acyclic bi-contact gluing]
\label{thm:acyclic-hyperbolic-gluing}
Let $(M_i,X_i)$, $i=1,2$, satisfy the hypotheses of
Theorem~\ref{hyps} and carry admissible ordered boundary
bi-foliations. Suppose an orientation-compatible diffeomorphism
$$
\phi:\Sigma_1\longrightarrow\Sigma_2,
\qquad
\Sigma_1\subset\partial_{\mathrm{out}}M_1,
\quad
\Sigma_2\subset\partial_{\mathrm{in}}M_2
$$
between unions of boundary components identifies these characteristic
bi-foliations with their ordering.
Then the induced vector field $X$ on $M_1\sqcup_\phi M_2$ admits a
smooth strongly adapted positive contact form $\alpha_+$ whose supporting
bi-contact structure induces the prescribed boundary bi-foliations. Any defining
coframe $(\mu,\lambda)$ for the bi-foliation on the identified surface
$\Sigma$ may be prescribed as
$$
j_X^1\alpha_+|_\Sigma=(\lambda,\mu).
$$
In particular, the quotient is a hyperbolic bi-contact plug.
\end{theorem}

\begin{proof}
Choose compatible local boundary defining forms $(\mu_i,\lambda_i)$ agreeing with the
prescribed coframe under $\phi$. Apply
Lemma~\ref{lem:one-sided-strong-admissibility} with the exit coframe
fixed on $M_1$ and the entrance coframe fixed on $M_2$.
In flow coordinates across $\Sigma$, prescribe the strongly adapted collar model
$$
\alpha_+(t)=\cosh(t)\lambda+\sinh(t)\mu,
\qquad X=\partial_t.
$$
This formula respects the common transition signs of $\lambda,\mu$.
Theorem~\ref{hyps} extends the resulting jets relative to this model.
The extensions glue smoothly, and Theorem~\ref{adapted} gives
hyperbolicity.
\end{proof}

\begin{remark}
In the above theorem, admissibility already implies
$\phi_*(\mathcal L_1^u)\pitchfork\mathcal L_2^s$ as their tangent
directions satisfy $\mu(v_u)\lambda(v_u)>0$ and
$\mu(v_s)\lambda(v_s)<0$, respectively. Hence the hypotheses of \cite[Proposition~1.1]{bby} hold.
Hyperbolicity could also be proved by gluing the admissible
characteristic bi-foliations and applying \cite{arh}, since acyclic
gluing creates no new periodic orbits or invariant irrational tori.
\end{remark}

\subsection{Self-gluing with fixed plug flow}

Self-gluing of a possibly disconnected hyperbolic plug may produce
new periodic orbits and invariant sets. We give a contact geometric
proof of hyperbolicity in the filling MS case, answering
\cite[Question~1.4]{bby}.

Write $\Lambda=\operatorname{Inv}(M,X)$ and let
$$
\Gamma:\partial_{\mathrm{in}}M\setminus\mathcal L^s
\longrightarrow\partial_{\mathrm{out}}M\setminus\mathcal L^u
$$
be the crossing map. We use the entrance-to-exit convention for a
gluing map $\psi:\Sigma_{\mathrm{in}}\to\Sigma_{\mathrm{out}}$.
Strong transversality means strong transversality of
$\psi_*(\mathcal L^s\cap\Sigma_{\mathrm{in}})$ and
$\mathcal L^u\cap\Sigma_{\mathrm{out}}$, as defined in
Section~\ref{s2.3}. Our main self-gluing theorem is the following.

\begin{theorem}[Contact geometric self-gluing]
\label{thm:MS-filling-self-gluing}
Let $(M,X)$ be a smooth hyperbolic plug
with filling MS boundary laminations, and let
$\psi_0:\Sigma_{\mathrm{in}}\to\Sigma_{\mathrm{out}}$
be an orientation-compatible strongly transverse gluing map.

There exist a smooth strongly adapted positive contact form
$\alpha_+$ for $X$ and a smooth isotopy $(\psi_r)_{r\in[0,1]}$
through strongly transverse maps such that the extension of $\psi_1$ to boundary neighborhoods
using the flow of $X$ preserves $\alpha_+$.

Thus $\alpha_+$ descends to a strongly adapted contact form on
$M_{\psi_1}$. Consequently, $M_{\psi_1}$ with the induced flow is a hyperbolic
bi-contact plug, and the flow is Anosov when $M_{\psi_1}$ is closed.
\end{theorem}

\begin{remark}
If $\psi_0$ identifies a smooth admissible boundary
bi-foliation $(S_-,S_+)$, the construction
may be chosen so that every $\psi_r$ identifies $S_+$ and
$
\ker(\alpha_+|_{T\partial M})=S_+.
$
Thus the positive characteristic foliation can be prescribed,
while its defining form and derivative along $X$ may change.
\end{remark}

The same approach gives a gluing result for partially hyperbolic plugs. 

\begin{corollary}[Partially hyperbolic self-gluing]
\label{cor:partially-hyperbolic-self-gluing}
Let $(M,X)$ be a smooth, structurally generic partially hyperbolic bi-contact plug.
Suppose that removing disjoint repelling neighborhoods $U_i$
of its source periodic orbits gives a hyperbolic plug
$$
N=M\setminus\bigcup_i\operatorname{int}U_i
$$
with filling MS boundary laminations.
Let $\psi_0:\Sigma_{\mathrm{in}}\to\Sigma_{\mathrm{out}}$
be an orientation-compatible strongly transverse gluing map
between unions of original boundary components.

If the construction of
Theorem~\ref{thm:MS-filling-self-gluing} on $N$ is performed
relative to supporting bi-contact structures agreeing with
those on the $U_i$ near their attaching boundaries, then
$\psi_0$ is isotopic through strongly transverse maps to
$\psi_1$ such that the induced flow on $M_{\psi_1}$ is
partially hyperbolic. In particular, complete gluing produces
a closed partially hyperbolic flow.
\end{corollary}

\begin{proof}
Apply Theorem~\ref{thm:MS-filling-self-gluing} to $N$, preserving
the supporting bi-contact structures near the attaching boundaries
as assumed, and reattach the source neighborhoods $U_i$ along
these unchanged collars. The resulting flow is induced by $X$,
and the supporting bi-contact structures glue smoothly, giving
a dominated normal splitting $E\oplus F$.

No orbit can re-enter a source neighborhood after leaving it.
Hence the nonwandering set of the resulting maximal invariant
set is contained in the hyperbolic maximal invariant set of
$N_{\psi_1}$ and the finitely many original source periodic orbits.
On the hyperbolic set, $F$ is the unstable normal line. On each
source orbit, both normal directions expand. Thus $F$ is uniformly
expanding on each of these compact invariant sets.
Lemma~\ref{lem:partially-hyperbolic-plug} gives partial hyperbolicity.
\end{proof}

\begin{corollary}[Anosov completion by reflection]
\label{cor:Anosov-doubling}
Let $(M,X)$ be a smooth hyperbolic plug
with filling MS boundary laminations. Its oriented
double $D(M)$ admits an Anosov vector field $Y$ and a smooth
embedding $\iota:M\hookrightarrow D(M)$ such that
$$
\iota_*X=Y|_{\iota(M)}.
$$
\end{corollary}

\begin{proof}
Consider $(M,X)\sqcup(-M,-X)$. Reversing the flow exchanges
the entrance and exit laminations, so the canonical boundary
identifications identify equal laminations.

On each boundary component, choose an MS completion of the
common lamination. The unoriented divided models of Section~\ref{s4}
give a small isotopic displacement of this completion transverse
to itself. Smoothing preserves transversality.
The filling condition then makes the displaced laminations
strongly transverse. Thus the canonical identifications can be
replaced by isotopic smooth strongly transverse gluing maps.

Apply Theorem~\ref{thm:MS-filling-self-gluing} to the disconnected
plug. The resulting closed quotient carries an Anosov vector
field induced by $X$ and $-X$. Since all boundary identifications
are isotopic to the canonical ones, the quotient is diffeomorphic
to $D(M)$, and the inclusion of the first copy gives the required
embedding.
\end{proof}

Corollary~\ref{cor:Anosov-doubling} strengthens the
embedding conclusion of \cite[Theorem~1.12(1)]{bby} by preserving
the original vector field and provides an alternative realization of the embedding by
doubling.

The proof of Theorem~\ref{thm:MS-filling-self-gluing} first prepares boundary characteristic bi-foliations using
ordinary admissibility. Regular invariant completions then provide
bounded holonomy distortion, allowing the crossing map to be stretched
until the boundary forms define a strongly admissible $X$-jet.
Theorem~\ref{hyps} extends this jet relative to strongly adapted
collar models that agree under gluing. Hyperbolicity follows from Theorem~\ref{adapted}.
The relative assertion preserves the positive characteristic
foliation while allowing its defining form and first derivative
to change.

\begin{proposition}[Gluing strongly admissible boundary jets]
\label{prop:strongly-adapted-gluing}
Let $(M,X)$ satisfy the hypotheses of Theorem~\ref{hyps}, and let
$\psi:\Sigma_{\mathrm{in}}\to\Sigma_{\mathrm{out}}$ be an
orientation-compatible boundary identification. Suppose a smooth
strongly admissible boundary $X$-jet $j=(\lambda,\mu)$ satisfies
\begin{equation}\label{eq:s5-jet-identification}
\psi^*j^{\mathrm{out}}=j^{\mathrm{in}}.
\end{equation}
Then the jet is realized by a smooth strongly adapted positive
contact form on $M_\psi$,
preserving any prescribed compatible strongly adapted form near
$\Lambda$. In particular, the induced flow has a hyperbolic maximal
invariant set.
\end{proposition}

\begin{proof}
Realize the identified jet by the same
strongly adapted collar model on both sides, as in the acyclic
construction. Theorem~\ref{hyps} extends $j$ relative to this model
and the prescribed strongly adapted form near $\Lambda$. The extensions therefore glue smoothly,
and Theorem~\ref{adapted} gives hyperbolicity.
\end{proof}

\subsection{Regularity and holonomy distortion}
We first establish regularity estimates for the invariant foliations and control the holonomy distortion of their boundary restrictions. These estimates will be used to prepare strongly admissible boundary jets for gluing.

We prove regularity of the invariant plane fields given by
Lemma~\ref{lem:bby-invariant-extension}. Throughout, regularity of
a foliation refers to its tangent distribution.

\begin{lemma}[Regularity of invariant completions]
\label{lem:regularity-invariant-completions}
Let $(M,X)$ be a smooth three-dimensional hyperbolic plug whose
boundary laminations are prefoliations. There is $\vartheta_*>0$
such that, for $0<\vartheta<\vartheta_*$, the invariant completion
of any $C^{1+\vartheta}$ entrance foliation transverse to
$\mathcal L^s$ is $C^{1+\vartheta}$ on $M$. It contains
$W^u(\Lambda)$ and agrees with the flow transport away from this
set. In particular, its exit restriction completes $\mathcal L^u$.
For a smooth entrance foliation, it is smooth away from $W^u(\Lambda)$.
The analogous statement holds with entrance and exit interchanged
and time reversed.
\end{lemma}

\begin{proof}
If $\Lambda=\varnothing$, finite-time transport suffices. Otherwise,
let $E^u$ be the continuous normal completion given by
Lemma~\ref{lem:bby-invariant-extension}, and choose a transverse
stable completion $E^s$. By
Lemma~\ref{lem:one-sided-strong-admissibility} and Theorem~\ref{hyps},
choose a strongly adapted form $\alpha_+$ whose stable and unstable
projective sectors contain $E^s,E^u$, respectively, preserving a fixed strongly
adapted form near $\Lambda$.

Write $\varphi^t$ for the flow of $X$ and $P_X^t$ for the linear
Poincar\'e flow on $\mathcal N_X=TM/\mathbb RX$, so that
$P_X^t[v]=[D\varphi^t v]$.
Using the line norms $\|v\|_\alpha=|\alpha_+(v)|$, write
$$
a_t(x)=\|P_X^t|_{E_x^s}\|_\alpha,\qquad
b_t(x)=\|P_X^t|_{E_x^u}\|_\alpha
$$
for the operator norms on the invariant normal lines.
Strong adaptation and compactness give constants $R\geq\kappa>0$,
depending only on the chosen form near $\Lambda$, such that
$$
e^{-Rt}\leq a_t(x)\leq e^{-\kappa t},
\qquad e^{\kappa t}\leq b_t(x)\leq e^{Rt}
\qquad(x\in\Lambda,\ t\geq0).
$$
Fix $0<\vartheta<\vartheta_*:=\kappa/R$.

Near $E^u$ in $\mathbb P(\mathcal N_X)$, use the smooth coordinate
$$
z_x(\ell)=\frac{(\mathcal L_X\alpha_+)(v)}{\alpha_+(v)},
\qquad 0\ne v\in\ell.
$$
The norms and this ratio are unchanged by sign reversal and hence
are globally defined.
Write the induced projective flow as
$(x,z)\mapsto(\varphi^t x,G_t(x,z))$ and set $u(x)=z_x(E_x^u)$.
In adapted unit frames, the slope map near $E^u$ is
$\zeta\mapsto(a_t/b_t)\zeta$. The inverse tangent flow has norm
$a_t^{-1}$ when the hyperbolic splitting is orthogonal and $X$ has
unit length. Thus, in a fixed smooth metric with $\|X\|=1$,
$$
|\partial_zG_t(x,u(x))|\,
\|D\varphi^{-t}_{\varphi^t x}\|^{1+\vartheta}
\leq C\frac{a_t(x)^{-\vartheta}}{b_t(x)}
\leq Ce^{-(\kappa-\vartheta R)t},\qquad x\in\Lambda.
$$
The adapted slope and tangent norm give $C=1$. The smooth coordinate
$z$ and smooth metric introduce a uniform comparison factor.

Choose a fixed time step for which this product is less than $1/2$
on $\Lambda$, and hence on a neighborhood of $\Lambda$.
The compact sets
$\{x\in M:\varphi^{[-A,A]}x\subset M\}$ decrease to $\Lambda$
as $A\to\infty$.
Consequently, every sufficiently long orbit segment in $M$ lies in
this neighborhood except for two end portions of uniformly bounded
duration. The chain rule and submultiplicativity of operator norms
allow us to multiply the fixed-time estimates along the middle
portion. Uniform bounds for the end portions and the remaining
fraction of a time step then give uniform decay of the displayed
product along long orbit segments. By continuity, we can therefore
choose $T>0$ and $\epsilon>0$ such that
$$
|\partial_zG_T(x,z)|\,
\|D\varphi^{-T}_{\varphi^T x}\|^{1+\vartheta}\leq\frac14
$$
whenever $|z-u(x)|\leq\epsilon$ and $\varphi^{[0,T]}x\subset M$.
Since $\|D\varphi^{-T}\|\geq1$, the maps $z\mapsto G_T(x,z)$ contract distances by a factor at most
$1/4$ and preserve these intervals about the invariant section $u$.

Choose $v_0\in C^{1+\vartheta}(M)$ uniformly $\epsilon$-close to $u$
and equal to the transported entrance section on a fixed entrance
collar of flow duration less than $T$. Define
$$
v_{n+1}(x)=G_T(\varphi^{-T}x,v_n(\varphi^{-T}x))
$$
when $\varphi^{[-T,0]}x\subset M$, and use entrance transport
otherwise. Near the boundary between these regions, backward images
from the first region lie in the fixed entrance collar, so both
formulas give entrance transport.
Thus every $v_n$ is $C^{1+\vartheta}$ and $v_n\to u$ uniformly.

For $y=\varphi^{-T}x$, differentiation gives
$$
Dv_{n+1}(x)=
\bigl[D_1G_T(y,v_n(y))
+\partial_zG_T(y,v_n(y))Dv_n(y)\bigr]D\varphi^{-T}_x,
$$
where $D_1$ denotes differentiation with respect to the base point.
The coefficient of $Dv_n$ has norm at most $1/4$, so
$$
\|Dv_{n+1}\|_\infty\leq B_1+\frac14\|Dv_n\|_\infty
$$
for a constant $B_1$ independent of $n$. Hence the first derivatives
are uniformly bounded. In H\"older differences, the additional
factor $\|D\varphi^{-T}\|^\vartheta$ is already included in the
time-$T$ estimate. In sufficiently small fixed normal coordinate charts,
its strict bound $1/4$ therefore gives
$$
[Dv_{n+1}]_\vartheta\leq B_2+\frac12[Dv_n]_\vartheta.
$$
Here $[\cdot]_\vartheta$ denotes the H\"older seminorm. The increase
to $1/2$ allows for comparing nearby points and coordinate norms.
The remaining terms and pairs of separated points are controlled by
smoothness of the finite-time maps, the uniform first derivative
bound, and the fixed entrance foliation. Thus $B_2$ is independent of $n$.

The iterates are uniformly bounded in $C^{1+\vartheta}$.
Arzel\`a-Ascoli gives a subsequence converging in $C^1$ to $u$,
and the H\"older bound passes to $Du$.
For a smooth entrance foliation, finite-time transport gives smoothness
away from $W^u(\Lambda)$. Reverse time for the exit assertion.
\end{proof}

We also need the following control of holonomy distortion, with
$0<\vartheta<1$ chosen as in
Lemma~\ref{lem:regularity-invariant-completions}.

\begin{lemma}[Uniform distortion of full strip holonomy]
\label{lem:MS-filling-distortion}
Let $F$ be a nonsingular $C^{1+\vartheta}$ Morse-Smale foliation
on a compact surface, with hyperbolic compact-leaf holonomy and a
filling MS sublamination $\mathcal L$. Let $G$ be a fixed smooth
transverse foliation. There is $C_{\mathrm{hol}}\geq1$ such that every $F$-holonomy
$H:I\to J$ between full $G$-segments crossing the same complementary
strip satisfies, in arclength coordinates for a fixed smooth
Riemannian metric,
$$
C_{\mathrm{hol}}^{-1}\leq\frac{|DH(x)|}{|DH(y)|}\leq C_{\mathrm{hol}}
\qquad(x,y\in I).
$$
Here a full segment is a component of a $G$-leaf in the strip,
with its endpoints on the boundary leaves.
\end{lemma}

\begin{proof}
Orientations are chosen only on individual strips and compact-leaf annuli.
Near a compact leaf, let $h$ denote its holonomy map in the contracting
direction. On a sufficiently small transversal,
$$
0<m\leq h'\leq q<1,\qquad
|\log h'(u)-\log h'(v)|\leq K|u-v|^\vartheta.
$$
Consequently, for every $n\geq1$,
$$
\left|\log\frac{(h^n)'(u)}{(h^n)'(v)}\right|=\left| \log\frac{\prod_{j=0}^{n-1}h'(h^j(u))}{\prod_{j=0}^{n-1}h'(h^j(v))}\right|
\leq K\sum_{j=0}^{n-1}|h^j(u)-h^j(v)|^\vartheta
\leq\frac{K|u-v|^\vartheta}{1-q^\vartheta}.
$$
Inverse branches have the same distortion bound on their image
intervals. Partial turns and changes between finitely many
transverse charts contribute uniformly bounded factors.

Choose disjoint attracting or repelling annular neighborhoods of
the finitely many compact leaves. Outside them, $F$-leaf segments
have uniformly bounded length. Otherwise, a limiting half-orbit
would avoid all these neighborhoods, contrary to the Morse-Smale
limit-set description. Holonomy along this bounded middle region
has derivative and inverse derivative uniformly bounded.
Every full strip holonomy is a composition of at most two
holonomy maps contained in the chosen compact-leaf neighborhoods
and a holonomy map along leaf segments of uniformly bounded length. The preceding estimates
give a uniform distortion bound, independently of the number of
turns or the strip. This is the decomposition of
\cite[Lemma~6.8]{bby}, at the present regularity.
\end{proof}

For later use, observe that
$$
\frac{\ell(J)}{\ell(I)}
=\frac{1}{\ell(I)}\int_I |DH(y)|\,dy.
$$
Bounded distortion therefore compares each value of $|DH|$
with this average. For an increasing diffeomorphism $h:I\to I$,
the chain rule also gives
\begin{equation}\label{eq:s5-holonomy-length}
\begin{aligned}
C_{\mathrm{hol}}^{-1}\frac{\ell(J)}{\ell(I)}
&\leq |DH(x)|
\leq C_{\mathrm{hol}}\frac{\ell(J)}{\ell(I)},\\
D(H\circ h\circ H^{-1})(Hx)
&=DH(hx)\,Dh(x)\,DH^{-1}(Hx)
=\frac{DH(hx)}{DH(x)}Dh(x).
\end{aligned}
\end{equation}
The first estimate gives a uniform positive lower bound for
$|DH^{-1}|$ when $I$ is fixed and $\ell(J)$ is uniformly bounded.
We use this to transfer crossing-map expansion across a strip.
The second shows that extending $h$ by holonomy changes its
derivative by at most the factor $C_{\mathrm{hol}}$, which
controls the subsequent interpolation to the identity near
the strip ends.

\subsection{Admissible boundary preparation}

In this subsection, we explain the boundary preparation required for self-gluing. For our boundary laminations, which admit $C^1$ extensions of their tangent line fields, an isotopy of the gluing map through transverse
configurations preserves the two laminations up to ambient isotopy,
and hence preserves their complementary quadrilaterals.
Thus strong transversality persists along such isotopies and is
$C^1$-open.



\begin{lemma}[Boundary preparation in the filling case]
\label{lem:MS-filling-matched-preparation}
Under the hypotheses of Theorem~\ref{thm:MS-filling-self-gluing},
the map $\psi$ is isotopic through strongly transverse maps to
$\psi_{\mathrm{prep}}$ for which $X$ admits a smooth supporting
bi-contact structure whose ordered boundary characteristic bi-foliations are identified by $\psi_{\mathrm{prep}}$. Its defining forms
may be scaled so that, on the identified components,
$$
\psi_{\mathrm{prep}}^*\beta_\pm^{\mathrm{out}}
=\beta_\pm^{\mathrm{in}}.
$$
\end{lemma}

\begin{proof}
By \cite[Lemma~3.26]{bby} and
Lemma~\ref{lem:regularity-invariant-completions}, choose a completion
$F$ of $\psi^{-1}_*\mathcal L^u$ transverse to $\mathcal L^s$,
with the regularity required by that lemma. This is achieved by using
a regular completion near the lamination and smoothing elsewhere.
Transversality is preserved since the tangent directions agree on
the lamination. On the remaining entrance components, choose $F$
smooth and transverse to $\mathcal L^s$. Its forward transport
completes to a $C^1$ exit foliation $F_u$.

On $\Sigma_{\mathrm{out}}$, both $F_u$ and $G=\psi_*F$ complete
$\mathcal L^u$, so they are topologically conjugate MS foliations
by \cite[Lemma~3.19]{bby}. To construct a $C^1$ path, choose the small
real lift of their relative projective angle near $\mathcal L^u$, equal to zero
on $\mathcal L^u$. The remaining compact portions of the complementary
strips lie in finitely many disjoint disks with boundaries in that
neighborhood. On each disk this angle has a real lift,
which can be chosen to agree with the prescribed lift along its
connected boundary. Interpolating this lift gives completions $H^r$
from $F_u$ to $G$, depending smoothly on $r$ in the $C^1$ topology.
Every $H^r$ is MS by the same lemma. Its compact leaves remain
hyperbolic. Transverse accumulation of $\mathcal L^u$ on these
leaves forces $C^1$ tangent line fields agreeing on $\mathcal L^u$ to
have the same first derivatives there, and hence the same linear
holonomy.

Choose small opposite exit tilts of $F_u$ and sufficiently smaller
opposite entrance tilts of $F$ to obtain admissible bi-foliations on the boundary.
Indeed, $F$ is transverse to $\mathcal L^s$, and the exit
bi-foliation straddles $\mathcal L^u$. Near $\mathcal L^s$, uniform inclination
places the transported entrance directions between the exit
directions. On the compact remainder, the required order follows
from $\Gamma_*F=F_u$ by decreasing the entrance tilts.

The Morse-Smale dividedness criterion \cite{hondanotes} applies on
the coorientation cover. For an odd defining form, taking the odd part
of a dividing function averages the original inequality with its deck
transform, preserving strictness.
The second assertion of Lemma~\ref{lem:parametric-divided-forms}
allows us to apply the $C^1$ version of
Theorem~\ref{prop:coordinated-divided-deformation} to $H^r$.
Use the collar model from that theorem and extend its
bi-contact structure over $M$ by Theorem~\ref{hypb}.
The coordinated deformation, with the endpoint tilts adjusted
by rescaling, makes the exit bi-foliation equal to the
$\psi$-image of the entrance bi-foliation.
Its transit maps preserve $\mathcal L^u$, since every $H^r$
contains this lamination.

Apply the relative smoothing of
Lemma~\ref{lem:relative-pair-smoothing}, preserving the endpoint
boundary identification, and absorb the deformation into the
gluing map as in
Remark~\ref{rem:realization-input-self-gluing}.
This gives a smooth supporting bi-contact structure for $X$
and a strongly transverse isotopy ending at
$\psi_{\mathrm{prep}}$, which identifies its boundary
characteristic bi-foliations.
Positive rescaling of the defining forms gives the stated
normalization. Orientation compatibility makes the two local
comparison signs agree.
\end{proof}

\subsection{Stretching and strongly admissible preparation}

After boundary preparation, denote the gluing map again by $\psi$
and set $(\mu,\lambda)=(\beta_-,\beta_+)|_{\partial M}$.
The construction below also applies to any smooth supporting
bi-contact forms with the same boundary identification.
We use $j=(\lambda,\mu)$ as a formal boundary $X$-jet, to be
realized by the strong extension theorem.
The unstable projective sector is $\{\mu(v)\lambda(v)>0\}$, and the
stable projective sector is $\{\mu(v)\lambda(v)<0\}$. The contact signs
give forward invariance of the first and backward invariance of
the second because the off-diagonal coefficients of
$\mathcal L_X(\beta_-,\beta_+)$ in this coframe are positive.
Local representatives are transported consistently along the flow
and the isotopies below.

To preserve the positive characteristic foliation, choose
$$
G^u_{\mathrm{in}}=\ker\lambda^{\mathrm{in}},
\qquad G^s_{\mathrm{out}}=\ker\lambda^{\mathrm{out}},
$$
locally oriented by $\mu^{\mathrm{in}}>0$ and $\mu^{\mathrm{out}}<0$,
respectively. Admissibility makes these directions transverse to
$\mathcal L^s$ and $\mathcal L^u$, respectively. Complete
$G^s_{\mathrm{out}}$ backward and $G^u_{\mathrm{in}}$ forward to obtain
$G^s_{\mathrm{in}}$ and $G^u_{\mathrm{out}}$, respectively.
Strict cone invariance places these completions inside the stable
entrance and unstable exit sectors. Thus we obtain transverse boundary bi-foliations
$$
(G^s_{\mathrm{in}},G^u_{\mathrm{in}}),\qquad
(G^s_{\mathrm{out}},G^u_{\mathrm{out}}),\qquad
\Gamma_*G^\sigma_{\mathrm{in}}=G^\sigma_{\mathrm{out}}.
$$
By Lemma~\ref{lem:regularity-invariant-completions}, the completed
foliations are $C^{1+\vartheta}$ and smooth off the respective
laminations, which they contain. They are MS with hyperbolic
compact-leaf holonomy, so Lemma~\ref{lem:MS-filling-distortion}
applies to both boundary bi-foliations. We call these four foliations
a scattering frame.
Fix a smooth Riemannian metric on $\partial M$. All norms,
lengths, angles, and distortion constants below refer to this metric.
The metric and foliations remain fixed as $L$ and $\epsilon$ vary.

\begin{lemma}[Relative strip stretching]
\label{lem:relative-boundary-stretching}
For every $L>1$ and $\epsilon>0$ there are smooth diffeomorphisms
$$
h_{\mathrm{in}}:\partial_{\mathrm{in}}M\to\partial_{\mathrm{in}}M,
\qquad
h_{\mathrm{out}}:\partial_{\mathrm{out}}M\to\partial_{\mathrm{out}}M,
$$
isotopic to the identity, with the following properties.
\begin{enumerate}
\item $h_{\mathrm{in}}$ preserves every $G^u_{\mathrm{in}}$-leaf
and is the identity near $\mathcal L^s$.
$h_{\mathrm{out}}$ preserves every $G^s_{\mathrm{out}}$-leaf
and is the identity near $\mathcal L^u$.
\item $(h_{\mathrm{in}}^{-1})_*G^s_{\mathrm{in}}$ and
$(h_{\mathrm{out}})_*G^u_{\mathrm{out}}$ are $\epsilon$-close
in tangent direction to $G^s_{\mathrm{in}}$ and $G^u_{\mathrm{out}}$.
\item $D(\Gamma\circ h_{\mathrm{in}})$ expands
$G^u_{\mathrm{in}}$ by more than $L$, and
$D(\Gamma^{-1}\circ h_{\mathrm{out}}^{-1})$ expands
$G^s_{\mathrm{out}}$ by more than $L$, on
$\partial_{\mathrm{in}}M\setminus\mathcal L^s$ and
$\partial_{\mathrm{out}}M\setminus\mathcal L^u$, respectively.
\end{enumerate}
The isotopies also preserve the leaves and are the identity
on fixed neighborhoods of the respective laminations.
\end{lemma}

\begin{proof}
We construct $h_{\mathrm{in}}$ using the strip method of
\cite[Proposition~6.2]{bby}. Applying the same construction to
$-X$ gives a map $k$ on $\partial_{\mathrm{out}}M$. Then
$h_{\mathrm{out}}=k^{-1}$ has the required exit properties.

By \cite[Lemma~3.24]{bby} applied to $-X$, the derivative of
$\Gamma^{-1}$ along $G^u_{\mathrm{out}}$ tends uniformly to zero
near $\mathcal L^u$. Since $\Gamma$ sends points approaching
$\mathcal L^s$ to points approaching $\mathcal L^u$, its expansion
along $G^u_{\mathrm{in}}$ tends uniformly to infinity near
$\mathcal L^s$. Choose an open neighborhood $U$ of $\mathcal L^s$
where this expansion exceeds $L$ on $U\setminus\mathcal L^s$.
The complementary strips form a disjoint open cover of the compact
set $\partial_{\mathrm{in}}M\setminus U$, so only finitely many
meet this set. The other strips need no modification.

In a selected strip, choose a full $G^u_{\mathrm{in}}$-segment $I$.
Its image under $\Gamma$ is an entire noncompact $G^u_{\mathrm{out}}$-leaf.
In compatible local arclength coordinates, $\Gamma|_I$ is an increasing diffeomorphism
$g:(0,\ell(I))\to\mathbb R$, with $g'$ tending to infinity at both
ends. For any $A>1$, \cite[Lemma~6.12]{bby} gives an increasing
$\widetilde g$ equal to $g$ near the ends with $\widetilde g'>A$.
Set $h=g^{-1}\circ\widetilde g$.

Extend $h$ by $G^s_{\mathrm{in}}$-holonomy. For a full segment $J$,
write $H_J:I\to J$ for this holonomy and
$h_J=H_JhH_J^{-1}$. The map
$K_J=\Gamma H_J\Gamma^{-1}:\Gamma(I)\to\Gamma(J)$ is
$G^s_{\mathrm{out}}$-holonomy, and
$$
\Gamma\circ h_J
=K_J\circ(\Gamma\circ h)\circ H_J^{-1}.
$$
Full transverse segments in filling strips have uniformly bounded
length. Since $G^s_{\mathrm{out}}$ is smooth and uniformly transverse
to $G^u_{\mathrm{out}}$, the derivatives of $K_J$ and $K_J^{-1}$
are uniformly bounded. Equation~\eqref{eq:s5-holonomy-length} gives
$$
|DH_J^{-1}|\geq
\frac{\ell(I)}{C_{\mathrm{hol}}\sup_J\ell(J)}>0.
$$
Thus expansion greater than $A$ on $I$ gives expansion greater
than $c_I A$ throughout the strip, for $c_I>0$ independent of $h$.
Choose $A$ so that $c_I A>L$.

We now cut off the strip map at both spiraling ends.
On the same reference segment $I$, set
$$
h_\tau=(1-\tau)\operatorname{id}+\tau h,\qquad 0\leq\tau\leq1.
$$
Near each limiting compact leaf, choose an annular coordinate
whose fibers are segments of $G^u_{\mathrm{in}}$, and let
$t\geq0$ be the lifted angular coordinate directed towards the end.
Write $H_t:I\to I_t$ for $G^s_{\mathrm{in}}$-holonomy to the full
fiber segment $I_t$. Choose a smooth cutoff $\chi(t)$ decreasing
from $1$ to $0$, constant near the ends of its transition interval,
and set
$$
\Psi|_{I_t}=H_t h_{\chi(t)}H_t^{-1}.
$$
Outside the two end regions, keep the original holonomy extension.
These definitions agree where $\chi=1$.

Put $m_h=\inf_I Dh>0$ and $M_h=\sup_I Dh$. Bounded distortion gives
$$
C_{\mathrm{hol}}^{-1}\min\{1,m_h\}
\leq D(\Psi|_{I_t})
\leq C_{\mathrm{hol}}\max\{1,M_h\}.
$$
Differentiating in $t$, the additional component transverse to
$G^s_{\mathrm{in}}$ is
$$
DH_t(h_{\chi(t)}(z))\,(h(z)-z)\,\chi'(t).
$$
The maps $H_t$ have uniformly bounded derivative. After a fixed
initial holonomy, they follow contracting compact-leaf holonomy.
The vector field tangent to $G^s_{\mathrm{in}}$, normalized by the angular
coordinate, is also uniformly bounded.
Together with the preceding positive lower bound, this yields
$$
\angle(\Psi_*G^s_{\mathrm{in}},G^s_{\mathrm{in}})
+\angle((\Psi^{-1})_*G^s_{\mathrm{in}},G^s_{\mathrm{in}})
\leq K_h\|\chi'\|_\infty,
$$
where $K_h$ is independent of the position and length of the
transition interval.

Choose the cutoffs after the finitely many maps $h$ have been fixed.
Start each transition sufficiently far along its end that the
original crossing expansion exceeds
$LC_{\mathrm{hol}}/\min\{1,m_h\}$ on every full segment there.
The cutoff preserves each segment, so the lower derivative bound
preserves expansion greater than $L$. Make the transition interval
long enough that $K_h\|\chi'\|_\infty<\epsilon$.
Elsewhere the original strip map already gives expansion greater
than $L$ and preserves $G^s_{\mathrm{in}}$.

The resulting maps are supported in a compact set disjoint from the lamination.
The holonomies are smooth there, so the maps extend smoothly by the identity.
Use $H_Jh_rH_J^{-1}$ in the middle of each strip and
$H_t h_{r\chi(t)}H_t^{-1}$ at its ends to obtain the required
leaf-preserving isotopies.
\end{proof}

We now use these diffeomorphisms to make the boundary jet strongly
admissible while preserving the positive characteristic foliation.

\begin{lemma}[Stretching the boundary jet]
\label{lem:MS-filling-strong-jets}
Let $\psi$ be strongly transverse and let $(\beta_-,\beta_+)$ be
smooth bi-contact forms supporting $X$ whose formal boundary
jet $j=(\lambda,\mu)=(\beta_+,\beta_-)|_{\partial M}$ satisfies
\eqref{eq:s5-jet-identification}, as in
Lemma~\ref{lem:MS-filling-matched-preparation}.
There is a strongly transverse isotopy to a map $\widehat\psi$
and a smooth strongly admissible boundary $X$-jet
$\widehat j=(\widehat\lambda,\widehat\mu)$ satisfying
\eqref{eq:s5-jet-identification} with $\psi$ replaced by
$\widehat\psi$. The jet agrees with $(\lambda,\mu)$ near
$\mathcal L^s\cup\mathcal L^u$, and
$\ker\widehat\lambda=\ker\lambda$ on $\partial M$.
Throughout the isotopy, the gluing maps identify these positive
characteristic foliations.
\end{lemma}

\begin{proof}
Fix the scattering frame and boundary metric chosen above.
Choose $\epsilon>0$ smaller than half the minimum angular
distance of $G^s_{\mathrm{in}}$ and $G^u_{\mathrm{out}}$
from the boundaries of their respective sectors.
For a parameter $L>1$, apply
Lemma~\ref{lem:relative-boundary-stretching}.
We will choose $L$ after obtaining estimates uniform in $L$. Put
$$
\Gamma_L=h_{\mathrm{out}}\circ\Gamma\circ h_{\mathrm{in}},
\qquad
\widehat\psi=h_{\mathrm{out}}^{-1}\circ\psi\circ
h_{\mathrm{in}}^{-1}.
$$
We estimate the auxiliary crossing map $\Gamma_L$ on the whole
transit domain. At a point $x$, choose locally positive unit vectors
$u_0,s_0$ tangent to
$G^u_{\mathrm{in}},(h_{\mathrm{in}}^{-1})_*G^s_{\mathrm{in}}$.
At $\Gamma_L(x)$ choose positive unit vectors $u_1,s_1$ tangent to
$(h_{\mathrm{out}})_*G^u_{\mathrm{out}},G^s_{\mathrm{out}}$.
Then
\begin{equation}\label{eq:MS-stretched-frame}
D\Gamma_L(u_0)=E u_1,\qquad D\Gamma_L(s_0)=F s_1,
\qquad E\geq c_0L,\quad 0<F\leq(c_0L)^{-1},
\end{equation}
where $c_0>0$ is independent of $L$.
Indeed, each leaf-preserving map moves points along full segments
of its fixed smooth foliation, whose lengths are uniformly bounded.
On the quotient by that tangent direction, its derivative is the
corresponding holonomy derivative and is uniformly bounded above
and away from zero. The angular control makes the original and
transported complementary directions uniformly transverse to the
preserved foliation. Thus $Dh_{\mathrm{out}}$ on
$G^u_{\mathrm{out}}$ and $D(h_{\mathrm{in}}^{-1})$ on
$G^s_{\mathrm{in}}$ have norm bounded below independently of $L$.
Combining these bounds with the two expansions in
Lemma~\ref{lem:relative-boundary-stretching} gives
\eqref{eq:MS-stretched-frame}. Compare \cite[Claims~6.3-6.5]{bby}.
The local orientations are transported consistently, so $E,F>0$.

At $x$ and $\Gamma_L(x)$, respectively, the coframe
matrices in these unit frames have the form
$$
B_0=\begin{pmatrix}p_0&-q_0\\0&w_0\end{pmatrix},
\qquad
B_1=\begin{pmatrix}p_1&-q_1\\r_1&0\end{pmatrix}.
$$
Simultaneous reversal of the local forms and frame vectors leaves
these matrices unchanged.
Here $p_0,q_0,w_0,p_1,q_1,r_1\in[m_0,M_0]$ for constants
$0<m_0<M_0$ independent of $L$. The zero entries follow from
$\lambda^{\mathrm{in}}(u_0)=\lambda^{\mathrm{out}}(s_1)=0$.
The bounds follow from compactness and the fixed angular margins.
Write
$$
\begin{pmatrix}\Gamma_L^*\mu^{\mathrm{out}}\\
\Gamma_L^*\lambda^{\mathrm{out}}\end{pmatrix}
=A\begin{pmatrix}\mu^{\mathrm{in}}\\\lambda^{\mathrm{in}}\end{pmatrix},
\qquad
A=\begin{pmatrix}a&b\\c&d\end{pmatrix}
=B_1\begin{pmatrix}E&0\\0&F\end{pmatrix}B_0^{-1}.
$$
Put $\Delta=\det A=(q_1r_1/(p_0w_0))EF>0$.
Direct multiplication gives
$$
d=\frac{r_1Eq_0}{p_0w_0}
\geq\frac{m_0^2}{M_0^2}E,
\qquad
\frac a\Delta
=\frac{p_1w_0}{q_1r_1F}
\geq\frac{m_0^2}{M_0^2}F^{-1}.
$$
Also $a,c,d>0$, while $b>0$ once $E/F>(M_0/m_0)^2$.
Choose $L>\max\{1,M_0^2/(c_0m_0^2)\}$. Then the ordinary
admissible order holds, and $d>1$, $a>\Delta$. With
$\omega_{\mathrm{in}}=\mu^{\mathrm{in}}\wedge\lambda^{\mathrm{in}}$,
these last inequalities are exactly
$$
\begin{aligned}
(\Gamma_L^*\lambda^{\mathrm{out}}-\lambda^{\mathrm{in}})
\wedge\mu^{\mathrm{in}}&=-(d-1)\omega_{\mathrm{in}}<0,\\
(\Gamma_L^*\lambda^{\mathrm{out}}-\lambda^{\mathrm{in}})
\wedge\Gamma_L^*\mu^{\mathrm{out}}&=(a-\Delta)\omega_{\mathrm{in}}>0.
\end{aligned}
$$

Set
$$
\widehat j^{\mathrm{in}}=(h_{\mathrm{in}}^{-1})^*j^{\mathrm{in}},
\qquad
\widehat j^{\mathrm{out}}=h_{\mathrm{out}}^*j^{\mathrm{out}}.
$$
Pulling back by $h_{\mathrm{in}}$ identifies their crossing
inequalities for $\Gamma$ with those just proved for $j$ and
$\Gamma_L$. The lamination conditions are unchanged, since the
maps are the identity near the laminations. Thus $\widehat j$
is strongly admissible on the whole boundary.

By Remark~\ref{rem:realization-input-self-gluing}, applied with
boundary maps $h_{\mathrm{in},r}^{-1}$ and $h_{\mathrm{out},r}$,
the maps
$$
\psi_r=h_{\mathrm{out},r}^{-1}\circ\psi\circ
h_{\mathrm{in},r}^{-1}
$$
form a strongly transverse isotopy. The pullback definitions give
$\widehat\psi^*\widehat j^{\mathrm{out}}
=\widehat j^{\mathrm{in}}$.
Finally, the leaf-preserving property in
Lemma~\ref{lem:relative-boundary-stretching} gives
$\ker\widehat\lambda=\ker\lambda$ and ensures that every $\psi_r$
identifies the original positive characteristic foliations.
\end{proof}


\subsection{Proof of Theorem~\ref{thm:MS-filling-self-gluing}}

\begin{proof}[Proof of Theorem~\ref{thm:MS-filling-self-gluing}]
Apply Lemma~\ref{lem:MS-filling-matched-preparation}, construct the
invariant boundary foliations, and apply
Lemma~\ref{lem:MS-filling-strong-jets}.
Proposition~\ref{prop:strongly-adapted-gluing} gives hyperbolicity
of the quotient from the resulting strongly admissible boundary jet.
Concatenate the preparation and stretching isotopies, reparametrizing
smoothly at their joining point. After absorbing the collar
preparations into the gluing maps as in
Remark~\ref{rem:realization-input-self-gluing}, the plug vector field
is $X$ throughout.
For complete gluing the quotient is closed, so its maximal invariant
set is the whole manifold and the induced vector field is Anosov.

For the relative assertion, choose compatible local defining coframes for the
prescribed admissible completion $(S_-,S_+)$ agreeing under $\psi$,
and extend them by Theorem~\ref{hypb}. Apply
Lemma~\ref{lem:MS-filling-strong-jets} directly to these supporting
bi-contact forms, omitting boundary preparation. It preserves $S_+$, and
Proposition~\ref{prop:strongly-adapted-gluing} realizes the resulting jet by a strongly
adapted form on the quotient with the required characteristic
foliation.
\end{proof}

\section{Building projectively Anosov flows}
\label{s6}

We apply the extension and gluing results to constructions with periodic
sources and sinks, invariant tori, and prescribed plug dynamics. We use
the unoriented forms of Section~\ref{s2}. Boundary preparations
refer to the operations of Section~\ref{s4}, with their stated hypotheses. They preserve the original plug flow up to smooth oriented orbit
equivalence and are fixed near its maximal invariant set.

\subsection{Flows without saddle periodic orbits}

For a structurally generic projectively Anosov flow with no saddle
hyperbolic periodic orbits, Corollary~\ref{corgeneric} gives a spectral
decomposition consisting of periodic sources and sinks. There is
exactly one of each. Their complement is connected, and its backward
and forward limiting periodic orbits are locally constant. Removing
transverse solid-torus neighborhoods of these two orbits leaves a
product plug. We use this decomposition to obtain the
following classification.

\begin{theorem}[Classification without saddle periodic orbits]
\label{thm:two-orbit-projectively-anosov}
A connected closed oriented $3$-manifold $M$ admits a structurally
generic projectively Anosov flow with no saddle hyperbolic periodic
orbits if and only if $M$ is diffeomorphic to $L(n,1)$ or
$L(4n,2n+1)$ for some
$n\in\mathbb Z$, where $L(0,1)=\mathbb S^2\times\mathbb S^1$.

After a suitable orientation-preserving identification, every such
flow is homotopic through projectively Anosov flows to a standard
model $X_n$ or $X_n^{\mathrm{tw}}$, according as it is transversely
orientable or not.
\end{theorem}

\begin{proof}
Use the standard source and sink models of
Proposition~\ref{reeb-boundary-source-sink} and its twisted version.
In local Floquet coordinates they have the form
$$
X=\partial_\theta+\nu x\partial_x+\mu y\partial_y,
$$
with $0<\nu<\mu$ at the source and $\nu<\mu<0$ at the sink.
Use two untwisted models or two twisted models, with identified ordered
boundary bi-foliations. Their characteristic foliations have two Reeb
annuli in the first case and one in the second.
Write $(\lambda_R,m_R)$ and $(\lambda_A,m_A)$ for longitude-meridian
bases. In the untwisted case the compact-leaf slope is $\lambda$, and
the framed gluing corresponding to $n$ is
\begin{equation}\label{eq:s6-source-sink-gluing}
(\psi_n)_*(\lambda_R)=\lambda_A,
\qquad
(\psi_n)_*(m_R)=n\lambda_A-m_A.
\end{equation}

For a twisted model, choose the half-turn longitude so that the
compact-leaf slope is $c=2\lambda-m$. The corresponding gluings are
\begin{equation}\label{eq:s6-twisted-source-sink-gluing}
\begin{aligned}
(\psi_n^{\mathrm{tw}})_*(c_R)&=c_A,\\
(\psi_n^{\mathrm{tw}})_*(\lambda_R)&=(n+1)c_A-\lambda_A,\\
(\psi_n^{\mathrm{tw}})_*(m_R)&=4n\lambda_A-(2n+1)m_A.
\end{aligned}
\end{equation}
Thus the resulting manifolds are $L(n,1)$ and $L(4n,2n+1)$, respectively.

All these gluings can be realized by boundary preparation. Indeed,
in coordinates with $\theta$ along the compact-leaf slope, write locally
$\eta=A(\varphi)d\theta+B(\varphi)d\varphi$, with dividing function
$f=f(\varphi)$. The shear
$$
\eta_r=A(\varphi)d\theta+
\bigl(B(\varphi)+rnA(\varphi)\bigr)d\varphi,
\qquad r\in[0,1],
$$
preserves the dividing inequality, since
$$
f\,d\eta_r+\eta_r\wedge df
=(Af'-fA')\,d\theta\wedge d\varphi.
$$
The coefficients and dividing function have common transition signs.
Applying the same shear to the complementary form preserves their
wedge product. Theorem~\ref{prop:coordinated-divided-deformation}
therefore realizes this path in a boundary collar, relative to the
supporting bi-contact structure near the periodic orbit. Its endpoint
is the pullback by
$\tau_n(\theta,\varphi)=(\theta+n\varphi,\varphi)$.
Starting with the matching gluing at $n=0$, this gives
\eqref{eq:s6-source-sink-gluing} or
\eqref{eq:s6-twisted-source-sink-gluing}.

Choose a smaller transverse source neighborhood
$V_R'\subset\operatorname{int}(V_R)$. The shell
$P=V_R\setminus\operatorname{int}(V_R')$ is a product
bi-contact plug whose outer boundary bi-foliation matches that of the sink.
Theorem~\ref{thm:smooth-matched-gluing} gives a smooth bi-contact
structure on
$$
V_R'\cup P\cup_\psi V_A.
$$
Here $\psi$ is the chosen gluing.
The supporting bi-contact structure makes the resulting model
$X_n$ or $X_n^{\mathrm{tw}}$ projectively Anosov.
Every orbit crossing the product travels from the source neighborhood
to the sink neighborhood, so the two original hyperbolic periodic
orbits form the entire nonwandering set. Thus both models are Morse-Smale.

Now let $X$ satisfy the hypotheses. After a positive time change,
Floquet coordinates and a $C^1$-small cutoff of the nonlinear terms
make $X$ linear near its two periodic orbits. Projective hyperbolicity
and structural stability persist. For
$$
\beta_u=dy-\mu y\,d\theta,
\qquad
\beta_s=dx-\nu x\,d\theta,
$$
the supporting forms near each orbit can be written locally, up to positive
factors, as
$\alpha_+=\beta_u-e^{b_+}\beta_s$ and
$\alpha_-=\beta_u+e^{b_-}\beta_s$.
Their contact conditions are $Xb_\pm<\mu-\nu$. A slow cutoff
interpolating $b_\pm$ to zero therefore gives the standard bi-contact
models near the periodic orbits, without changing the structures
outside these neighborhoods. The rates can likewise be joined to
those of the reference models within the ordered source and sink ranges.

Remove standard neighborhoods $V_R,V_A$. The remainder has empty
maximal invariant set, so flowing its entrance boundary to its exit
boundary identifies it, after a positive time change, with
$(\mathbb T^2\times[0,1],\partial_t)$. Let $\Gamma$ be its crossing
map. Product admissibility makes
$\Gamma_*\mathcal S_+^R$ transverse to $\mathcal S_+^A$.
The two periodic neighborhoods are both untwisted or both twisted,
since their boundary bi-foliations have the same coorientability.
Two transverse Reeb-annulus foliations have the same compact-leaf
slope. To see this, pass to their common coorientation cover.
A vector field directing one lifted foliation
points inward along both boundary circles of a Reeb annulus of the
other, or outward along both. After reversing time if necessary,
Poincar\'e-Bendixson gives an essential periodic orbit inside the
annulus. Thus $\Gamma$ identifies the compact-leaf slopes up to sign.
With compatible orientations of the framed bases, $\Gamma$ has
the action in \eqref{eq:s6-source-sink-gluing} or
\eqref{eq:s6-twisted-source-sink-gluing} for some $n$.
This gives the asserted list of manifolds.

It remains to compare $X$ with the corresponding standard model. Normalize the periodic
neighborhoods and fix the entrance characteristic bi-foliation on the
complementary product. The admissible exit bi-foliations form a
contractible space. Their angles in the interval between the entrance
directions satisfy $0<a_-<a_+<1$, and these inequalities are convex.

Using the standard neighborhood identifications in the fixed framed
gluing class, realize this boundary homotopy on the adjacent periodic
neighborhoods, preserving the standard models near the periodic orbits. The relative and parametric
statement of Theorem~\ref{prob} extends it over the product, and its
fixed-boundary contractibility removes the remaining choice of
supporting extension. After an orientation-preserving identification
$h$ to the corresponding lens space, this gives the required homotopy
from $h_*X$ to $X_n$ or $X_n^{\mathrm{tw}}$.

Throughout the homotopy, transversality to the product slices prevents
additional recurrent orbits, and the source and sink remain hyperbolic.
Hence every intermediate flow is Morse-Smale, and structural stability
gives topological orbit equivalence of the endpoints.
\end{proof}

\subsection{Suspended-gradient plugs and saddle counts}

We replace an elementary attracting solid torus by one containing
$k$ saddle and $k+1$ attracting periodic orbits, while preserving
the vector field and both supporting contact forms near its boundary.

\begin{theorem}[Arbitrarily many saddle periodic orbits]
\label{thm:arbitrary-saddle-count}
For every integer $k\geq0$, there is a closed oriented $3$-manifold
carrying a transversely oriented projectively Anosov Morse-Smale flow
with exactly $k$ saddle periodic orbits. The flow may be chosen to
have one repelling periodic orbit and $k+1$ attracting periodic
orbits. In particular, it has $2k+2$ periodic orbits in total and
zero topological entropy.
\end{theorem}

\begin{proof}
Let $D$ be a slightly rounded square $[-1,1]^2$. Choose a smooth
function $f:[-1,1]\to\mathbb R$, equal to $-y$ near the endpoints,
with exactly $2k+1$ simple zeros
$$
y_1<\cdots<y_{2k+1}.
$$
Choose $f$ linear near each zero. Its derivatives at these zeros
alternate in sign, with $f'(y_i)<0$ for odd $i$ and $f'(y_i)>0$
for even $i$. Using a cutoff in $x$, choose a smooth function $g$ on
$D$ such that
$$
g(x,y)=f(y)\quad\text{near }x=0,
\qquad
g(x,y)=-y\quad\text{near }\partial D.
$$
Choose
$$
a>\|\partial_x g\|_\infty+\|\partial_y g\|_\infty
$$
and set, on $V=\mathbb S^1_\theta\times D$,
$$
X=\partial_\theta+Y,
\qquad
Y=-ax\partial_x+g(x,y)\partial_y.
$$
Define
$$
\beta_u=dy-g\,d\theta,
\qquad
\beta_s=dx+ax\,d\theta,
\qquad
\alpha_\pm=\beta_u\mp\beta_s.
$$
For $\Omega=d\theta\wedge dx\wedge dy$, direct calculation gives
$$
\begin{aligned}
\alpha_+\wedge d\alpha_+
&=(a+\partial_y g+\partial_x g)\Omega>0,\\
\alpha_-\wedge d\alpha_-
&=-(a+\partial_y g-\partial_x g)\Omega<0,\\
\alpha_-\wedge\alpha_+&=2\iota_X\Omega.
\end{aligned}
$$
Thus $(\ker\alpha_-,\ker\alpha_+)$ supports $X$. Near $\partial V$,
both the vector field and the forms agree exactly with the elementary
attracting model of Proposition~\ref{reeb-boundary-source-sink},
with rates $-a<-1<0$.

Since $\dot x=-ax$, the periodic orbits are precisely
$\gamma_i=\mathbb S^1\times\{(0,y_i)\}$, with normal exponents
$-a$ and $f'(y_i)$. Hence there are $k$ saddles and $k+1$ sinks.
The vector field $Y$ is gradient-like. For sufficiently large $B>0$,
$$
F(x,y)=Bx^2-\int_0^y f(s)\,ds
$$
strictly decreases outside its critical points, since
$YF=-2aBx^2-f(y)g(x,y)$ and $g=f$ near $x=0$.
Every saddle's unstable branches lie in $x=0$ and converge to the
adjacent sinks. There are therefore no saddle connections.

Removing small standard neighborhoods of the $k+1$ sinks gives a
hyperbolic suspension plug
$$
M_k=\mathbb S^1\times\Sigma_k,
\qquad
\Sigma_k=D\setminus\bigcup_{j=1}^{k+1}\operatorname{int}D_j,
$$
with one entrance, $k+1$ exits, and precisely the $k$ saddle orbits
as its maximal invariant set. Its induced boundary bi-foliation is
admissible. Each exit bi-foliation is exactly that of the elementary
attracting cap removed there, while the entrance bi-foliation is unchanged
from the original attracting model.

Take a source-sink example from
Theorem~\ref{thm:two-orbit-projectively-anosov}, choosing its
elementary attracting model with these rates, and replace that cap
by $V$. The vector fields and both contact forms agree on a boundary
collar, so the replacement directly gives a smooth closed bi-contact
structure and hence a projectively Anosov flow.

No recurrent orbit crosses the boundary of $V$. The nonwandering
set consequently consists of the original source and the $2k+1$
periodic orbits in $V$. There are no saddle connections, and
intersections involving the unstable manifold of the source or
the stable manifold of a sink are automatically transverse. The
flow is therefore Morse-Smale, with the asserted orbit count and
zero topological entropy.
\end{proof}

\begin{remark}
One can show that, for every $n\in\mathbb Z$ and $k\geq0$,
$L(n,1)$ admits a projectively Anosov Morse-Smale flow with exactly
one repelling periodic orbit, $k$ saddle periodic orbits, and
$k+1$ attracting periodic orbits.
\end{remark}
\subsection{Capping by finite Morse-Smale pieces}
\label{subsec:finite-PA-capping}

The saddle attachments of Section~\ref{s3} prepare
boundary bi-foliations for the non-hyperbolic caps without assuming
coorientability, winding parity, or symmetry. This builds on the finite capping constructions
of \cite{bbp}.

\begin{lemma}[Preparation for finite capping]
\label{lem:finite-cap-preparation}
Let $T$ be a boundary component of a bi-contact plug.
Finitely many twisted-saddle and two-saddle attachments and product
collars produce a terminal boundary bi-foliation admitting a
compatible annular double, with the convention for an attracting
cap at an exit and a repelling cap at an entrance. The characteristic
bi-foliation on $T$ is preserved, and the added pieces have only
finitely many periodic orbits, all hyperbolic.
\end{lemma}

\begin{proof}
By Proposition~\ref{prop:twisted-saddle-preparation}, at most three
twisted-saddle attachments make the terminal bi-foliation coorientable
with odd oriented winding, preserving the bi-foliation on $T$.
Add a product collar. By Peixoto's density theorem and
Corollary~\ref{cor:small-nonconvex-realization}, its outer
characteristic bi-foliation may be made Morse-Smale by a
$C^1$-small smooth perturbation, preserving transversality and
odd oriented winding while keeping the supported vector field and the
bi-foliation on $T$ fixed. Consider an exit component.
Reverse time for an entrance.
Repeatedly apply Lemma~\ref{lem:gradient-boundary-simplification}
until $S_-$ has only negative-turning Reeb annuli and parallel
annuli. Each attachment changes its oriented winding by $-2$, so the final
oriented winding is $-n$ for some positive odd integer $n$.

Write the current characteristic bi-foliation as
$S_\pm=\ker\eta_\pm$. Choose a smooth foliation $G=\langle g\rangle$,
sufficiently $C^1$-close to $S_-$, with
$\eta_-(g)\eta_+(g)<0$, the attracting cap quadrant.
It has the same Morse-Smale arrangement as $S_-$. In the notation
of Section~\ref{subsubsec:gradient-topological-admissibility}, its
cyclic annular word is
$$
R_-P^{2p_1}\cdots R_-P^{2p_{2n}},
\qquad p_i\geq0.
$$
There are $2n$ turning annuli, and the intervening parallel annuli
occur in pairs because the adjacent turning annuli have the same
holonomy sign. Inserting cancelling pairs of compact leaves in
$G$ increases any $p_i$ by one. These changes affect only $G$ and
can be made arbitrarily $C^0$-small, preserving the quadrant
condition. Taking $p=\max_i p_i$ gives the word
$$
\bigl(R_-P^{2p}\bigr)^{2n}.
$$

This word has a smooth reflection-symmetric representative $H$.
Cutting through the middle of two opposite turning annuli gives
a vector field on an annulus with opposite transverse directions at its two
boundary circles, since $n$ is odd. Orient it outward along both
circles. Its annular double has the foliation $H$. Use the
time-reversed model for the attracting cap.

Straighten the compact leaves and interpolate their hyperbolic
local models and the transverse components between them,
preserving their signs. This gives a smooth path from $G$ to $H$
through nonsingular Morse-Smale foliations. Each is divided by
transverse circles in the complementary annuli.
Lemma~\ref{lem:parametric-divided-forms} supplies smooth defining
forms and dividing functions along this path.

Choose the attracting version of the collar model in
Theorem~\ref{prop:coordinated-divided-deformation}, with sufficiently
small tilts about $G$ in the admissible order. Theorem~\ref{prob}
joins the current boundary bi-foliation to that on the inner boundary
of this model, preserving the model near the joining torus. The coordinated
divided deformation then realizes the path from $G$ to $H$ on
this collar. The terminal bi-foliation consists of small tilts
about $H$ in the attracting cap quadrant and therefore admits
the required compatible annular double. The supported vector fields
remain transverse to every collar slice, so the collars introduce
no periodic orbits. Each twisted block contributes one saddle orbit
and each two-saddle block contributes two.
\end{proof}

\begin{theorem}[Finite projectively Anosov capping]
\label{thm:finite-PA-capping}
Let $(M,\xi_-,\xi_+,X)$ be a bi-contact plug, and let $\mathcal T$
be any union of boundary components. Attaching
finitely many saddle plugs and attracting or
repelling solid-torus caps produces a bi-contact plug
$(\widehat M,\widehat\xi_-,\widehat\xi_+,\widehat X)$ with
$$
\partial\widehat M=\partial M\setminus\mathcal T,
\qquad \widehat X|_M=X.
$$
The original contact structures are preserved outside arbitrarily
small collars of $\mathcal T$. The added nonwandering dynamics
consist of finitely many hyperbolic periodic orbits, and every
orbit crossing an attaching torus is wandering.

If $\mathcal T=\partial M$, the resulting closed flow is
projectively Anosov and has the same topological entropy as
$X|_{\operatorname{Inv}(M,X)}$. It has finitely many periodic
orbits whenever $(M,X)$ does. If $(M,X)$ is partially hyperbolic and
$\mathcal T\subset\partial_{\mathrm{in}}M$, the resulting plug
is partially hyperbolic.
\end{theorem}

\begin{proof}
Apply Lemma~\ref{lem:finite-cap-preparation} at each selected
component. Proposition~\ref{stable-nonhyperbolic-cap-realization}
supplies the terminal solid-torus caps, and
Theorem~\ref{thm:smooth-matched-gluing} attaches them. Perform the
preparations in the added pieces. Smooth gluing preserves $X$ on
$M$ and changes its contact structures only in arbitrarily small
collars of the attaching tori.

The attachments form directed chains from repelling caps towards
$M$ at its entrances and from $M$ towards attracting caps at its
exits. Every orbit crossing an attaching torus is therefore
wandering. Each added piece has Morse-Smale dynamics with finitely
many hyperbolic periodic orbits, which gives the assertion about
nonwandering dynamics. When all boundary components are capped,
the closed supporting bi-contact structure gives projective
Anosovity.

Only finitely many periodic orbits are added to the nonwandering
set, and $\widehat X|_M=X$. Since periodic orbits have zero
entropy and entropy is determined by the nonwandering set
\cite[Theorem~4.2.26]{hflows}, we obtain
$
h_{\mathrm{top}}(\widehat X)
=h_{\mathrm{top}}\bigl(X|_{\operatorname{Inv}(M,X)}\bigr).
$

If $(M,X)$ is partially hyperbolic and only entrance components
are capped, the added periodic orbits are saddles and sources.
The dominant normal direction expands uniformly on the original
maximal invariant set and on each added periodic orbit.
Lemma~\ref{lem:partially-hyperbolic-plug} therefore gives partial
hyperbolicity of the resulting plug.
\end{proof}

For a hyperbolic plug, it suffices to prescribe an admissible
boundary bi-foliation and apply Theorem~\ref{hypb} first.
The complete piece attached to an original
boundary component need not itself be a solid torus.

\subsection{Toric transitions and stacks}
\label{subsec:toric-stacks}

The normal models of Section~\ref{s3} realize prescribed nonsingular
flows on invariant tori. We glue these models using
Proposition~\ref{two-ended-boundary-realization} and
Theorem~\ref{thm:smooth-matched-gluing}.

Let $Y_i$ be smooth nonsingular vector fields on oriented tori $T_i$,
$i=1,2$, and let $g:T_1\to T_2$ be an orientation-preserving
diffeomorphism with
$g_*\langle Y_1\rangle\pitchfork\langle Y_2\rangle$. Choose a
bi-foliation $(S_-,S_+)$ with projective cyclic order
$$
\langle Y_1\rangle,\quad S_-,\quad
(g^{-1})_*\langle Y_2\rangle,\quad S_+.
$$
This is core-compatible on the positive face of the first normal
model and, after applying $g$, on the negative face of the second.
For opposite normal types, put
$$
\delta(g;Y_1,Y_2)
=\operatorname{sgn}\omega_1\bigl(Y_1,(g^{-1})_*Y_2\bigr),
$$
where $\omega_1$ is a positive area form. Ordinary defining forms on
successive models match after simultaneous sign changes satisfying
\begin{equation}\label{eq:s6-toric-coorientation}
\sigma_2=-\delta(g;Y_1,Y_2)\sigma_1.
\end{equation}
The contact structures glue regardless of these signs.

\begin{lemma}[Toric transition]
\label{lem:toric-transition}
Let $Y_R,Y_A$ be smooth nonsingular vector fields on oriented tori
$T_R,T_A$, and let $g:T_R\to T_A$ be an orientation-preserving
diffeomorphism with
$g_*\langle Y_R\rangle\pitchfork\langle Y_A\rangle$.
Joining their normal half-models along $g$ gives a product carrying
a bi-contact structure supporting $X$, with
$$
X|_{T_R}=Y_R,\qquad X|_{T_A}=Y_A.
$$
The boundary tori are normally repelling and attracting,
respectively. The vector field is transverse to the intermediate tori,
and every interior point $x$ satisfies
$$
\alpha_X(x)\subset T_R,\qquad \omega_X(x)\subset T_A.
$$
\end{lemma}

\begin{proof}
Use a repelling positive half-model and an attracting negative
half-model. Proposition~\ref{two-ended-boundary-realization}
realizes the identified core-compatible bi-foliations above.
Theorem~\ref{thm:smooth-matched-gluing} glues them while preserving
the supported vector fields. The normal models give the orbit limits.
\end{proof}

Iterating along finitely many successive tori in
$\mathbb T^2\times[0,1]$ gives a finite stack whenever consecutive
prescribed vector fields are transverse under the product identifications.
Either alternating assignment of normal types is possible, and
the prescribed vector fields are preserved. Every intervening orbit
approaches the repelling torus backward in time and the attracting
torus forward in time. The absolute normal rates can be made
arbitrarily large relative to the tangential derivatives.
When the vector fields on the tori are not prescribed, they may be chosen
Morse-Smale in the required open sectors. Peixoto's density theorem
provides a small $C^1$ perturbation preserving all strict sector
conditions.

\begin{theorem}[Cyclic toric stacks]
\label{thm:cyclic-toric-stack}
Let $Y_i$ be smooth nonsingular vector fields on oriented tori $T_i$,
$1\leq i\leq2m$, and let $g_i:T_i\to T_{i+1}$ be
orientation-preserving diffeomorphisms, with indices modulo $2m$.
Suppose
$$
(g_i)_*\langle Y_i\rangle\pitchfork\langle Y_{i+1}\rangle.
$$
The torus bundle with monodromy $g_{2m}\circ\cdots\circ g_1$
admits a bi-contact structure supporting a projectively
Anosov vector field $X$ with $X|_{T_i}=Y_i$. These invariant tori alternate
between normal repulsion and attraction. Every other orbit is
wandering and approaches a repelling torus backward in time and
an attracting torus forward in time.
The supporting contact structures are coorientable precisely when
$\prod_{i=1}^{2m}\delta(g_i;Y_i,Y_{i+1})=1$.
\end{theorem}

\begin{proof}
Join alternating full normal models by the preceding construction.
The identifications give the stated monodromy, and the
supporting bi-contact structure gives projective Anosovity.
The normal models give the orbit description.
Equation~\eqref{eq:s6-toric-coorientation} gives the coorientability criterion.
\end{proof}

\begin{corollary}
\label{cor:pa-on-every-torus-bundle}
Every oriented torus bundle over $\mathbb S^1$ admits a cooriented
bi-contact structure supporting a projectively Anosov flow of the
preceding type, with any prescribed even number $2m\geq2$ of
distinguished invariant tori.
\end{corollary}

\begin{proof}
Represent the monodromy by $A\in\mathrm{SL}(2,\mathbb Z)$.
Unless $A=I$ and $2m=2$, take $g_i=\operatorname{id}$ for $i<2m$
and $g_{2m}=A$. Choose a nonzero constant vector $v_1$ and let
$\Theta\in(0,2\pi]$ be the oriented angle from $v_1$ to
$A^{-1}v_1$, choosing $v_1$ so that $\Theta<2\pi$ when $2m=2$.
Take constant vector fields $v_i$ at lifted angles
$(i-1)\Theta/(2m)$ from $v_1$. All successive increments lie in
$(0,\pi)$, so all signs $\delta$ are positive.

For $A=I$ and $2m=2$, on
$(\mathbb R/2\pi\mathbb Z)^2$ take
$$
Y_1=\cos u\,\partial_u+\sin u\,\partial_v,
\qquad
Y_2=-\sin u\,\partial_u+\cos u\,\partial_v,
$$
with $g_1=\operatorname{id}$ and $g_2(u,v)=(u+\pi,v)$.
Then $(g_2^{-1})_*Y_1=-Y_1$, so both signs $\delta$ are positive,
and $g_2$ is isotopic to the identity. Apply
Theorem~\ref{thm:cyclic-toric-stack} in both cases.
\end{proof}

\begin{remark}\label{rem:noda-regularity}
These relative constructions complement Noda's realization of
nonsingular torus flows and his examples on torus mapping tori
\cite[Proposition~3.3 and the remark following Example~3.1]{noda2000}.
No smoothness of the global weak invariant plane fields is asserted.
The regular classification of \cite{noda,asaoka} requires additional
regularity.
\end{remark}
\subsection{Doubling bi-contact plugs}
\label{subsec:bicontact-doubling}

The following toric insertion requires coorientable boundary
bi-foliations. A nowhere-vanishing vector field transverse to a
line field on an oriented torus coorients that line field.

\begin{theorem}[Bi-contact doubling]
\label{thm:bicontact-doubling}
Let $(M,\xi_-,\xi_+,X)$ be a compact bi-contact plug with coorientable
boundary characteristic bi-foliations.
Its oriented double $D(M)$ admits a projectively Anosov vector field
$\widehat X$ restricting to $X$ on both copies. Its supporting
bi-contact structure agrees with the original structures outside
arbitrarily small boundary collars, with the contact signs
interchanged on $-M$.

Each doubled boundary component gives an invariant torus carrying
a Morse-Smale flow, normally repelling at an entrance and normally
attracting at an exit. No orbit passes between the two copies,
and the only additional nonwandering orbits are finitely many
hyperbolic periodic orbits.

If the original contact structures are coorientable and all boundary
components have the same type, the doubled structures can be chosen
coorientable. If all are entrance components and the maximal
invariant set of $X$ is partially hyperbolic, $\widehat X$ can also
be chosen partially hyperbolic.
\end{theorem}

\begin{proof}
If $\partial M=\varnothing$, take the disjoint union of the two
copies. Otherwise, keep $X$ on $-M$ and interchange the contact
structures. Choose local defining forms $\alpha_\pm$ on $M$,
ordinary near each boundary component.
On each boundary torus $T$, write $\eta_\pm=\alpha_\pm|_{TT}$ and
put $\varepsilon_T=-1$ at an entrance and $\varepsilon_T=1$ at an
exit. Choose a smooth nonsingular Morse-Smale vector field $Y_T$ satisfying
$$
\eta_+(Y_T)>0,\qquad -\varepsilon_T\eta_-(Y_T)>0.
$$
Such a vector field is obtained by a small perturbation within this open
sector.

Insert the normal model
$$
Z_T=Y_T+K_Ts\partial_s
\quad\text{on }[-1,1]_s\times T,
\qquad
\operatorname{sgn}K_T=-\varepsilon_T,
$$
with $|K_T|$ sufficiently large and $T$ oriented by
$-\varepsilon_T\eta_-\wedge\eta_+$.
The boundary bi-foliations defined by
$$
(\eta_-,\eta_+)\quad\text{at }s=1,
\qquad
(\varepsilon_T\eta_+,\varepsilon_T\eta_-)
\quad\text{at }s=-1
$$
are core-compatible. Proposition~\ref{two-ended-boundary-realization}
realizes them, and Theorem~\ref{thm:smooth-matched-gluing} attaches
the model to both copies while preserving the supported vector fields.
These product insertions give $D(M)$.

The contact planes glue independently at each boundary component.
When the original structures are coorientable and all components
have the same type, using
$(\varepsilon_T\alpha_+,\varepsilon_T\alpha_-)$ throughout $-M$
also gives global coorientations.
The resulting bi-contact structure supports a projectively Anosov
vector field. The equation $\dot s=K_Ts$ gives the normal types and prevents
passage between the copies. Every orbit crossing an attaching torus
is wandering. The added recurrent orbits are therefore the finitely
many periodic orbits of the $Y_T$. They are sources and saddles at entrances,
and sinks and saddles at exits.

Finally, suppose every boundary component is an entrance and the
original maximal invariant set is partially hyperbolic.
Choose $K_T>0$ so that normal expansion dominates the tangential
derivatives. The dominant normal direction then expands uniformly
on the inserted tori and on both original maximal invariant sets.
Every remaining orbit is wandering, so
Lemma~\ref{lem:partially-hyperbolic-plug} gives partial hyperbolicity.
\end{proof}

The vector fields on the tori may instead be prescribed arbitrarily within the
indicated sectors. Only the assertion about additional nonwandering orbits requires the
Morse-Smale choice. In particular, every hyperbolic bi-contact plug
with coorientable boundary bi-foliations
embeds in a closed projectively Anosov flow on its oriented double,
preserving its vector field.




\section{Building Liouville domains}
\label{sec:liouville-domains-repelling-plugs}

Throughout this section, we assume that the contact structures are
cooriented and use the ordinary notion of global contact forms (in contrast to unoriented contact forms).

We extend the Liouville construction of \cite{mit,hoz3,hoz6} to
partially hyperbolic bi-contact plugs, preserving their supporting
contact structures.

\begin{lemma}[An expanding defining form]
\label{lem:expanding-form-partially-hyperbolic-plug}
Let $(M,\xi_-,\xi_+,X)$ be a compact cooriented bi-contact plug
whose maximal invariant set $\Lambda$ is partially hyperbolic,
with normal splitting $E\oplus F$ and $F$ uniformly expanding.
There is a nowhere-vanishing continuous $1$-form $\alpha$ on $M$,
with continuous Lie derivative along $X$, such that
$$
\alpha(X)=0,\qquad \mathcal L_X\alpha=r\alpha,\qquad r>0,
$$
where $r$ is continuous and $\ker\alpha|_\Lambda$ is the inverse
image of $E$ under $TM\to\mathcal N_X$.
\end{lemma}

\begin{proof}
Backward transport of smooth exit lines in the dominated cone,
together with cone contraction on forward trapped orbits, gives
a continuous invariant extension $\widetilde E$ over $M$.
Choose a unit defining form $\beta$ for its inverse image in $TM$.
Invariance gives $\mathcal L_X\beta=\rho\beta$, with $\rho$ continuous.

Extend the flow and $\beta$ beyond the exit boundary, keeping
$\ker\beta$ invariant and arranging $\rho=1$ beyond a fixed collar.
Write $(\varphi^t)^*\beta=c_t\beta$, so that
$$
c_t(x)=\exp\!\left(\int_0^t\rho(\varphi^u x)\,du\right).
$$
Uniform expansion of $\mathcal N_X/E$ on $\Lambda$ and the averaging
argument in the proof of Lemma~\ref{lem:partially-hyperbolic-plug}
give uniformly positive averages of $\rho$ along sufficiently long
orbit segments contained in $M$. After exiting, an orbit spends
bounded time in the transition collar and then enters the region where $\rho=1$.
Thus $c_T>1$ on $M$ for some fixed $T>0$.

Set
$$
\alpha=\int_0^T(\varphi^t)^*\beta\,dt
=\left(\int_0^T c_t\,dt\right)\beta.
$$
Then
$$
\mathcal L_X\alpha
=(\varphi^T)^*\beta-\beta
=\frac{c_T-1}{\int_0^T c_t\,dt}\,\alpha,
$$
which proves the claim.
\end{proof}

\begin{remark}
If $\widetilde E$ is $C^1$, the construction gives a $C^1$ form
$\alpha$. In general, branching at repelling periodic orbits
is possible and $C^1$ regularity does not always hold.
\end{remark}

\begin{theorem}[Liouville structures from partially hyperbolic plugs]
\label{thm:liouville-domain-repelling-plug}
Let $(M,\xi_-,\xi_+,X)$ be a compact cooriented bi-contact plug
whose maximal invariant set is partially hyperbolic.
There are smooth contact forms $\alpha_\pm$ defining $\xi_\pm$
such that
$$
\lambda=(1-s)\alpha_-+(1+s)\alpha_+
$$
is a nowhere-vanishing Liouville form on $[-1,1]_s\times M$, with
Liouville vector field
$$
Z=A\partial_s+BX,\qquad B>0.
$$
The horizontal faces are convex. The vertical faces are convex
over $\partial_{\mathrm{out}}M$ and concave over
$\partial_{\mathrm{in}}M$. If $M$ is repelling, rounding the corners
gives a Liouville domain whose skeleton projects homeomorphically onto the
maximal invariant set of $X$.
\end{theorem}

\begin{proof}
Choose $\alpha$ from the preceding lemma. The analogous forward
construction gives an invariant completion $\widetilde F$ in the
dominant normal cone. Choose smooth defining forms $\gamma_\pm$
of $\xi_\pm$. After choosing the sign of $\alpha$, positively
rescale them to agree with $\alpha$ on $\widetilde F$:
$$
\widehat\alpha_-=h_-\gamma_-=\alpha-\beta,
\qquad
\widehat\alpha_+=h_+\gamma_+=\alpha+f\beta,
\qquad h_\pm,f>0,
$$
where $\beta$ annihilates $X$ and $\widetilde F$.
These forms and functions are continuous with continuous derivatives
along $X$. Since $\mathcal L_X\beta$ is a multiple of $\beta$,
the interpolation
$\widehat\lambda=(1-s)\widehat\alpha_-+(1+s)\widehat\alpha_+$
satisfies
$$
-\partial_s\widehat\lambda\wedge\mathcal L_X\widehat\lambda
=2r\,\widehat\alpha_-\wedge\widehat\alpha_+.
$$
Smoothly approximate $h_\pm$ and their derivatives $Xh_\pm$
uniformly, using smoothing in flow boxes. The resulting positive
smooth rescalings $\alpha_\pm$ still define $\xi_\pm$ and satisfy
\begin{equation}
-\partial_s\lambda\wedge\mathcal L_X\lambda
=\Delta\,\alpha_-\wedge\alpha_+,
\qquad \Delta>0.
\label{eq:liouville-interpolation-criterion}
\end{equation}

Take $\Omega>0$ with $\iota_X\Omega=\alpha_-\wedge\alpha_+$.
Cartan's formula gives
\begin{equation}
(d\lambda)^2=2\Delta\,ds\wedge\Omega>0,
\label{eq:liouville-square}
\end{equation}
and solving $A\partial_s\lambda+B\mathcal L_X\lambda=\lambda$
gives $Z=A\partial_s+(2/\Delta)X$.
Linear independence of $\alpha_-,\alpha_+$ makes $\lambda$ nowhere zero.
The restrictions $2\alpha_\pm$ give convex horizontal faces,
and $B>0$ gives the asserted vertical boundary signs.
When all faces are convex, the corners can be rounded while preserving
outward transversality.

The skeleton consists of the points whose Liouville trajectories
remain in the domain for all time. By domination, these are
precisely the points over $\operatorname{Inv}(M,X)$ where the
normal covector $\lambda_{(s,x)}$ annihilates the dominated line
$E_x$. The linear interpolation determines a unique such $s$
in each fiber, depending continuously on $x$. Thus projection
restricts to a homeomorphism, and sufficiently small corner
rounding leaves the skeleton unchanged.
\end{proof}

\begin{remark}[Strongly adapted contact forms]
\label{lem:scaled-strong-adaptation-liouville}
\label{rem:liouville-scaling-constant}
For a hyperbolic plug with a strongly adapted positive contact
form $\alpha_+$, set $\alpha_-=\mathcal L_X\alpha_+$ and
\begin{equation}
\lambda_C=C(1-s)\alpha_-+(1+s)\alpha_+.
\label{eq:scaled-liouville-form}
\end{equation}
Writing $\Omega=\alpha_+\wedge d\alpha_+$ and
$$
-\alpha_-\wedge d\alpha_-=a\Omega,\qquad
\alpha_+\wedge d\alpha_-=b\Omega,\qquad a>0,
$$
one obtains
$$
(d\lambda_C)^2
=2\bigl[C(1-s)(b+Ca)+(1+s)\bigr]\,ds\wedge\Omega>0
$$
whenever
\begin{equation}
C>\max_{x\in M}\max\{0,-b(x)/a(x)\}.
\label{eq:C-liouville-bound}
\end{equation}
This gives the same boundary and projection properties directly.
\footnote{The factor $C$ was omitted from the corresponding
formula in Corollary~1.1 of \cite{hoz7}.}
\end{remark}

For general $M$, the construction gives an exact symplectic
cobordism with corners whose negative and positive faces are
$$
\partial_-W=[-1,1]\times\partial_{\mathrm{in}}M,
\qquad
\partial_+W=(\{-1,1\}\times M)
\cup([-1,1]\times\partial_{\mathrm{out}}M).
$$
All boundary faces are convex precisely when
$\partial_{\mathrm{in}}M=\varnothing$. Rounding then gives a
Liouville domain. This applies to cooriented repelling hyperbolic
plugs, to partially hyperbolic plugs after capping all entrance
components under the hypotheses of Theorem~\ref{thm:finite-PA-capping},
and to the cooriented partially hyperbolic doubles of
Theorem~\ref{thm:bicontact-doubling}.


\Addresses

\end{document}

%% file: bicontact_tikz_figures.tex
\providecolor{adaptblue}{RGB}{0,105,210}
\providecolor{adaptred}{RGB}{205,45,45}

\newcommand{\ProductAdmissibilityDiagram}{%
\begin{tikzpicture}[
    x=1cm,y=1cm,
    line cap=round,
    line join=round,
    >=Latex,
    direction/.style={line width=0.8pt},
    orderarrow/.style={-{Latex[length=2.2mm]},line width=0.75pt}
]
    \coordinate (O) at (0,0);

    \draw[direction] (-2.35, 1.36) -- ( 2.35,-1.36); 
    \draw[direction] (-1.45, 2.07) -- ( 1.45,-2.07); 
    \draw[direction] (-1.45,-2.07) -- ( 1.45, 2.07); 
    \draw[direction] (-2.35,-1.36) -- ( 2.35, 1.36); 

    \fill (O) circle (1.25pt);

    \draw[orderarrow] (-1.18,1.48) to[bend right=18] (-1.60,1.20);
    \draw[orderarrow] ( 1.18,1.48) to[bend left=18]  ( 1.60,1.20);

    \node[fill=white,inner sep=1.2pt] at (-1.32,2.22)
        {$S_-^{\mathrm{in}}$};
    \node[fill=white,inner sep=1.2pt,anchor=east] at (-2.22,1.02)
        {$S_-^{\mathrm{out}}$};
    \node[fill=white,inner sep=1.2pt] at (1.32,2.22)
        {$S_+^{\mathrm{in}}$};
    \node[fill=white,inner sep=1.2pt,anchor=west] at (2.22,1.02)
        {$S_+^{\mathrm{out}}$};
\end{tikzpicture}%
}

\newcommand{\StrongAdmissibilityDiagram}{%
\begin{tikzpicture}[
    x=1cm,y=1cm,
    line cap=round,
    line join=round,
    >=Latex,
    axis/.style={-{Latex[length=2.5mm]},line width=0.8pt},
    tangent/.style={-{Latex[length=2.4mm]},adaptblue,line width=1.15pt}
]
    \coordinate (A) at (0,3.05);
    \coordinate (B) at (4.15,0);

    \draw[axis] (-0.55,0) -- (5.35,0)
        node[below right=-1pt] {$\langle\alpha_{\mathrm{out}}\rangle$};
    \draw[axis] (0,-0.62) -- (0,4.05)
        node[above left=-1pt] {$\langle\alpha_{\mathrm{in}}\rangle$};

    \draw[densely dashed,gray!75,line width=0.8pt] (A) -- (B);

    \fill[adaptblue!10]
        (A) -- ++(0,-1.08) -- ++(0.82,0.47) -- cycle;
    \fill[adaptblue!10]
        (B) -- ++(0.98,0) -- ++(-0.10,-0.48) -- cycle;

    \draw[gray!75,line width=0.65pt] (A) -- ++(0,-1.08);
    \draw[gray!75,line width=0.65pt] (A) -- ++(0.82,-0.61);
    \draw[gray!75,line width=0.65pt] (B) -- ++(0.98,0);
    \draw[gray!75,line width=0.65pt] (B) -- ++(0.88,-0.48);

    \draw[tangent] (A) -- ++(0.36,-0.88)
        node[pos=0.72,right=5pt] {$\dot\lambda_+^{\mathrm{in}}$};
    \draw[tangent] (B) -- ++(0.82,-0.20)
        node[pos=0.48,above=5pt] {$\dot\lambda_+^{\mathrm{out}}$};

    \fill (A) circle (1.2pt);
    \fill (B) circle (1.2pt);
    \node[left=6pt]  at (A) {$\lambda_+^{\mathrm{in}}$};
    \node[below=4pt] at (B) {$\lambda_+^{\mathrm{out}}$};
\end{tikzpicture}%
}

\newcommand{\HyperbolicAdmissibilityDiagram}{%
\begin{tikzpicture}[
    x=1cm,y=1cm,
    scale=0.93,
    every node/.style={transform shape},
    line cap=round,
    line join=round,
    >=Latex,
    direction/.style={line width=0.75pt},
    stablelamination/.style={adaptblue,line width=3.2pt},
    unstablelamination/.style={adaptred,line width=3.2pt},
    orderarrow/.style={-{Latex[length=2.0mm]},line width=0.72pt}
]
    \begin{scope}
        \draw[direction] (-1.78, 1.02)--( 1.78,-1.02);
        \draw[direction] (-1.10, 1.58)--( 1.10,-1.58);
        \draw[direction] (-1.10,-1.58)--( 1.10, 1.58);
        \draw[direction] (-1.78,-1.02)--( 1.78, 1.02);
        \fill (0,0) circle (1.1pt);
        \draw[orderarrow] (-0.92,1.14) to[bend right=18] (-1.28,0.92);
        \draw[orderarrow] ( 0.92,1.14) to[bend left=18]  ( 1.28,0.92);
        \node[fill=white,inner sep=1pt] at (-1.12,1.76)
            {$\Gamma_*S_-^{\mathrm{in}}$};
        \node[fill=white,inner sep=1pt,anchor=east] at (-1.70,0.70)
            {$S_-^{\mathrm{out}}$};
        \node[fill=white,inner sep=1pt] at (1.12,1.76)
            {$\Gamma_*S_+^{\mathrm{in}}$};
        \node[fill=white,inner sep=1pt,anchor=west] at (1.70,0.70)
            {$S_+^{\mathrm{out}}$};
        \node[font=\small] at (0,-2.03) {Transit orbit};
    \end{scope}

    \begin{scope}[xshift=5.30cm]
        \draw[stablelamination] (-1.85,0)--(1.85,0);
        \draw[direction] (-1.08,1.55)--(1.08,-1.55);
        \draw[direction] (-1.08,-1.55)--(1.08,1.55);
        \fill (0,0) circle (1.1pt);
        \node[fill=white,inner sep=1pt] at (-1.06,1.75)
            {$S_-^{\mathrm{in}}$};
        \node[fill=white,inner sep=1pt] at (1.06,1.75)
            {$S_+^{\mathrm{in}}$};
        \node[below right=1pt,text=adaptblue] at (1.45,0)
            {$\mathcal L^s$};
        \node[font=\small] at (0,-2.03) {Near entrance lamination};
    \end{scope}

    \begin{scope}[xshift=10.10cm]
        \draw[unstablelamination] (0,-1.72)--(0,1.72);
        \draw[direction] (-1.16,1.55)--(1.16,-1.55);
        \draw[direction] (-1.16,-1.55)--(1.16,1.55);
        \fill (0,0) circle (1.1pt);
        \node[fill=white,inner sep=1pt] at (-1.12,1.75)
            {$S_-^{\mathrm{out}}$};
        \node[fill=white,inner sep=1pt] at (1.12,1.75)
            {$S_+^{\mathrm{out}}$};
        \node[below right=1pt,text=adaptred] at (0,-1.52)
            {$\mathcal L^u$};
        \node[font=\small] at (0,-2.03) {Near exit lamination};
    \end{scope}
\end{tikzpicture}%
}